\documentclass[eqthmnum,nocolour]{jt-calcs}
\usepackage{mathrsfs}
\usepackage{varwidth}
\usepackage{mdframed}
\usepackage{tikz}
\usetikzlibrary{matrix}

\usepackage[backend=bibtex,style=alphabetic,sorting=nyt]{biblatex}
\bibliography{bibliography}
\AtEveryBibitem{%
  \clearfield{pagetotal}%
}

\DeclareMathOperator{\Lie}{Lie}
\DeclareMathOperator{\der}{der}

\DeclareMathOperator{\spr}{spr}
\DeclareMathOperator{\nil}{nil}

\newcommand{\Centn}{\ensuremath{\mathrm{Cent}_{\mathrm{nil}}}}
\newcommand{\Centu}{\ensuremath{\mathrm{Cent}_{\mathrm{uni}}}}

\newcommand{\gr}{\ensuremath{\mathrm{gr}\,}}

\Crefname{enumi}{}{}

\setlist[enumerate,1]{label=(\roman*), leftmargin=1.3cm}
\setlist[enumerate,2]{leftmargin=1.3cm}

\title{Generalised Gelfand--Graev Representations in Small Characteristics}
\author{Jay Taylor}
\address{FB Mathematik, TU Kaiserslautern, Postfach 3049, 67653 Kaiserslautern, Germany}
\email{taylor@mathematik.uni-kl.de}
\mscno{2010}{20C33}{20G40}
\keywords{Finite reductive groups, character sheaves, generalised Gelfand--Graev representations, wave front sets}

\begin{document}
\begin{abstract}
Let $\bG$ be a connected reductive algebraic group over an algebraic closure $\overline{\mathbb{F}_p}$ of the finite field of prime order $p$ and let $F : \bG \to \bG$ be a Frobenius endomorphism with $G = \bG^F$ the corresponding $\mathbb{F}_q$-rational structure. One of the strongest links we have between the representation theory of $G$ and the geometry of the unipotent conjugacy classes of $\bG$ is a formula, due to Lusztig \cite{lusztig:1992:a-unipotent-support}, which decomposes Kawanaka's Generalised Gelfand--Graev Representations (GGGRs) in terms of characteristic functions of intersection cohomology complexes defined on the closure of a unipotent class. Unfortunately, the formula given in \cite{lusztig:1992:a-unipotent-support} is only valid under the assumption that $p$ is large enough. In this article we show that Lusztig's formula for GGGRs holds under the much milder assumption that $p$ is an acceptable prime for $\bG$ ($p$ very good is sufficient but not necessary). As an application we show that every irreducible character of $G$, resp., character sheaf of $\bG$, has a unique wave front set, resp., unipotent support, whenever $p$ is good for $\bG$.
\end{abstract}

\section{Introduction}
\begin{pa}
Let $\bG$ be a connected reductive algebraic group defined over an algebraic closure $\mathbb{K} = \overline{\mathbb{F}_p}$ of the finite field of prime order $p$ and let $F : \bG \to \bG$ be a Frobenius endomorphism defining an $\mathbb{F}_q$-rational structure $G = \bG^F$ on $\bG$. Assuming $p$ is a good prime for $\bG$ a theory of generalised Gelfand--Graev representations (GGGRs) was developed by Kawanaka in \cite{kawanaka:1986:GGGRs-exceptional} building on his investigations in \cite{kawanaka:1985:GGGRs-and-ennola-duality}. These are certain unipotently supported representations $\Gamma_u$ of $G$ which are defined for any unipotent element $u \in G$. Note that, identifying $\Gamma_u$ with its character, we have $\Gamma_u = \Gamma_v$ whenever $u, v \in G$ are $G$-conjugate so the GGGRs are naturally indexed by the unipotent conjugacy classes of $G$.
\end{pa}

\begin{pa}
Let $\rho \in \Irr(G)$ be an irreducible character and $\mathcal{O}$ an $F$-stable unipotent conjugacy class of $\bG$. We will denote by $\AV(\rho,\mathcal{O})$ the average value $\sum_{g \in \mathcal{O}^F} \rho(g)$ of $\rho$ on the rational points $\mathcal{O}^F$. We say $\mathcal{O}$ is a \emph{unipotent support} of $\rho$ if:
\begin{enumerate}[label=(US\arabic*), leftmargin=1.5cm]
	\item $\AV(\rho,\mathcal{O}) \neq 0$ and
	\item\label{US:2} $\AV(\rho,\widetilde{\mathcal{O}}) \neq 0$ implies $\dim \widetilde{\mathcal{O}} \leqslant \dim\mathcal{O}$ with $\widetilde{\mathcal{O}}$ any $F$-stable unipotent class.
\end{enumerate}
For any unipotent element $v \in \bG$ we denote by $\mathcal{O}_v$ the $\bG$-conjugacy class containing $v$. With this we say that $\mathcal{O}$ is a \emph{wave front set} of $\rho$ if:
\begin{enumerate}[label=(WF\arabic*), leftmargin=1.5cm]
	\item $\langle \Gamma_u, \rho \rangle \neq 0$ for some $u \in \mathcal{O}$ and
	\item\label{WF:2} $\langle \Gamma_v, \rho \rangle \neq 0$ implies $\dim \mathcal{O}_v \leqslant \dim\mathcal{O}$ with $v \in G$ any unipotent element.
\end{enumerate}
When $u$ is the identity element we have $\Gamma_u$ is the regular representation of $G$ so $\langle \Gamma_1,\rho\rangle = \rho(1) = \AV(\rho,\{1\})$; hence every irreducible character of $G$ admits a unipotent support and a wave-front set. However, it was conjectured by Lusztig \cite{lusztig:1980:some-problems-in-the-representation-theory}, resp., Kawanaka \cite{kawanaka:1985:GGGRs-and-ennola-duality}, that each irreducible character $\rho \in \Irr(G)$ admits a unique unipotent support $\mathcal{O}_{\rho}$, resp., wave front set $\mathcal{O}_{\rho}^*$. If this conjecture is satisfied then we say the unipotent support/wave front set of $\rho$ is well defined.
\end{pa}

\begin{pa}
Assuming $p$ and $q$ are sufficiently large then Lusztig has shown in \cite{lusztig:1992:a-unipotent-support} that the unipotent support and wave front set of an irreducible character are well defined. He also gave a definition for the unipotent support of a character sheaf and similarly showed that each character sheaf has a well-defined unipotent support, see \cref{def:unipotent-support-char-sheaf} for the definition. These results provide one of the most profound relationships between irreducible characters of $G$ and the geometry of the algebraic group $\bG$. They also highlight the central role that character sheaves play in the representation theory of finite reductive groups.
\end{pa}

\begin{pa}
Using Lusztig's results Geck was able to show that each irreducible character has a unique unipotent support whenever $p$ is a good prime for $\bG$, see \cite{geck:1996:on-the-average-values}. In turn, using Geck's result together with ideas developed in \cite{lusztig:1986:on-the-character-values} Aubert was able to prove, in certain special cases, that character sheaves admit a unique unipotent support whenever $p$ is good, see \cite{aubert:2003:characters-sheaves-and-generalized}. The following completes this picture in good characteristic, see \cref{thm:unip-support-char-sheaves,thm:existence-wave-front-sets}, thus proving Kawanaka's conjecture in general.
\end{pa}

\begin{thm*}
Assume $p$ is a good prime for $\bG$ then any irreducible character of $G$ has a unique wave front set and any character sheaf of $\bG$ has a unique unipotent support.
\end{thm*}

\begin{pa}\label{pa:geometric-condition}
Thanks to results of Achar and Aubert \cite{achar-aubert:2007:supports-unipotents-de-faisceaux} we may even give a geometric refinement of the conditions \cref{US:2,WF:2}. Namely, for any $F$-stable unipotent class $\widetilde{\mathcal{O}}$ of $\bG$ or unipotent element $v \in G$ we have
\begin{enumerate}[leftmargin=1.5cm]
	\item[(US2')] $\AV(\rho,\widetilde{\mathcal{O}}) \neq 0$ implies $\widetilde{\mathcal{O}} \subseteq \overline{\mathcal{O}_{\rho}}$,
	\item[(WF2')] $\langle \Gamma_v,\rho\rangle \neq 0$ implies $\mathcal{O}_v \subseteq \overline{\mathcal{O}_{\rho}^*}$,
\end{enumerate}
see \cite[Th\'eor\`eme 6.3, Th\'eor\`eme 9.1]{achar-aubert:2007:supports-unipotents-de-faisceaux} and \cref{prop:geometric-refinement}. Here $\overline{\mathcal{O}_{\rho}}$ denotes the Zariski closure of $\mathcal{O}_{\rho}$ and similarly for $\mathcal{O}_{\rho}^*$. Now, if $p$ is a good prime for $\bG$ then we have two well-defined maps
\begin{equation*}
\Irr(G) \to \{F\text{-stable unipotent conjugacy classes of }\bG\}
\end{equation*}
given by $\rho \mapsto \mathcal{O}_{\rho}$ and $\rho \mapsto \mathcal{O}_{\rho}^*$. These turn out to be dual in the following sense. Let $\rho^* \in \Irr(G)$ be the unique irreducible character such that $\rho^* = \pm D_G(\rho)$ where $D_G(\rho)$ is the Alvis--Curtis dual of $\rho$ then $\mathcal{O}_{\rho^*} = \mathcal{O}_{\rho}^*$, see \cref{lem:duality-unip-supp-wave-front}. In other words, the unipotent support of the Alvis--Curtis dual of $\rho$ is the wave front set of $\rho$.
\end{pa}

\begin{pa}
We note here that in \cite{lusztig:2015:restriction-of-a-character-sheaf-to-conjugacy-classes} Lusztig has obtained a refinement of the notion of unipotent support for a character sheaf in good characteristic. There it is stated that the uniqueness of unipotent supports in good characteristic may be deduced from the case of large characteristic by standard methods. In \cref{prop:unipotent-supports-SLn} we give an example of such methods in the special case where $\bG$ is $\SL_n(\mathbb{K})$.
\end{pa}

\begin{pa}
We now give an overview of the arguments used in this paper. In \cite{lusztig:1992:a-unipotent-support} Lusztig gave a formula relating GGGRs and IC complexes on unipotent classes, which we will refer to as \emph{Lusztig's formula}. Reading carefully \cite{lusztig:1992:a-unipotent-support} one sees that Lusztig's proof that the wave front set of an irreducible character is well defined ultimately relies on the validity of Lusztig's formula and the validity of the results in \cite{lusztig:1990:green-functions-and-character-sheaves}. If $p$ is good for $\bG$ then the results of \cite{lusztig:1990:green-functions-and-character-sheaves} are true if $q$ is sufficiently large. However, if $\bG$ has a connected centre then Shoji has shown that this restriction on $q$ may be dropped, see \cite{shoji:1996:on-the-computation}. In other words, if $p$ is a good prime for $\bG$ and $Z(\bG)$ is connected then the results of \cite{lusztig:1990:green-functions-and-character-sheaves} are true.
\end{pa}

\begin{pa}
One of the main results in this paper shows that Lusztig's formula remains valid whenever $p$ is an \emph{acceptable prime} for $\bG$, see \cref{def:acceptable-prime,thm:main-theorem-unipotent}. Combining this with Shoji's result from \cite{shoji:1996:on-the-computation} one sees that Lusztig's proof in \cite{lusztig:1992:a-unipotent-support} showing that the wave front set is well defined remains valid if $p$ is an acceptable prime for $\bG$ and $Z(\bG)$ is connected. Now, the notion of acceptable prime lies in between the notions of very good prime and good prime. More precisely, we have a series of implications
\begin{equation*}
p\text{ is very good for }\bG \Rightarrow p\text{ is acceptable for }\bG \Rightarrow p\text{ is good for }\bG.
\end{equation*}
If $\bG$ is $\GL_n(\mathbb{K})$ or $Z(\bG)$ is connected and $\bG/Z(\bG)$ is simple not of type $\A$ then a prime $p$ is acceptable for $\bG$ if and only if it is good for $\bG$. Thus, for these groups we have the wave front set is well defined if $p$ is good for $\bG$ by \cref{thm:main-theorem-unipotent}, \cite{shoji:1996:on-the-computation} and \cite{lusztig:1992:a-unipotent-support}. Using a series of standard reduction arguments we may then deduce that the wave front set is well defined for any connected reductive algebraic group $\bG$ assuming only that $p$ is good for $\bG$. Similar reduction arguments were also used by Geck in \cite{geck:1996:on-the-average-values} to show that the unipotent support is well defined. Our approach for dealing with unipotent supports of character sheaves is entirely similar, however in this case things are somewhat simpler as one may ignore the $\mathbb{F}_q$-structure. These reduction arguments are carried out in \S12 - \S15.
\end{pa}

\begin{pa}
Recall that if $p$ is sufficiently large then one may define inverse isomorphisms $\exp : \mathcal{N} \to \mathcal{U}$ and $\log : \mathcal{U} \to \mathcal{N}$ between the unipotent variety of $\bG$ and the nilpotent cone of its Lie algebra $\lie{g}$. In \cite{lusztig:1992:a-unipotent-support} Lusztig uses the $\exp$ and $\log$ maps to define the GGGRs and to transfer their study to that of $\bG$-invariant functions supported on $\mathcal{N}$. The upshot of this is that one acquires a powerful new tool, namely the Fourier transform. Lusztig then uses results from \cite{lusztig:1987:fourier-transforms-on-a-semisimple-lie-algebra} on the Fourier transform to deduce the formula.
\end{pa}

\begin{pa}
The results in \cite{lusztig:1987:fourier-transforms-on-a-semisimple-lie-algebra} are proved under the assumption that $p$ is sufficiently large. However, in \cite{letellier:2005:fourier-transforms}, Letellier has shown that the main results of \cite{lusztig:1987:fourier-transforms-on-a-semisimple-lie-algebra} still hold if $p$ is an acceptable prime for $\bG$. This is where our assumption on the characteristic comes from as we prolifically use the results of Letellier throughout this article. With this in hand, our strategy for proving Lusztig's formula is to show that Lusztig's argument still applies if one replaces the $\exp$ and $\log$ maps with $\phi_{\spr}^{-1}$ and $\phi_{\spr}$, where $\phi_{\spr} : \mathcal{U} \to \mathcal{N}$ is a suitably chosen Springer isomorphism.
\end{pa}

\begin{pa}
For this strategy to work we must address several technical details, which are dealt with in \S2 - \S5. For instance, to ensure that a Springer isomorphism exists one needs to make some assumption on the group $\bG$. In particular, if we assume that a simply connected covering of the derived subgroup of $\bG$ is a separable morphism then a Springer isomorphism will exist. We call a group satisfying this condition \emph{proximate}. In \S4 we show that if $\bG$ is proximate then one can find a special type of Springer isomorphism $\phi_{\spr}$ satisfying properties which mean that $\phi_{\spr}$ can be used in the definition of the GGGRs of $G$. This construction is then given in \S5 following work of Kawanaka. Note this is very important to our cause. If one constructs the GGGRs as in \cite{kawanaka:1986:GGGRs-exceptional} without using a Spinger isomorphism then one runs into serious problems when trying to compute the Fourier transform of the corresponding GGGR on the Lie algebra.
\end{pa}

\begin{pa}
As well as finding a good Springer isomorphism we will also need to know various facts concerning centralisers of nilpotent elements. These are used in proving statements that are needed for the definition of GGGRs. Under the assumption that $\bG$ is proximate we derive these results in \S3 from the work of Premet \cite{premet:2003:nilpotent-orbits-in-good-characteristic}. Note that these results can fail if $\bG$ is not proximate. In light of this, we show in \S2 that for any connected reductive algebraic group $\bG$ and Frobenius endomorphism $F : \bG \to \bG$ there exists a proximate algebraic group $\overline{\bG}$, a bijective morphism of algebraic groups $\phi : \overline{\bG} \to \bG$ and a Frobenius endomorphism $\overline{F} : \overline{\bG} \to \overline{\bG}$ such that $\phi\circ\overline{F} = F \circ \phi$. In particular, we have $\phi$ restricts to an isomorphism $\overline{\bG}^{\overline{F}} \to \bG^F$. This shows that, for the purposes of defining the GGGRs of $\bG^F$, we can assume that $\bG$ is proximate so that the desired results are available to us.
\end{pa}

\begin{pa}
With the groundwork on GGGRs in place we then proceed to prove Lusztig's formula in sections \S6 - \S11. Here we follow \cite{lusztig:1992:a-unipotent-support} to the letter, simply finding alternative arguments when either the $\exp$ or $\log$ maps were used or when the theory of $\lie{sl}_2$-triples was used. We have tried not to unnecessarily repeat arguments from \cite{lusztig:1992:a-unipotent-support} but some things are repeated to improve the quality of the exposition. Having said this, we have chosen to give most of the arguments from \cite[\S6]{lusztig:1992:a-unipotent-support} as this was originally proved under the assumption that $F$ is split but later remarked that this assumption is unnecessary \cite[8.7]{lusztig:1992:a-unipotent-support}.
\end{pa}

\begin{acknowledgments}
The author would like to thank Pramod Achar, Olivier Dudas, Simon Goodwin, Sebastian Herpel, Gunter Malle, Iulian Simion and Britta Sp\"ath for useful conversations/correspondence. We especially thank Simon Goodwin for pointing us towards \cite{gan-ginzburg:2002:quantization-of-slodowy-slices} and \cite{bardsley-richardson:1985:Etale-slices-for-algebraic-transformation-groups}, which formed important ingredients in this paper, and George McNinch both for useful discussions on Springer isomorphisms and for allowing us to include his proof of \cref{lem:sep-isog-iso-uni-vars}. Finally the author gratefully acknowledges the financial support of ERC Advanced Grant 291512 awarded to Prof.\ Gunter Malle.
\end{acknowledgments}

\section{Proximate Algebraic Groups}
\begin{assumption}
We will assume that $\bG$ is a connected reductive algebraic group defined over $\mathbb{K}$, an algebraic closure of a finite field $\mathbb{F}_p$ of prime order $p$, and that $F : \bG \to \bG$ is a Frobenius endomorphism defining an $\mathbb{F}_q$-rational structure $G = \bG^F$.
\end{assumption}

\begin{pa}
Recall that a root datum $\mathcal{R} = (X,\Phi,Y,\widecheck{\Phi})$ is a quadruple such that $X$ and $Y$ are free $\mathbb{Z}$-modules of finite rank equipped with a perfect pairing $\langle-,-\rangle_{\mathcal{R}} : X \times Y \to \mathbb{Z}$ and $\Phi \subseteq X$ and $\widecheck{\Phi}\subseteq Y$ are finite subsets equipped with a bijection $\alpha \mapsto \widecheck{\alpha}$ satisfying the conditions in \cite[7.4.1, 7.4.4]{springer:2009:linear-algebraic-groups}. In particular, all our root data are assumed to be reduced. Now assume $\mathcal{R}' = (X',\Phi',Y',\widecheck{\Phi}')$ is another root datum and $\varphi : X' \to X$ is a $\mathbb{Z}$-module homomorphism. We will denote by $\widecheck{\varphi} : Y \to Y'$ the transpose of $\varphi$ which is the unique $\mathbb{Z}$-module homomorphism satisfying
\begin{equation*}
\langle \varphi(x), y\rangle_{\mathcal{R}} = \langle x,\widecheck{\varphi}(y) \rangle_{\mathcal{R}'}
\end{equation*}
for all $x \in X'$ and $y \in Y$. We say $\varphi$ is a homomorphism of root data if $\varphi$ restricts to a bijection $\Phi' \to \Phi$ and $\widecheck{\varphi}(\widecheck{\alpha}) = \widecheck{\beta}$ whenever $\varphi(\beta) = \alpha$. A homomorphism of root data is an isomorphism if $\varphi$ is a $\mathbb{Z}$-module isomorphism. Following \cite{steinberg:1999:the-isomorphism-and-isogeny-theorems} we say that $\varphi$ is an isogeny of root data if the following hold:
\begin{enumerate}
	\item $\varphi : X' \to X$ and $\widecheck{\varphi} : Y \to Y'$ are injective,
	\item there exists a bijection $b : \Phi \to \Phi'$ and a map $q : \Phi \to \{p^a \mid a \in \mathbb{Z}_{\geqslant 0}\}$ such that for any $\alpha \in \Phi$ we have $\varphi(\beta) = q(\alpha)\alpha$ and $\widecheck{\varphi}(\widecheck{\alpha}) = q(\alpha)\widecheck{\beta}$ where $\beta = b(\alpha)$.
\end{enumerate}
\end{pa}

\begin{pa}\label{pa:pairs}
If $\bT \leqslant \bG$ is a maximal torus then one may construct a root datum $\mathcal{R}(\bG,\bT) = (X(\bT),\Phi,Y(\bT),\widecheck{\Phi})$ as in \cite[7.4.3]{springer:2009:linear-algebraic-groups}, where $X(\bT)$, resp., $Y(\bT)$, is the character, resp., cocharacter, group of $\bT$. Recall that an isogeny of algebraic groups is defined to be a surjective homomorphism of algebraic groups whose kernel is finite. If $(\bH,\bS)$ is another pair consisting of a connected reductive algebraic group $\bH$ over $\mathbb{K}$ and a maximal torus $\bS$ then we say $\varphi : (\bG,\bT) \to (\bH,\bS)$ is an isogeny, resp., isomorphism, if $\varphi : \bG \to \bH$ is an isogeny, resp., isomorphism, of algebraic groups and $\varphi(\bT) = \bS$. If $\varphi$ is such an isogeny, resp., isomorphism, then the map $\varphi^* : X(\bS) \to X(\bT)$ given by $\varphi^*(\chi) = \chi\circ\varphi$ is an isogeny, resp., isomorphism, of root data. Now, for any $g \in \bG$ we denote by $\Inn g : \bG \to \bG$ the inner automorphism defined by $\Inn g(x) = gxg^{-1}$ for all $x \in \bG$. With this we have the following classical result, see \cite[1.5]{steinberg:1999:the-isomorphism-and-isogeny-theorems} and \cite[3.17]{digne-michel:1991:representations-of-finite-groups-of-lie-type}.
\end{pa}

\begin{thm}[Isogeny Theorem]\label{thm:isogeny-theorem}
Let $(\bG,\bT)$ and $(\bH,\bS)$ be two pairs as in \cref{pa:pairs} then every isogeny, resp., isomorphism, of root data $\mathcal{R}(\bH,\bS) \to \mathcal{R}(\bG,\bT)$ is of the form $\varphi^*$ for some isogeny, resp., isomorphism, $\varphi : (\bG,\bT) \to (\bH,\bS)$. Moreover, if $\varphi, \psi : (\bG,\bT) \to (\bH,\bS)$ are two isogenies such that $\varphi^* = \psi^*$ then $\psi = \varphi\circ\Inn t$ for some $t \in \bT$. Furthermore an isogeny $\varphi : (\bG,\bT) \to (\bG,\bT)$ is a Frobenius endomorphism defining an $\mathbb{F}_q$-rational structure on $\bG$ if and only if $\varphi^* = q\tau$ with $\tau : X(\bT) \to X(\bT)$ a finite order automorphism.
\end{thm}

\begin{rem}\label{rem:Frob-end}
Note that $\varphi^*$ is of the form $q\tau$ with $\tau : X(\bT) \to X(\bT)$ a finite order automorphism if and only if its transpose is of the form $q\widecheck{\tau}$ with $\widecheck{\tau} : Y(\bT) \to Y(\bT)$ a finite order automorphism.
\end{rem}

\begin{pa}\label{pa:induced-root-datum}
For any $\mathbb{Z}$-module $W$ we will denote by $W^* = \Hom(W,\mathbb{Z})$ the dual module. Now assume $\mathcal{R} = (X,\Phi,Y,\widecheck{\Phi})$ is a root datum and $B \subseteq Y$ is a submodule such that $\widecheck{\Phi} \subseteq B$. We denote by $\langle -,-\rangle_{\mathcal{R}^B} : B \times B^* \to \mathbb{Z}$ the natural perfect pairing given by $\langle x,y\rangle_{\mathcal{R}^B} = y(x)$. By the definition of perfect pairing the homomorphisms
\begin{align*}
\Theta_X : X &\to Y^* & \Theta_Y : Y &\to X^*\\
x &\mapsto \langle x,-\rangle_{\mathcal{R}} & y &\mapsto \langle -,y\rangle_{\mathcal{R}}
\end{align*}
are isomorphisms. Now if $\iota : B \to Y$ is the natural inclusion then it is easily checked that $\widecheck{\iota} = \iota^*\circ\Theta_X : X \to B^*$ is the transpose of $\iota$, where $\iota^* : X^* \to B^*$ is the map defined by $\iota^*(f) = f\circ\iota$. By \cite[XXI, 1.1.4]{sga:2011:sga3-TomeIII} we have $\widecheck{\iota} : X \to B^*$ restricts to a bijection $\Phi \to \widecheck{\iota}(\Phi)$ and by definition
\begin{equation*}
\langle \widecheck{\iota}(\alpha), \widecheck{\alpha} \rangle_{\mathcal{R}_A} = \langle \alpha,\iota(\widecheck{\alpha})\rangle_{\mathcal{R}} = 2
\end{equation*}
for any $\alpha \in \Phi$. From this it is easy to see that $\mathcal{R}^B = (B^*,\widecheck{\iota}(\Phi),B,\widecheck{\Phi})$ is a root datum such that $\iota(\widecheck{\alpha}) = \widecheck{\beta}$ if $\beta = \widecheck{\iota}(\alpha)$ for any $\alpha \in \Phi$ hence $\widecheck{\iota} : A \to X$ is a homomorphism of root data. Following \cite[XXI, 6.5]{sga:2011:sga3-TomeIII} we call $\mathcal{R}^B$ the \emph{root datum coinduced by $B$}.
\end{pa}

\begin{pa}\label{pa:perp-top}
Now for any \emph{subset} $B \subseteq Y$ we define submodules
\begin{align*}
B^{\perp} &= \{x \in X \mid \langle x,y\rangle_{\mathcal{R}} = 0\text{ for all }y \in B\} \subseteq Y,\\
B^{\top} &= \{y \in Y \mid ny \in \mathbb{Z}B\text{ for some non-zero } n \in \mathbb{Z}\} \subseteq Y,
\end{align*}
where $\mathbb{Z}B$ is the submodule generated by $B$. For any subset $A \subseteq X$ we have the submodules $A^{\perp} \subseteq Y$ and $A^{\top} \subseteq X$ are defined analogously. Note that $B^{\top}/\mathbb{Z}B$ is nothing other than the torsion submodule of $Y/\mathbb{Z}B$ and so the module $Y/B^{\top}$ is a free module.
\end{pa}

\begin{pa}\label{pa:derived-subgroup-RD}
By \cref{pa:perp-top} we have the quotient $Y/\widecheck{\Phi}^{\top}$ is a free module so the natural short exact sequence
\begin{equation*}
\begin{tikzpicture}[baseline={([yshift=-1ex]current bounding box.center)}]
\matrix (m) [matrix of math nodes, row sep=4em, column sep=2em, text height=1.5ex, text depth=0.25ex]
{ 0 & B & Y & Y/B & 0\\};

\path [>=stealth,->]	(m-1-1) edge (m-1-2)
		(m-1-2) edge node [above] {$\iota$} (m-1-3)
		(m-1-3) edge (m-1-4)
		(m-1-4) edge (m-1-5);
\end{tikzpicture}
\end{equation*}
splits. In particular, there exists a submodule $N \subseteq Y$ such that $Y = \widecheck{\Phi}^{\top} \oplus N$ which implies that $Y^* = (\widecheck{\Phi}^{\top})^* \oplus N^*$. Now $\iota^* : Y^* \to (\widecheck{\Phi}^{\top})^*$ is the natural projection homomorphism so the transpose $\widecheck{\iota} : X \to (\widecheck{\Phi}^{\top})^*$ is surjective with kernel $\widecheck{\Phi}^{\perp}$, c.f., \cref{pa:induced-root-datum}. Let $\mathcal{R}_{\der} = (X_{\der},\Phi,Y_{\der},\widecheck{\Phi})$ be the root datum with $X_{\der} = X/\widecheck{\Phi}^{\perp}$, $Y_{\der} = \widecheck{\Phi}^{\top}$ and the perfect pairing $\langle -,-\rangle_{\mathcal{R}_{\der}} : X_{\der} \times Y_{\der} \to \mathbb{Z}$ given by
\begin{equation*}
\langle x+\widecheck{\Phi}^{\perp},y\rangle_{\mathcal{R}_{\der}} := \langle x,y\rangle_{\mathcal{R}}
\end{equation*}
for all $x+\widecheck{\Phi}^{\top} \in X_{\der}$ and $y \in Y_{\der}$. The map $X/\widecheck{\Phi}^{\perp} \to B^*$ given by $x + \widecheck{\Phi}^{\perp} \mapsto \widecheck{\iota}(x)$ is then an isomorphism of root data $\mathcal{R}_{\der} \to \mathcal{R}^{\widecheck{\Phi}^{\top}}$.
\end{pa}

\begin{lem}[{}{\cite[8.1.9]{springer:2009:linear-algebraic-groups}}]
If $\bG$ is a connected reductive algebraic group with root datum $\mathcal{R}(\bG,\bT)$ then the root datum $\mathcal{R}(\bG_{\der},\bT\cap\bG_{\der})$ of the derived subgroup $\bG_{\der}$ of $\bG$ is isomorphic to $\mathcal{R}(\bG,\bT)_{\der} \cong \mathcal{R}(\bG,\bT)^{\widecheck{\Phi}^{\top}}$.
\end{lem}

\begin{pa}\label{pa:simply-connected-RD}
Let $\mathbb{Q}X_{\der}$, resp., $\mathbb{Q}Y_{\der}$, be the $\mathbb{Q}$-vector space $\mathbb{Q} \otimes_{\mathbb{Z}} X_{\der}$, resp., $\mathbb{Q} \otimes_{\mathbb{Z}} Y_{\der}$. The $\mathbb{Z}$-submodules $\mathbb{Z}\Phi \subseteq X_{\der}$ may naturally be considered as submodules of the weight lattice
\begin{equation*}
\Lambda = \{\omega \in \mathbb{Q}\Phi \mid \langle \omega, \widecheck{\alpha}\rangle_{\mathcal{R}} \in \mathbb{Z}\text{ for all }\widecheck{\alpha} \in \widecheck{\Phi}\} \subseteq \mathbb{Q}X_{\der},
\end{equation*}
where the form $\langle-,-\rangle_{\mathcal{R}}$ is extended naturally to the $\mathbb{Q}$-vector spaces. We may then form the root datum $\mathcal{R}_{\simc} = (\Lambda,\Phi,\mathbb{Z}\widecheck{\Phi},\widecheck{\Phi})$ with the perfect pairing $\langle-,-\rangle_{\mathcal{R}_{\simc}}$ being the natural one inherited from the extension of $\langle-,-\rangle_{\mathcal{R}}$ to coefficients over $\mathbb{Q}$. The root datum $\mathcal{R}_{\simc}$ is the root datum of a connected reductive algebraic group $\bG_{\simc}$, which is the simply connected group of the same type as $\bG_{\der}$. Now we have a natural injective homomorphism $f_{\simc} : X_{\der} \to \Lambda$ given by inclusion whose dual is simply the natural inclusion $\widecheck{f}_{\simc} : \mathbb{Z}\widecheck{\Phi} \to Y_{\der}$. In particular, one easily checks that $f_{\simc}$ is an isogeny of root data which lifts to an isogeny $\phi_{\simc} : \bG_{\simc} \to \bG_{\der}$. We call $\phi_{\simc}$ a \emph{simply connected covering} of $\bG_{\der}$.
\end{pa}

\begin{assumption}
From now on we assume that $\phi_{\simc} : \bG_{\simc} \to \bG_{\der}$ is a fixed simply connected covering of the derived subgroup of $\bG$.
\end{assumption}

\begin{definition}
We say an algebraic group $\bG$ is \emph{proximate} if $\bG$ is connected, reductive and the isogeny $\phi_{\simc}$ is separable.
\end{definition}

\begin{rem}
The isogeny $\phi_{\simc}$ is not uniquely defined but as any two simply connected coverings differ by an inner automorphism of $\bG$ we see that $\phi_{\simc}$ is separable if and only if every simply connected covering is separable.
\end{rem}

\begin{pa}
For any algebraic group $\bH$ we will denote by $\Lie(\bH)$ the corresponding Lie algebra, i.e., the tangent space $T_1(\bH)$ at the identity. Moreover, for notational convenience, we define $\lie{g} = \Lie(\bG)$, $\lie{g}_{\der} = \Lie(\bG_{\der})$ and $\lie{g}_{\simc} = \Lie(\bG_{\simc})$. The following easy result shows the utility of a proximate algebraic group,  c.f., \cite[4.3.7(iii)]{springer:2009:linear-algebraic-groups}.
\end{pa}

\begin{prop}\label{prop:lie-algebra-iso}
If $\phi : \bH \to \bH'$ is an isogeny between two affine algebraic groups then $\phi$ is separable if and only if the differential $\mathrm{d}\phi : \Lie(\bH) \to \Lie(\bH')$ is an isomorphism.
\end{prop}


\begin{pa}
We now wish to show that we can replace our algebraic group $\bG$ by a proximate algebraic group without affecting the finite reductive group $\bG^F$ (up to isomorphism). This will be important later as we will need to assume that $\bG$ is proximate to define Kawanaka's GGGRs. Before doing this we will need the following characterisation of proximate algebraic groups.
\end{pa}

\begin{lem}\label{lem:p-tor-springer-good}
If $\bG$ has root datum $\mathcal{R}(\bG,\bT) = (X,\Phi,Y,\widecheck{\Phi})$ with respect to some maximal torus $\bT\leqslant\bG$ then $\bG$ is proximate if and only if $Y/\mathbb{Z}\widecheck{\Phi}$ has no $p$-torsion.
\end{lem}

\begin{proof}
Let the notation be as in \cref{pa:simply-connected-RD}. The groups $\Lambda/X_{\der}$ and $Y_{\der}/\mathbb{Z}\widecheck{\Phi}$ are finite groups in duality, c.f., \cite[XXI, 6.2.3]{sga:2011:sga3-TomeIII}, so $Y_{\der}/\mathbb{Z}\widecheck{\Phi}$ has no $p$-torsion if and only if $\Lambda/X_{\der}$ has no $p$-torsion. However, this latter condition is precisely equivalent to $\phi_{\simc}$ being separable, see \cite[Proposition 2.4.4]{letellier:2005:fourier-transforms}. It now suffices to observe that the torsion submodule of $Y/\mathbb{Z}\widecheck{\Phi}$ is naturally isomorphic to the quotient $Y_{\der}/\mathbb{Z}\widecheck{\Phi}$ because $Y_{\der}$ has a complement, c.f., \cref{pa:derived-subgroup-RD}.
\end{proof}

\begin{prop}\label{prop:pretty-good-exists}
There exists a proximate algebraic group $\overline{\bG}$, defined over $\mathbb{K}$, and a bijective morphism of algebraic groups $\phi : \overline{\bG} \to \bG$. Furthermore, there exists a Frobenius endomorphism $\overline{F} : \overline{\bG} \to \overline{\bG}$ such that $\phi\circ\overline{F} = F\circ\phi$ so that $\phi$ restricts to an isomorphism $\overline{\bG}^{\overline{F}} \to G$.
\end{prop}

\begin{proof}
Let $\mathcal{R}(\bG,\bT) = (X,\Phi,Y,\widecheck{\Phi})$ be the root datum of $\bG$ with respect to a maximal torus $\bT \leqslant \bG$. A submodule $V \subseteq Y$ is said to have property ($\star$) if the following holds:
\begin{enumerate}
	\item[($\star$)] $\widecheck{\Phi} \subseteq V$ and $Y/V$ is a $p$-group.
\end{enumerate}
Note, in particular, that $Y/V$ is necessarily finite. Now assume $W \subseteq Y$ also has property ($\star$) then we claim that $V\cap W$ has property ($\star$) as well. Clearly we need only show that $Y/V\cap W$ is a $p$-group. By the second isomorphism theorem we have $V/V\cap W \cong V+W/V$ which is a $p$-group because $V+W/V \leqslant Y/V$. Moreover by the first isomorphism theorem we have $|Y/V\cap W| = |Y/V|\cdot|V/V\cap W|$ which shows that $Y/V\cap W$ is a $p$-group, as desired.

With this we see that there is a unique minimal submodule $\overline{Y} \subseteq Y$ having property ($\star$), which is simply the intersection of all the submodules having property ($\star$). If $\iota : \overline{Y} \to Y$ is the natural inclusion then we may construct the root datum $\mathcal{R}^{\overline{Y}} = (\overline{Y}^*,\widecheck{\iota}(\Phi),\overline{Y},\widecheck{\Phi})$ coinduced by $\overline{Y}$, c.f., \cref{pa:induced-root-datum}. Now consider the natural short exact sequence
\begin{equation*}
\begin{tikzpicture}[baseline={([yshift=-1ex]current bounding box.center)}]
\matrix (m) [matrix of math nodes, row sep=4em, column sep=2em, text height=1.5ex, text depth=0.25ex]
{ 0 & \overline{Y} & Y & Y/\overline{Y} & 0\\};

\path [>=stealth,->]	(m-1-1) edge (m-1-2)
		(m-1-2) edge node [above] {$\iota$} (m-1-3)
		(m-1-3) edge (m-1-4)
		(m-1-4) edge (m-1-5);
\end{tikzpicture}
\end{equation*}
By assumption $Y/\overline{Y}$ is finite so applying $\Hom(-,\mathbb{Z})$ to this sequence we obtain an exact sequence
\begin{equation*}
\begin{tikzpicture}[baseline={([yshift=-1ex]current bounding box.center)}]
\matrix (m) [matrix of math nodes, row sep=4em, column sep=2em, text height=1.5ex, text depth=0.25ex]
{ 0 & 0 & Y^* & \overline{Y}^*\\};

\path [>=stealth,->]	(m-1-1) edge (m-1-2)
		(m-1-2) edge (m-1-3)
		(m-1-3) edge node [above] {$\iota^*$} (m-1-4);
\end{tikzpicture}
\end{equation*}
In particular, we have $\iota^* : Y^* \to \overline{Y}^*$ is injective hence so is the dual $\widecheck{\iota} : X \to \overline{Y}^*$. It is then easily checked that $\widecheck{\iota} : X \to \overline{Y}^*$ is an isogeny of root data.

We will denote by $(\overline{\bG},\overline{\bT})$ a connected reductive algebraic group $\overline{\bG}$ and maximal torus $\overline{\bT}\leqslant\overline{\bG}$ such that $\mathcal{R}(\overline{\bG},\overline{\bT}) = \mathcal{R}^{\overline{Y}}$. By \cref{thm:isogeny-theorem} there exists an isogeny $\phi : \overline{\bG} \to \bG$ of algebraic groups such that $\widecheck{\iota} = \phi^*$. In particular, $\phi$ is surjective and has finite kernel. By design the quotient $Y/\overline{Y}$ is a $p$-group so according to \cite[1.11]{bonnafe:2006:sln} we have $\Ker(\phi) = \{1\}$, which shows that $\phi$ is a bijective morphism of algebraic groups. Clearly $\overline{Y}/\mathbb{Z}\widecheck{\Phi}$ has no $p$-torsion because $\overline{Y}$ is the unique minimal submodule having property ($\star$). Hence $\overline{\bG}$ is proximate by \cref{lem:p-tor-springer-good}.

Let us now consider the statement concerning the Frobenius endomorphism. We may assume that our chosen maximal torus $\bT$ is $F$-stable so that $F$ induces an isogeny of root data $F^* : X \to X$. By \cref{thm:isogeny-theorem} $F^*$ is necessarily of the form $q\tau$ with $\tau : X \to X$ a finite order automorphism. As the dual $\widecheck{\tau} : Y \to Y$ is an automorphism stabilising $\widecheck{\Phi}$ we see that $\widecheck{\tau}(\overline{Y}) = \overline{Y}$ from the definition of $\overline{Y}$. In particular $\widecheck{F}^*$ restricts to an injective homomorphism $\widecheck{\sigma} : \overline{Y} \to \overline{Y}$ which satisfies $\iota \circ \widecheck{\sigma} = \widecheck{F}^*\circ \iota$ by definition. Consequently we have $\sigma\circ\widecheck{\iota} = \widecheck{\iota}\circ F^*$ where $\sigma : \overline{Y}^* \to \overline{Y}^*$ is the dual of $\widecheck{\sigma}$. From this it follows easily that $\sigma$ is an isogeny of root data because $F^*$ is.

As $\widecheck{\sigma} = q\widecheck{\tau}|_{\overline{Y}}$ we have by \cref{thm:isogeny-theorem,rem:Frob-end} that there exists a Frobenius endomorphism $\overline{F}' : \overline{\bG} \to \overline{\bG}$ stabilising $\overline{\bT}$ such that $\overline{F}'^* = \sigma$. Now, as $\sigma\circ\widecheck{\iota} = \widecheck{\iota}\circ F^*$ we have $(\phi\circ \overline{F}')^* = (F \circ\phi)^*$ so there exists an element $t \in \overline{\bT}$ such that $\phi\circ \overline{F}'\circ\Inn t = F \circ \phi$ by \cref{thm:isogeny-theorem}. Setting $\overline{F} = \overline{F}'\circ \Inn t$ we thus have $\phi\circ \overline{F} = F \circ \phi$ as desired. We now need only show that $\overline{F}$ is a Frobenius endomorphism. By \cite[3.6(i)]{digne-michel:1991:representations-of-finite-groups-of-lie-type} it suffices to show that $\overline{F}^n = \overline{F}'^n$ for some integer $n>0$. However, for any $n > 0$ we have
\begin{equation*}
\overline{F}^n = \Inn \overline{F}'(t)\overline{F}'^2(t)\cdots \overline{F}'^n(t) \circ \overline{F}'^n.
\end{equation*}
The statement now follows because every element of $\bG$ is fixed by some power of $\overline{F}'$ and has finite order.
\end{proof}



\section{Unipotent and Nilpotent Elements}\label{sec:uni-nil-elements}
\begin{assumption}
From this point forward we assume that $p$ is a good prime for $\bG$. Moreover we assume that $\bT_0 \leqslant \bB_0 \leqslant \bG$ are a fixed choice of $F$-stable maximal torus and Borel subgroup. We will denote by $\Phi \subseteq X(\bT_0)$, resp., $\widecheck{\Phi} \subseteq \widecheck{X}(\bT_0)$, the roots, resp., coroots, of $\bG$ with respect to $\bT_0$. Furthermore, we denote by $\Delta \subset \Phi^+ \subset \Phi$ the set of simple and positive roots determined by $\bB_0$.
\end{assumption}

\subsection{Springer Isomorphisms}
\begin{pa}
As $\bG$ is equipped with a Frobenius endomorphism so is $\lie{g}$ and we will denote this again by $F : \lie{g} \to \lie{g}$. Our main interest of study will be the variety of unipotent elements $\mathcal{U}(\bG) \subset \bG$ and the nilpotent cone $\mathcal{N}(\lie{g}) \subset \lie{g}$. If there is no ambiguity over the ambient group, resp., Lie algebra, then we will simply write $\mathcal{U}$, resp., $\mathcal{N}$, for $\mathcal{U}(\bG)$, resp., $\mathcal{N}(\lie{g})$. To relate these two objects we will need the existence of a \emph{Springer isomorphism}, which is a $\bG$-equivariant isomorphism of varieties $\phi_{\spr} : \mathcal{U} \to \mathcal{N}$ compatible with the Frobenius endomorphisms on $\bG$ and $\lie{g}$. Note that this is $\bG$-equivariant in the sense that
\begin{equation*}
\phi_{\spr}\circ\Inn g = \Ad g \circ \phi_{\spr}
\end{equation*}
for all $g \in \bG$ where $\Ad g = d(\Inn g) : \lie{g} \to \lie{g}$ is the differential of $\Inn g : \bG \to \bG$, c.f., \cref{pa:pairs}.
\end{pa}

\begin{pa}
We will need the following well known result which states that a separable isogeny between connected algebraic groups restricts to an isomorphism between unipotent varieties. This result was stated by Springer in \cite{springer:1969:the-unipotent-variety} but he omits the proof. An argument is given in \cite[\S6.2]{humphreys:1995:conjugacy-classes} for this statement but this seems only to show that the restriction is a birational morphism. If one knew the target unipotent variety was normal then one could conclude that this was indeed an isomorphism but usually one uses such an isomorphism to deduce the unipotent variety is normal so this is circular. The following argument was communicated to us by George McNinch; we thank him for kindly allowing us to include it here.
\end{pa}

\begin{lem}\label{lem:sep-isog-iso-uni-vars}
Assume $\bH$ and $\bH'$ are connected affine algebraic groups defined over $\mathbb{K}$. If $\phi : \bH \to \bH'$ is a separable isogeny then $\phi$ restricts to an isomorphism $\mathcal{U}(\bH) \to \mathcal{U}(\bH')$ which is equivariant with respect to the natural conjugation action.
\end{lem}

\begin{proof}
Let us denote by $Z$ the kernel of $\phi$, which is by assumption a finite subgroup of the centre $Z(\bH)$. We start by noting that $(\bH',\phi)$ is the affine quotient of $\bH$ by $Z$ in the sense of \cite[2.5.8]{geck:2003:intro-to-algebraic-geometry}. Indeed, if $\pi : \bH \to \bH/Z$ is the natural projection map then there exists a unique morphism $\beta : \bH/Z \to \bH'$ such that $\phi = \pi\circ\beta$ by the universal property of the quotient. Clearly $\beta$ is bijective. Moreover, by assumption, the field extension $\mathbb{K}(\bH)/\phi^*(\mathbb{K}(\bH'))$ is separable hence so is $\pi^*(\mathbb{K}(\bH/Z))/\phi^*(\mathbb{K}(\bH'))$ but as $\pi^*$ is injective this implies $\mathbb{K}(\bH/Z)/\beta^*(\mathbb{K}(\bH')$ is separable. Hence $\beta$ is bijective and separable and so is an isomorphism, c.f., \cite[4.3.7(iii), 5.3.3(ii)]{springer:2009:linear-algebraic-groups}.

Now let us denote by $\bV$ the preimage $\phi^{-1}(\mathcal{U}(\bH'))$ of the unipotent variety under $\phi$. The Jordan decomposition shows that $\bV$ is the disjoint union $\bigsqcup_{z \in Z} z\mathcal{U}(\bH)$ and clearly each $z\mathcal{U}(\bH)$ is a closed subvariety isomorphic to $\mathcal{U}(\bH)$. The kernel $Z$ acts simply transitively on $\bV$ by left multiplication and we claim that $(\mathcal{U}(\bH'),\phi)$ is the affine quotient of $\bV$ by $Z$. Certainly the fibres of $\phi$ are $Z$-orbits and $\phi$ is an open map by \cite[2.5.6(b)]{geck:2003:intro-to-algebraic-geometry}. We just need to show that the map $\phi^* : \mathbb{K}[\mathcal{U}(\bH')] \to \mathbb{K}[\bV]^Z$ is surjective. If $\iota_{\bV} : \bV \to \bH$ and $\iota_{\mathcal{U}(\bH')} : \mathcal{U}(\bH') \to \bH'$ are the natural closed embeddings then we have a commutative square
\begin{equation*}
\begin{tikzpicture}[baseline={([yshift=-1ex]current bounding box.center)}]
\matrix (m) [matrix of math nodes, row sep=3em, column sep=3em, text height=1.5ex, text depth=0.25ex]
{ \mathbb{K}[\bH'] & \mathbb{K}[\bH]^Z\\
\mathbb{K}[\mathcal{U}(\bH')] & \mathbb{K}[\bV]^Z\\};

\path [>=stealth,->]	(m-1-1) edge node [above] {$\phi^*$} (m-1-2)
		(m-1-1) edge node [left] {$\iota_{\mathcal{U}(\bH)}^*$} (m-2-1)
		(m-2-1) edge node [above] {$\phi^*$} (m-2-2)
		(m-1-2) edge node [left] {$\iota_{\bV}^*$} (m-2-2);
\end{tikzpicture}
\end{equation*}
where the vertical arrows are surjective. By the first part the top arrow is surjective, which proves the claim.

Now let $\alpha : \mathcal{U}(\bH) \to \bV/Z$ be the morphism obtained as the composition of the natural closed embedding $\mathcal{U}(\bH) \to \bV$ and the natural projection $\bV \to \bV/Z$. As $\phi : \mathcal{U}(\bH) \to \mathcal{U}(\bH')$ factors as a composition of $\alpha$ and an isomorphism $\bV/Z \to \mathcal{U}(\bH')$ it suffices to show that $\alpha$ is an isomorphism. As $\bV = \bigsqcup_{z \in Z} z\mathcal{U}(\bH)$ we have a $\mathbb{K}$-algebra isomorphism $\mathbb{K}[\bV] \to \prod_{z \in Z} \mathbb{K}[z\mathcal{U}(\bH)]$ defined by $f \mapsto (f|_{z\mathcal{U}(\bH)})_{z \in Z}$. This clearly maps $\mathbb{K}[\bV]^Z$ onto the subalgebra $\{(f_z) \in \prod_{z \in Z} \mathbb{K}[z\mathcal{U}(\bH)] \mid f_z(zu) = f_z(u)$ for all $z \in Z$ and $u \in \mathcal{U}(\bH)\}$. Moreover, under this identification the comorphism of the closed embedding $\mathcal{U}(\bH) \to \bV$ is simply the projection onto $\mathbb{K}[\mathcal{U}(\bH)]$, i.e., the factor indexed by the identity. With this it is clear that the comorphism $\alpha^* : \mathbb{K}[\bV/Z] \to \mathbb{K}[\mathcal{U}(\bH)]$ is an isomorphism, hence $\alpha$ is an isomorphism as desired.
\end{proof}

\begin{lem}\label{lem:springer-morphism}
If $\bG$ is proximate then there exists a Springer isomorphism $\phi_{\spr} : \mathcal{U}(\bG) \to \mathcal{N}(\lie{g})$.
\end{lem}

\begin{proof}
Firstly let us note that $\mathcal{U}(\bG) = \mathcal{U}(\bG_{\der})$ and $\mathcal{N}(\lie{g}) = \mathcal{N}(\lie{g}_{\der})$, e.g., see \cite[pg.\ 2966]{premet:1995:an-analogue}. As $\phi_{\simc} : \bG_{\simc} \to \bG_{\der}$ is assumed to be a separable isogeny we have by \cref{lem:sep-isog-iso-uni-vars} that it restricts to an equivariant isomorphism $\alpha : \mathcal{U}(\bG_{\simc}) \to \mathcal{U}(\bG_{\der})$. Now we may view the vector spaces $\lie{g}_{\simc}$ and $\lie{g}_{\der}$ as affine varieties, in which case the nilpotent cones are closed subvarieties. By \cref{prop:lie-algebra-iso} the differential $d\phi_{\simc}$ is an isomorphism of vector spaces which must therefore also be an isomorphism of varieties. This is because the inverse of $d\phi_{\simc}$ is a linear map and any linear map between vector spaces is clearly a morphism of varieties. In particular, we have $d\phi_{\simc}$ restricts to an equivariant isomorphism $\beta : \mathcal{N}(\lie{g}_{\simc}) \to \mathcal{N}(\lie{g}_{\der})$.

Let us denote by $F_{\simc} : \bG_{\simc} \to \bG_{\simc}$ a Frobenius endomorphism such that $F\circ \phi_{\simc} = \phi_{\simc} \circ F_{\simc}$, which exists by \cite[9.16]{steinberg:1968:endomorphisms-of-linear-algebraic-groups}. By \cite[Corollary 8.5]{jantzen:2004:nilpotent-orbits} we know that the nilpotent cone $\mathcal{N}(\lie{g}_{\simc})$ is a normal variety. Hence, by \cite[III, 3.12]{springer-steinberg:1970:conjugacy-classes} and the remark following the statement of the theorem there exists a Springer isomorphism $\tilde{\phi}_{\spr} : \mathcal{U}(\bG_{\simc}) \to \mathcal{N}(\lie{g}_{\simc})$ with respect to $F_{\simc}$, see also \cite[\S6.20]{humphreys:1995:conjugacy-classes}. Clearly one then also has that $\phi_{\spr} = \beta \circ \tilde{\phi}_{\spr}\circ \alpha^{-1} : \mathcal{U}(\bG_{\der}) \to \mathcal{N}(\lie{g}_{\der})$ is an equivariant isomorphism as it is a composition of equivariant isomorphisms. Finally, as $\bG = \bG_{\der}Z(\bG)$ and $Z(\bG)$ acts trivially on both $\mathcal{U}$ and $\mathcal{N}$ we may consider $\phi_{\spr}$ to be $\bG$-equivariant. Hence, we have $\phi_{\spr}$ is a Springer isomorphism.
\end{proof}

\begin{rem}\label{rem:conv-spring-iso}
It seems reasonable to suspect that the converse to \cref{lem:springer-morphism} is also true. In other words, we have a Springer isomorphism if and only if $\bG$ is proximate. Note also that the Springer isomorphism $\phi_{\spr}$ is by no means unique. In fact, in the appendix to \cite{mcninch:2005:optimal-sl2-homomorphisms} Serre has shown that the Springer isomorphisms form a variety whose dimension is given by the rank of $\bG$.
\end{rem}

\subsection{Separability of Centralisers}
\begin{pa}\label{pa:centralisers}
An issue for us in this article will also be the so-called separability of centralisers. For any subset $\lie{h} \subseteq \lie{g}$ we define
\begin{align*}
C_{\bG}(\lie{h}) &= \{g \in \bG \mid \Ad g(x) = x\text{ for all }x \in \lie{h}\},\\
\lie{c_g}(\lie{h}) &= \{y \in \lie{g} \mid [y,x] = 0\text{ for all }x \in \lie{h}\}.
\end{align*}
If $\lie{h} = \{x\}$ then we simply write $C_{\bG}(x)$, resp., $\lie{c_g}(x)$, for $C_{\bG}(\lie{h})$, resp., ${\lie{c_g}}(\lie{h})$. Furthermore we will denote by $\lie{z(g)} = \bigcap_{x \in \lie{g}} \lie{c_g}(x)$ the centre of the Lie algebra. Note that, in general, we do \emph{not} have $\lie{z(g)} = \Lie(Z(\bG))$. We will be interested in knowing when $C_{\bG}(x)$ is \emph{separable} in the sense that $\Lie(C_{\bG}(x)) = \lie{c_g}(x)$. The following gives some equivalent characterisations of separability.
\end{pa}

\begin{lem}[{}{\cite[II, 6.7]{borel:1991:linear-algebraic-groups}}]\label{lem:separable-centraliser}
For any element $x \in \lie{g}$ the following are equivalent:
\begin{enumerate}
	\item $C_{\bG}(x)$ is separable,
	\item the orbit map $\pi_x : \bG \to (\Ad\bG)x$ defined by $\pi_x(g) = (\Ad g)x$ is separable,
	\item $\pi_x$ factors as $\sigma \circ \overline{\pi}_x$ where $\sigma : \bG \to \bG/C_{\bG}(x)$ is the natural projection morphism and $\overline{\pi}_x : \bG/C_{\bG}(x) \to (\Ad\bG)x$ is an isomorphism of varieties,
	\item $T_x((\Ad\bG)x) = [\lie{g},x]$.
\end{enumerate}
\end{lem}

\begin{pa}
In \cite{herpel:2013:on-the-smoothness} Herpel has considered when the scheme-theoretic centraliser of a closed subgroup scheme of $\bG$ is smooth and shown that this is related to the notion of separability. To apply his results to centralisers we will need to recall his elegant notion of a pretty good prime.
\end{pa}

\begin{definition}[{}{Herpel, \cite[Definition 2.11]{herpel:2013:on-the-smoothness}}]\label{def:pretty-good-prime}
Assume $\mathcal{R}(\bG,\bT) = (X,\Phi,Y,\widecheck{\Phi})$ is the root datum of $\bG$ with respect to some maximal torus $\bT\leqslant \bG$ then we say $p$ is a \emph{pretty good prime} for $\bG$ if $X/\mathbb{Z}\Psi$ and $Y/\mathbb{Z}\widecheck{\Psi}$ have no $p$-torsion for any subsets $\Psi \subseteq \Phi$ and $\widecheck{\Psi} \subseteq \widecheck{\Phi}$.
\end{definition}

\begin{prop}[Herpel]\label{prop:smooth-centraliser}
Assume $p$ is a pretty good prime for $\bG$ then $C_{\bG}(x)$ is separable for all $x \in \lie{g}$.
\end{prop}

\begin{proof}
Let $\lie{h} \subseteq \lie{g}$ be the 1-dimensional subalgebra generated by $x$ then it is clear that we have $C_{\bG}(x) = C_{\bG}(\lie{h})$ and $\lie{c_g}(x) = \lie{c_g}(\lie{h})$. In particular, to show that $C_{\bG}(x)$ is separable it suffices to show that $C_{\bG}(\lie{h})$ is separable. In \cite[Lemma 3.1(ii)]{herpel:2013:on-the-smoothness} Herpel constructs a closed subgroup scheme $\bH' \leqslant \bG$ and shows that $C_{\bG}(\lie{h})$ is separable if and only if the scheme-theoretic centraliser of $\bH'$ in $\bG$ is smooth. By \cite[Theorem 1.1]{herpel:2013:on-the-smoothness} the scheme-theoretic centraliser of any closed subgroup scheme of $\bG$ is smooth if $p$ is a pretty good prime. Hence we can conclude that $C_{\bG}(x)$ is smooth.
\end{proof}

\begin{rem}
If $\bG$ is simple or $\GL_n(\mathbb{K})$ then this result is classical and contained in \cite[I, 5.6]{springer-steinberg:1970:conjugacy-classes}, see also \cite[3.13, Theorem]{slodowy:1980:simple-singularities}.
\end{rem}

\subsection{$\mathbb{G}_m$-varieties}
\begin{pa}\label{pa:action-on-aff-alg}
We will denote by $\mathbb{G}_m$ the set $\mathbb{K}\setminus\{0\}$ viewed as an algebraic group under multiplication. It will be useful to recall here some properties of $\mathbb{G}_m$-actions that we will use several times. Assume $\bX$ is an affine $\mathbb{G}_m$-variety, in the sense of \cite[\S2.3.1]{springer:2009:linear-algebraic-groups}, with action map $\alpha : \mathbb{G}_m \times \bX \to \bX$. The affine algebra $\mathbb{K}[\bX]$ of $\bX$ then becomes an abstract $\mathbb{G}_m$-module by setting
\begin{equation*}
(k\cdot f)(x) = f(\alpha(k^{-1},x))
\end{equation*}
for all $k \in \mathbb{G}_m$, $f \in \mathbb{K}[\bX]$ and $x \in \bX$. By \cite[2.3.9(i),3.2.3(c)]{springer:2009:linear-algebraic-groups} we have $\mathbb{K}[\bX] = \bigoplus_{n \in \mathbb{Z}} \mathbb{K}[\bX]_n$ where we have
\begin{equation*}
\mathbb{K}[\bX]_n = \{f \in \mathbb{K}[\bX] \mid k\cdot f = k^nf\text{ for all }k \in \mathbb{G}_m\}
\end{equation*}
is the corresponding weight space. For any $x \in \bX$ we have a morphism $\alpha_x : \mathbb{G}_m \to \bX$ defined by setting $\alpha_x(k) = \alpha(k,x)$ for all $k \in \mathbb{G}_m$. We say the limit $\lim_{k \to 0} \alpha_x(k)$ exists if the morphism $\alpha_x : \mathbb{G}_m \to \bX$ extends to a morphism $\tilde{\alpha}_x : \mathbb{A}^1 \to \bX$; note such an extension is unique by \cite[1.6.11(ii)]{springer:2009:linear-algebraic-groups}. Moreover we write $\lim_{k \to 0} \alpha_x(k) = z$ if $\tilde{\alpha}_x(0) = z$. Now we say the $\mathbb{G}_m$-action is \emph{contracting} if there exists a fixed point $x_0 \in \bX^{\mathbb{G}_m}$ such that $\lim_{k \to 0} \alpha_x(k) = x_0$ for all $x \in \bX$. We will also say that the action is a \emph{contraction to $x_0$}. With this we have the following.
\end{pa}

\begin{lem}\label{lem:contracting-action}
If $\alpha : \mathbb{G}_m \times \bX \to \bX$ is a contracting $\mathbb{G}_m$-action then the following hold:
\begin{enumerate}
	\item the fixed point $x_0 \in \bX^{\mathbb{G}_m}$ is unique,
	\item $\mathbb{K}[\bX]_n = \{0\}$ for all $n < 0$,
	\item $\mathbb{K}[\bX]_n$ is finite dimensional as a $\mathbb{K}$-vector space for all $n \in \mathbb{Z}$.
\end{enumerate}
\end{lem}

\begin{proof}
(i). Assume $x_0' \in X^{\mathbb{G}_m}$ was another fixed point and let $\tilde{\alpha}_{x_0'} : \mathbb{A}^1 \to \bX$ be the morphism extending $\alpha_{x_0'} : \mathbb{G}_m \to \bX$. Let $f_{x_0'} : \mathbb{A}^1 \to \bX$ be the morphism defined by $f_{x_0'}(k) = x_0'$ then as $x_0'$ is fixed by the $\mathbb{G}_m$ action we have $\tilde{\alpha}_{x_0'}(k) = f_{x_0'}(k)$ for all $k \in \mathbb{G}_m$. As noted above, the extending morphism is unique so we must have $\tilde{\alpha}_{x_0'} = f_{x_0'}$. However, by assumption $\tilde{\alpha}_{x_0'}(0) = x_0$ which implies $x_0' = x_0$.

(ii). For any $x \in \bX$ we have the morphism $\alpha_x : \mathbb{G}_m \to \bX$ extends to $\mathbb{A}^1$ if and only if the image of the comorphism $\alpha_x^* : \mathbb{K}[\bX] \to \mathbb{K}[\mathbb{G}_m] = \mathbb{K}[T,T^{-1}]$ is contained in $\mathbb{K}[T]$. For any $f \in \mathbb{K}[\bX]$ we may write $f = \sum_{n \in \mathbb{Z}} f_n$ with $f_n \in \mathbb{K}[\bX]_n$. It's easy to see that we have
\begin{equation*}
\alpha_x^*(f) = \sum_{n \in \mathbb{Z}} f_n(x)T^n.
\end{equation*}
Hence for any $n < 0$ we must have $f_n(x) = 0$ for all $x \in \bX$ and so $\mathbb{K}[\bX]_n = \{0\}$ for any $n<0$.

(iii). By part (ii) we need only show that $\dim_{\mathbb{K}} \mathbb{K}[\bX]_n <\infty$ when $n \geqslant 0$. Moreover, as $\mathbb{K}[\bX]_0$ is nothing other than the affine algebra $\mathbb{K}[\bX^{\mathbb{G}_m}]$ of the fixed points we have $\mathbb{K}[\bX]_0 = \mathbb{K}$ by (i) so we may assume $n > 0$. Let $\{g_1,\dots,g_k\} \subseteq \mathbb{K}[\bX]$ be a generating set for the algebra then the set $\mathcal{B} = \{g_1^{i_1}\cdots g_k^{i_k} \mid i_1,\dots,i_k \geqslant 0\}$ is a $\mathbb{K}$-basis of the algebra, where we assume that not all powers are 0. Let us assume that $g_j$ is contained in the weight space $\mathbb{K}[\bX]_{n_j}$ then by (i) and (ii) we have $n_j > 0$. The subset $\mathcal{B}_n = \{g_1^{i_1}\cdots g_k^{i_k} \mid n = n_1i_1+\cdots + n_ki_k\} \subseteq \mathcal{B}$ is obviously a basis for the weight space $\mathbb{K}[\bX]_n$. However as $n,n_1,\dots,n_k > 0$ are all strictly positive we see that the set $\mathcal{B}_n$ must be finite.
\end{proof}

\begin{rem}\label{rem:contracting-lin-action}
Assume that $\bX$ is a $\mathbb{K}$-vector space and the $\mathbb{G}_m$-action $\alpha : \mathbb{G}_m \times \bX \to \bX$ is linear in the sense that $\alpha(k,x) = \rho(k)x$ for some rational representation $\rho : \mathbb{G}_m \to \GL(\bX)$. In this situation we can directly break up $\bX$ as a direct sum $\bigoplus_{n \in \mathbb{Z}} \bX_n$ of its weight spaces, c.f., \cite[3.2.3(c)]{springer:2009:linear-algebraic-groups}. Using the above arguments one easily sees that the action is contracting if and only if $\bX_n = \{0\}$ for all $n \leqslant 0$; then $0$ is the unique fixed point of the action.
\end{rem}

\subsection{Cocharacters}
\begin{pa}\label{pa:weight-spaces}
We now recall some results from \cite{premet:2003:nilpotent-orbits-in-good-characteristic}. Assume $\lambda \in Y(\bG) = \Hom(\mathbb{G}_m,\bG)$ is a cocharacter of $\bG$ then $\lambda$ defines a $\mathbb{Z}$-grading $\lie{g} = \bigoplus_{i \in \mathbb{Z}} \lie{g}(\lambda,i)$ on the Lie algebra by setting
\begin{equation*}
\lie{g}(\lambda,i) = \{x \in \lie{g} \mid (\Ad \lambda(k))(x) = k^ix\text{ for all }k \in \mathbb{G}_m\}.
\end{equation*}
From the definition we see immediately that
\begin{equation}\label{eq:commutator-weight-space}
[\lie{g}(\lambda,i),\lie{g}(\lambda,j)] \subseteq \lie{g}(\lambda,i+j)
\end{equation}
for all $i,j \in \mathbb{Z}$. We note the following useful observation concerning the weight spaces. It is clear that we have $\lie{g} = \lie{g}_{\der} + \Lie(Z(\bG))$ but this sum need not be direct. Furthermore, it is also obvious that $\Lie(Z(\bG)) \subseteq \lie{g}(\lambda,0)$ so we must have $\lie{g}(\lambda,i) \subseteq \lie{g}_{\der}$ for any $i \neq 0$. In particular, if $\lambda \in Y(\bG_{\der}) \subseteq Y(\bG)$ then we have $\lie{g}(\lambda,i) = \lie{g}_{\der}(\lambda,i)$ for all $i \neq 0$.
\end{pa}

\begin{pa}\label{pa:canonical-parabolic-levi}
To any cocharacter $\lambda \in Y(\bG)$ we assign a parabolic subgroup $\bP(\lambda) \leqslant \bG$ with unipotent radical $\bU(\lambda)$ given by
\begin{align*}
\bP(\lambda) &= \{ x \in \bG \mid \lim_{k\to 0} \lambda(k)x\lambda(k)^{-1}\text{ exists}\}\\
\bU(\lambda) &= \{x \in \bG \mid \lim_{k\to 0} \lambda(k)x\lambda(k)^{-1} = 1\},
\end{align*}
see \cite[3.2.15, 8.4.5]{springer:2009:linear-algebraic-groups}. We then have a Levi decomposition $\bP(\lambda) = \bL(\lambda)\bU(\lambda)$ by setting $\bL(\lambda) = C_{\bG}(\lambda(\mathbb{G}_m))$. The Lie algebras of these subgroups are given by
\begin{equation*}
\lie{p}(\lambda) = \bigoplus_{i \geqslant 0} \lie{g}(\lambda,i) \qquad\qquad \lie{l}(\lambda) = \bigoplus_{i = 0} \lie{g}(\lambda,i) \qquad\qquad \lie{u}(\lambda) = \bigoplus_{i > 0} \lie{g}(\lambda,i).
\end{equation*}
For any $i \in \mathbb{Z}_{> 0}$ we will also need the Lie subalgebra $\lie{u}(\lambda,i) = \bigoplus_{j\geqslant i} \lie{g}(\lambda,j)$ and its corresponding closed connected unipotent subgroup $\bU(\lambda,i) \leqslant \bU(\lambda)$.

Assume now that $\lambda \in Y(\bT_0)$ and let $\bX_{\alpha} \leqslant \bG$ be the 1-dimensional unipotent root subgroup whose Lie algebra is the root space $\lie{g}_{\alpha}$ for any $\alpha \in \Phi$. If $\bP(\lambda)$ is standard, i.e., it contains $\bT_0 \leqslant \bB_0$, then we may also write the above as
\begin{align*}
\bP(\lambda) &= \langle \bT_0, \bX_{\alpha} \mid \langle \alpha,\lambda\rangle \geqslant 0 \rangle & \bL(\lambda) &= \langle \bT_0, \bX_{\alpha} \mid \langle \alpha,\lambda\rangle = 0 \rangle & \bU(\lambda) &= \langle \bX_{\alpha} \mid \langle \alpha,\lambda\rangle > 0 \rangle,\\
\lie{p}(\lambda) &= \lie{t}_0 \oplus\bigoplus_{\langle \alpha,\lambda \rangle \geqslant 0} \lie{g}_{\alpha} &
\lie{l}(\lambda) &= \lie{t}_0 \oplus\bigoplus_{\langle \alpha,\lambda \rangle = 0} \lie{g}_{\alpha} &
\lie{u}(\lambda) &= \bigoplus_{\langle \alpha,\lambda \rangle > 0} \lie{g}_{\alpha},
\end{align*}
as $\lie{t}_0 \subseteq \lie{g}(\lambda,0)$. With this we see that we have an analogue of \cref{eq:commutator-weight-space} which follows immediately from Chevalley's commutator relation \cite[8.2.3]{springer:2009:linear-algebraic-groups}. Namely,
\begin{equation}\label{eq:commutator-weight-space-grp}
[\bU(\lambda,i),\bU(\lambda,j)] \subseteq \bU(\lambda,i+j)
\end{equation}
for any $i,j \geqslant 1$. Note this clearly holds for all cocharacters.
\end{pa}

\subsection{Weighted Dynkin Diagrams}
\begin{pa}\label{pa:chevalley-basis}
Let us assume temporarily that $\bG$ is semisimple and simply connected. We will denote by $\bG_{\mathbb{C}}$ an algebraic group over $\mathbb{C}$ and by $\bT_{\mathbb{C}} \leqslant \bB_{\mathbb{C}} \leqslant \bG_{\mathbb{C}}$ a maximal torus and Borel subgroup of $\bG_{\mathbb{C}}$ such that the root datum, and simple roots, of $\bG_{\mathbb{C}}$ with respect to $\bT_{\mathbb{C}} \leqslant \bB_{\mathbb{C}}$ is the same as that of $\bG$. Fix a Chevalley basis $\mathcal{B}_{\mathbb{C}} = \{X_{\alpha} \mid \alpha \in \Phi\} \cup \{H_{\alpha} \mid \alpha \in \Delta\}$ of $\lie{g}_{\mathbb{C}}$ and denote by $\lie{g}_{\mathbb{Z}} \subset \lie{g}_{\mathbb{C}}$ the $\mathbb{Z}$-span of $\mathcal{B}_{\mathbb{C}}$. As $\bG$ is simply connected we may identify $\lie{g}$ with $\lie{g}_{\mathbb{Z}} \otimes_{\mathbb{Z}} \mathbb{K}$ as Lie algebras. Setting $e_{\alpha} = X_{\alpha} \otimes 1$ and $h_{\alpha} = H_{\alpha} \otimes 1$ for each $\alpha \in \Phi$ we have $\mathcal{B} = \{e_{\alpha} \mid \alpha \in \Phi\} \cup \{h_{\alpha} \mid \alpha \in \Delta\}$ is a basis for $\lie{g}$ and the corresponding root space $\lie{g}_{\alpha}$ of $\lie{g}$ is simply $\mathbb{K}e_{\alpha}$.

By the Jacobson--Morozov theorem, see \cite[5.3.2]{carter:1993:finite-groups-of-lie-type}, any nilpotent element $e \in \lie{g}_{\mathbb{C}} = \Lie(\bG_{\mathbb{C}})$ is contained in an $\lie{sl}_2$-triple $\{e,h,f\} \subset \lie{g}_{\mathbb{C}}$. Arguing as in \cite[pg.\ 344]{premet:2003:nilpotent-orbits-in-good-characteristic} we may assume, after possibly replacing $\{e,h,f\}$ by $\{\Ad g(e), \Ad g(h), \Ad g(f)\}$ for some $g \in \bG$, that $h = \sum_{\alpha \in \Delta} q_{\alpha}H_{\alpha}$ with $q_{\alpha} \in \mathbb{Z}$ and $\alpha(h) \geqslant 0$ for all $\alpha \in \Delta$. The function $d : \Phi \to \mathbb{Z}$ given by $\alpha \mapsto \alpha(h)$ is called the \emph{weighted Dynkin diagram} of the nilpotent orbit $\mathcal{O} \subset \lie{g}_{\mathbb{C}}$ containing $e$. We denote by $\mathcal{D}_{\Phi}$ the set of weighted Dynkin diagrams.

Note that, for each weighted Dynkin diagram $d \in \mathcal{D}_{\Phi}$, there exists a cocharacter $\lambda_d^{\bG} \in Y(\bT_0) \subseteq Y(\bG)$ such that $\Ad \lambda_d^{\bG}(k)(e_{\alpha}) = k^{d(\alpha)}e_{\alpha}$ and $\Ad \lambda_d^{\bG}(k)(x) = x$ for all $\alpha \in \pm\Delta$, $x \in \lie{t}_0$ and $k \in \mathbb{G}_m$. In fact this can be constructed by setting $\lambda_d^{\bG} = \sum_{\alpha \in \Delta} q_{\alpha}\widecheck{\alpha} \in \mathbb{Z}\widecheck{\Phi}$ where $q_{\alpha} \in \mathbb{Z}$ is as above and $\widecheck{\alpha} \in \widecheck{\Phi}$ is the coroot corresponding to $\alpha \in \Delta$.
\end{pa}

\begin{pa}\label{pa:chevalley-basis-general}
We now drop our assumption that $\bG$ is semisimple and simply connected. Let us denote by $\bT_{\simc} \leqslant \bG_{\simc}$ the maximal torus such that $\phi_{\simc}(\bT_{\simc}) \leqslant \bT_0$. The isogeny $\phi_{\simc}$ then induces a natural injection $Y(\bG_{\simc}) \to Y(\bG)$ which maps $Y(\bT_{\simc})$ into $Y(\bT_0)$. In \cref{pa:chevalley-basis} we have defined a cocharacter $\lambda_d^{\bG_{\simc}} \in Y(\bT_{\simc})$ for each weighted Dynkin diagram $d \in \mathcal{D}_{\Phi}$. We now set $\lambda_d^{\bG} := \phi_{\simc}\circ\lambda_d^{\bG_{\simc}}$ and denote by $\mathcal{D}_{\Phi}(\bG)$ the set $\{\lambda_d^{\bG} \mid d \in \mathcal{D}_{\Phi}\}$ of resulting cocharacters. Note that the definition of $\lambda_d^{\bG}$ does not depend upon the choice of simply connected cover $\phi_{\simc}$. Indeed, if $\phi_{\simc}'$ is another simply connected cover then $\phi_{\simc}' = \phi_{\simc}\circ\Inn t$ for some element $t \in \bT_{\simc}$ by \cref{thm:isogeny-theorem}. However, clearly $\Inn t\circ\lambda_d^{\bG_{\simc}} = \lambda_d^{\bG_{\simc}}$ because the image is contained in $\bT_{\simc}$, which is abelian. Moreover we will denote by $\mathcal{D}(\bG)$ the set
\begin{equation*}
\{\Inn g\circ\lambda_d^{\bG} \mid d \in \mathcal{D}_{\Phi}\text{ and }g \in \bG\}.
\end{equation*}
Finally we remark that if $\bG$ is a proximate algebraic group then $\mathrm{d}\phi_{\simc}$ is an isomorphism so we may, and will, identify the Chevalley basis considered in \cref{pa:chevalley-basis} with a basis of $\lie{g}_{\der}$.
\end{pa}

\subsection{Classification of Nilpotent Orbits}
\begin{pa}
We now have the following cheap generalisation of \cite{premet:2003:nilpotent-orbits-in-good-characteristic}, which gives a case-free proof of the classification of nilpotent orbits in good characteristic by weighted Dynkin diagrams. For historical remarks concerning this theorem see \cite[Remark 2]{premet:2003:nilpotent-orbits-in-good-characteristic}. In fact, we will also give certain statements concerning centralisers of nilpotent elements but for this we will need to assume that $\bG$ is proximate. This assumption cannot be dropped, in general, as is evident by the example given at the end of \cite[Introduction]{premet:1995:an-analogue}.
\end{pa}

\begin{thm}[Kawanaka, Premet]\label{thm:class-nilpotent-orbits}
For any cocharacter $\lambda \in \mathcal{D}(\bG) \subseteq Y(\bG)$ let $\lie{g}(\lambda,2)_{\reg}$ be the unique open dense orbit of $\bL(\lambda)$ acting on $\lie{g}(\lambda,2)$, c.f., \cite[pg.\ 132]{carter:1993:finite-groups-of-lie-type}.
\begin{enumerate}
	\item The map $\mathcal{D}_{\Phi}(\bG) \to \mathcal{N}(\lie{g})/\bG$ given by $\lambda_d^{\bG} \mapsto \mathcal{O}_{\bG}(d) := (\Ad\bG)\lie{g}(\lambda_d^{\bG},2)_{\reg}$ is a bijection.
	
	\item Assume $\bG$ is proximate then for any $\lambda \in \mathcal{D}(\bG)$ and $e \in \lie{g}(\lambda,2)_{\reg} \subseteq \mathcal{N}(\lie{g})$ the following hold:
\begin{enumerate}
	\item $\lie{c_g}(e) \subseteq \lie{p}(\lambda)$ and $C_{\bG}(e) \subseteq \bP(\lambda)$,
	\item $[e,\lie{p}(\lambda)] = \lie{u}(\lambda,2)$,
	\item $(\Ad\bP(\lambda))e$ is dense in $\lie{u}(\lambda,2)$.
\end{enumerate}
\end{enumerate}
\end{thm}

\begin{proof}
(i). Assume $\varphi : \bG \to \bH$ is a surjective homomorphism of algebraic groups such that $\Ker(\varphi) \subseteq Z(\bG)$ and $\Ker(\mathrm{d}\varphi) \subseteq \lie{z(g)}$. Let $\lie{h} = \Lie(\bH)$ then according to \cite[2.7, Proposition]{jantzen:2004:nilpotent-orbits} the differential $\mathrm{d}\varphi : \lie{g} \to \lie{h}$ induces a bijection $\mathcal{N}(\lie{g})/\bG \to \mathcal{N}(\lie{h})/\bH$. Now the restriction of $\varphi$ to the derived subgroup $\bG_{\der}$ defines an isogeny between the derived subgroups $\bG_{\der} \to \bH_{\der}$. Moreover $\varphi\circ\phi_{\simc} : \bG_{\simc} \to \bH_{\der}$ is a simply connected covering. As remarked in \cref{pa:chevalley-basis-general} the cocharacter $\lambda_d^{\bH}$ does not depend upon the choice of simply connected cover used to define it. Hence, we must have $\lambda_d^{\bH} = \varphi\circ\lambda_d^{\bG}$ for any weighted Dynkin diagram $d \in \mathcal{D}_{\Phi}$. In particular, this shows that $\varphi$ induces a bijection $\mathcal{D}_{\Phi}(\bG) \to \mathcal{D}_{\Phi}(\bH)$. It is easily checked that $\varphi(\bL(\lambda_d^{\bG})) = \bL(\lambda_d^{\bH})$ and $\mathrm{d}\varphi(\lie{g}(\lambda_d^{\bG},2)) = \lie{h}(\lambda_d^{\bH},2)$, c.f., \cref{pa:canonical-parabolic-levi} and \cite[4.4.11(7)]{springer:2009:linear-algebraic-groups}. Moreover $\mathrm{d}\varphi$ induces a bijection between the $\bL(\lambda_d^{\bG})$-orbits on $\lie{g}(\lambda_d^{\bG},2)$ and the $\bL(\lambda_d^{\bH})$-orbits on $\lie{h}(\lambda_d^{\bH},2)$. Now $\mathrm{d}\varphi : \mathcal{N}(\lie{g}) \to \mathcal{N}(\lie{h})$ is a homeomorphism as it is a bijective morphism between irreducible varieties, c.f., \cite[5.2.3, 5.2.6, 5.2.9(4)]{springer:2009:linear-algebraic-groups}. In particular, the restriction $\mathrm{d}\varphi : \lie{g}(\lambda_d^{\bG},2) \to \lie{h}(\lambda_d^{\bH},2)$ is also a homeomorphism so clearly $\mathrm{d}\varphi(\mathcal{O}_{\bG}(d)) = \mathcal{O}_{\bH}(d)$. With this we see that (i) holds in $\bG$ if and only if (i) holds in $\bH$.

Now let $\bG_{\ad} = \Ad(\bG)$ be the image of the adjoint representation $\Ad : \bG \to \GL(V)$, where $V = \lie{g}$. Certainly $\Ad$ satisfies the above hypotheses, c.f., \cite[2.7]{jantzen:2004:nilpotent-orbits}, so applying the above argument we see that we need only prove (i) for $\bG_{\ad}$. As $\bG_{\ad}$ is an adjoint semisimple group it is a direct product of simple algebraic groups. Hence, we clearly need only prove (i) in the case where $\bG_{\ad}$ is simple. Let us now assume this to be the case. If $\bG_{\ad}$ is of type $\A_n$ then we assume $\bG = \GL_{n+1}(\mathbb{K})$ and $\varphi : \bG \to \bG_{\ad}$ is the adjoint representation. Otherwise we assume that $\varphi : \bG \to \bG_{\ad}$ is a simply connected covering of $\bG_{\ad}$. By \cite[2.6, 2.7]{premet:2003:nilpotent-orbits-in-good-characteristic} we then have (i) holds for $\bG$ so by the above argument we have (i) holds for $\bG_{\ad}$.

(ii). We assume $\lambda^{\bG} \in \mathcal{D}(\bG)$ is a cocharacter. Now, assume $\pi : \bG \to \widetilde{\bG}$ is a closed embedding of algebraic groups such that $\widetilde{\bG}$ is a connected reductive algebraic group and $\pi$ restricts to an isomorphism between the derived subgroup of $\bG$ and the derived subgroup of $\widetilde{\bG}$. In particular, $\widetilde{\bG} = \pi(\bG)Z(\widetilde{\bG})$ and $\widetilde{\bG}$ is proximate because we assume $\bG$ is proximate. For convenience we will consider $\bG$ as a subgroup of $\widetilde{\bG}$ by identifying it with $\pi(\bG)$. Moreover we will identify $\lie{g}$ with a subalgebra of $\widetilde{\lie{g}} = \Lie(\widetilde{\bG})$. By the remarks in \cref{pa:weight-spaces} we see that $\mathcal{N}(\lie{g}) = \mathcal{N}(\widetilde{\lie{g}})$ and $\lie{g}(\lambda^{\bG},2) = \widetilde{\lie{g}}(\lambda^{\widetilde{\bG}},2)$ where $\lambda^{\widetilde{\bG}} := \pi \circ \lambda^{\bG}$.

From the definition one can easily check that $\bP(\lambda^{\widetilde{\bG}}) = \bP(\lambda^{\bG})Z(\widetilde{\bG})$ and $\lie{p}(\lambda^{\widetilde{\bG}}) = \lie{p}(\lambda^{\bG}) + \Lie(Z(\widetilde{\bG}))$. In particular, for any $e \in \lie{g}(\lambda^{\bG},2)_{\reg} = \widetilde{\lie{g}}(\lambda^{\widetilde{\bG}},2)_{\reg}$ we see that
\begin{equation*}
\lie{c_g}(e) \subseteq \lie{p}(\lambda^{\bG}) \Leftrightarrow \lie{c_g}(e) + \Lie(Z(\widetilde{\bG})) \subseteq \lie{p}(\lambda^{\bG}) + \Lie(Z(\widetilde{\bG})) \Leftrightarrow \lie{c_{\widetilde{g}}}(e) \subseteq \lie{p}(\lambda^{\widetilde{\bG}}),
\end{equation*}
which shows that (a) holds in $\bG$ if and only if (a) holds in $\widetilde{\bG}$. As $\Lie(Z(\widetilde{\bG})) \subseteq \lie{z(\widetilde{g})}$ we have $[e,\lie{p}(\lambda^{\widetilde{\bG}})] = [e,\lie{p}(\lambda^{\bG})]$ so the observations in \cref{pa:weight-spaces} show that (b) holds in $\bG$ if and only if (b) holds in $\widetilde{\bG}$. Finally we clearly have $(\Ad\bP(\lambda^{\widetilde{\bG}}))e = (\Ad\bP(\lambda^{\bG}))e$ hence (c) holds in $\bG$ if and only if (c) holds in $\widetilde{\bG}$. Hence we have shown that (ii) holds in $\bG$ if and only if (ii) holds in $\widetilde{\bG}$.

Applying this argument to the natural closed embedding $\bG_{\der} \hookrightarrow \bG$ we see that we need only prove (ii) for $\bG_{\der}$. Now assume (ii) holds for $\bG_{\simc}$ then we claim (ii) holds for $\bG_{\der}$. Let $\lambda^{\bG_{\der}} \in \mathcal{D}(\bG_{\der})$ be a cocharacter then by the definition of the set $\mathcal{D}(\bG_{\der})$ there exists a cocharacter $\lambda^{\bG_{\simc}} \in \mathcal{D}(\bG_{\simc})$ such that $\lambda^{\bG} = \phi_{\simc} \circ \lambda^{\bG_{\simc}}$. From the definition we see that $\phi_{\simc}(\bP(\lambda^{\bG_{\simc}})) = \bP(\lambda^{\bG_{\der}})$ and as $\mathrm{d}\phi_{\simc} : \lie{g}_{\simc} \to \lie{g}_{\der}$ is an isomorphism it is clear that (c), (b) and the first part of (a) hold in $\bG_{\der}$ if they hold in $\bG_{\simc}$. To see that the second part of (a) holds in $\bG_{\der}$ if it holds in $\bG_{\simc}$ one only needs to note that $\phi_{\simc}(C_{\bG_{\simc}}(e)) = C_{\bG_{\der}}(e)$, where we identify $e$ with an element of $\lie{g}_{\simc}$. This is because the kernel of $\phi_{\simc}$ is contained in the centre $Z(\bG_{\simc})$ which is the kernel of $\Ad$.

We now need only prove (ii) for $\bG_{\simc}$. As $\bG_{\simc}$ is simply connected it is a direct product of simple groups, hence we may clearly assume that $\bG_{\simc}$ is simple. If $\bG_{\simc}$ is of type $\A_n$ then we choose a closed embedding $\pi : \bG_{\simc} \to \widetilde{\bG}$ such that $\widetilde{\bG} = \GL_{n+1}(\mathbb{K})$ and $\pi(\bG_{\simc})$ is the derived subgroup, which clearly exists as we have an isomorphism $\bG_{\simc} \cong \SL_{n+1}(\mathbb{K})$. If $\bG_{\simc}$ is not of type $\A_n$ then we assume $\widetilde{\bG} = \bG_{\simc}$ and $\pi$ is the identity. By the previous argument we need only prove (ii) for $\widetilde{\bG}$. Let $\lambda^{\widetilde{\bG}} \in \mathcal{D}(\widetilde{\bG})$ be a cocharacter. The first two parts (a) and (b) are given by \cite[Theorem 2.3]{premet:2003:nilpotent-orbits-in-good-characteristic}. For part (c) we will use the fact that $C_{\bG}(e)$ is separable because $p$ is a pretty good prime for $\widetilde{\bG}$, see \cref{prop:smooth-centraliser}. In particular, by (a) we have $C_{\widetilde{\bG}}(e) = C_{\bP(\lambda^{\widetilde{\bG}})}(e)$ so
\begin{equation*}
\dim (\Ad \bP(\lambda^{\widetilde{\bG}})e) = \dim \bP(\lambda^{\widetilde{\bG}}) - \dim C_{\bP(\lambda^{\widetilde{\bG}})}(e) = \dim \lie{p}(\lambda^{\widetilde{\bG}}) - \dim \lie{c}_{\widetilde{\lie{g}}}(e) = \dim [e,\lie{p}_{\widetilde{\lie{g}}}(\lambda)].
\end{equation*}
Combining this with (b) gives (c).
\end{proof}

\begin{pa}
Assume now that $x \in \lie{g}$ is any element of the Lie algebra. Following \cite[\S5.1, pg., 60]{slodowy:1980:simple-singularities} we say a locally closed subvariety $\Sigma \subseteq \lie{g}$ is a \emph{transverse slice} to $(\Ad\bG)x$ at $x$ if the following hold:
\begin{enumerate}
	\item $x \in \Sigma$,
	\item the action map $\bG \times \Sigma \to \lie{g}$, defined by $(g,s) \mapsto \Ad g(s)$, is a smooth morphism,
	\item $\dim \Sigma$ is minimal with respect to (i) and (ii).
\end{enumerate}
Note that we have $\dim \Sigma \geqslant \dim C_{\bG}(x)$ with equality if and only if $C_{\bG}(x)$ is separable, c.f., \cite[\S5.1, pg., 61]{slodowy:1980:simple-singularities}. The following corollary of \cref{thm:class-nilpotent-orbits} shows that under some mild restrictions we can always find a transverse slice to a nilpotent orbit. This fact will be used later.
\end{pa}

\begin{cor}\label{cor:existence-transversal-slice}
Assume $\bG$ is proximate and $e \in \mathcal{N}(\lie{g})$ is a nilpotent element such that $C_{\bG}(e)$ is separable. For any cocharacter $\lambda \in \mathcal{D}(\bG)$ satisfying $e \in \lie{g}(\lambda,2)_{\reg}$ there exists a subspace $\lie{s} \subset \lie{g}$ such that the following hold:
\begin{itemize}
	\item $\lie{g} = \lie{s} \oplus T_e(\Ad\bG(e))$,
	\item $\lie{s}$ is stable under the $\mathbb{G}_m$-action $(k,x) \mapsto (\Ad\lambda(k))(x)$,
	\item $\lie{s} \subseteq \bigoplus_{i \leqslant 0} \lie{g}(\lambda,i)$.
\end{itemize}
For any such subspace we have $\Sigma = e+\lie{s}$ is a transverse slice to $\Ad\bG(e)$ at $e$. Moreover let $x \in \Sigma$ be an element whose centraliser $C_{\bG}(x)$ is separable then we have
\begin{equation*}
T_x(\lie{g}) = T_x(\Sigma) + T_x(\Ad\bG(x)).
\end{equation*}
In other words, $\Sigma$ intersects any such $\Ad\bG$-orbit transversally.
\end{cor}

\begin{proof}
Firstly, as $C_{\bG}(e)$ is separable we have $T_e(\Ad\bG(e)) = [\lie{g},e]$ by \cref{lem:separable-centraliser} so $T_e(\Ad\bG(e))$ is stable under the $\mathbb{G}_m$-action. Hence, one can find a graded, i.e., $\mathbb{G}_m$-invariant, complement $\lie{s}$ of $T_e(\Ad\bG(e))$ in $\lie{g}$ by simply picking a complement in each graded piece. By \cref{thm:class-nilpotent-orbits} we certainly have $\lie{u}(\lambda,2) \subseteq [\lie{g},e]$. Now the map $\lie{g}(\lambda,-1) \to \lie{g}(\lambda,1)$ defined by $x \mapsto [x,e]$ is injective because $\lie{g}(\lambda,-1)\cap \lie{c_g}(e) = \{0\}$ by \cref{thm:class-nilpotent-orbits}. However it must also be surjective because $\dim \lie{g}(\lambda,-1) = \dim \lie{g}(\lambda, 1)$, c.f., \cref{pa:dual-space}. This shows that $\lie{u}(\lambda,1) \subseteq T_e(\Ad\bG(e))$ and so $\lie{s} \subseteq \bigoplus_{i \leqslant 0} \lie{g}(\lambda,i)$.

Now the set $\Sigma$ is certainly locally closed and has dimension $\dim C_{\bG}(e)$, so we need only show that the action map $\pi : \bG \times \Sigma \to \lie{g}$ is a smooth morphism. Recall that for a morphism of varieties $f : \bX \to \bY$ we say $x \in \bX$ is a smooth point of $f$ if $x$ and $f(x)$ are non-singular and the differential $d_xf : T_x\bX \to T_{f(x)}\bY$ is surjective. Now clearly $(1,e)$ is a smooth point of $\pi$ because $\lie{g} = \lie{s} \oplus [\lie{g},e]$. In particular, by \cite[Theorem 4.3.6]{springer:2009:linear-algebraic-groups}, we have $\pi$ is dominant and separable and the set $O$ of smooth points of $\pi$ is non-empty and open because $\bG \times \Sigma$ and $\lie{g}$ are irreducible. By \cite[III, \S10, 10.4]{hartshorne:1977:algebraic-geometry} we then have $\pi$ is a smooth morphism if and only if $O = \bG\times \Sigma$ because $\bG \times \Sigma$ and $\lie{g}$ are non-singular. To show this we argue as in \cite[pg.\ 111]{slodowy:1980:simple-singularities}.

We define a linear action $\alpha : \mathbb{G}_m \times \lie{g} \to \lie{g}$ by setting $\alpha(k,x) = \rho(k)x$, where $\rho : \mathbb{G}_m \to \GL(\lie{g})$ is defined by $\rho(k)(x) = k^2(\Ad \lambda(k^{-1}))(x)$ for all $x \in \lie{g}$ and $k \in \mathbb{G}_m$. Note that the $\mathbb{G}_m$-action $\rho$ preserves $\lie{s}$ and as $e \in \lie{g}(\lambda,2)$ we have $\rho(k)(e+y) = e + \rho(k)(y)$ for all $y \in \lie{s}$. In particular, the restriction of $\alpha$ to $\mathbb{G}_m \times \Sigma$ is a contraction to $e$, c.f., \cref{rem:contracting-lin-action}. Now we can define an action of the direct product $\bG \times \mathbb{G}_m$ on $\bG \times \Sigma$ defined by $(h,k) \cdot (g,x) = (hg\lambda(k),\rho(k)(x))$ and on $\lie{g}$ defined by $(h,k)\cdot y = k^2\Ad h(y)$. An easy calculation shows that $\pi$ is equivariant with respect to these actions. Hence we need only show that $O$ meets every $\bG \times \mathbb{G}_m$-orbit of $\bG \times \Sigma$. Assume $n$ is the minimal integer such that $\lie{g}(\lambda,n) \neq \{0\}$ then it is easy to see that every $\bG \times \mathbb{G}_m$-orbit of $\bG\times\Sigma$ contains a set of the form
\begin{equation*}
X_y = \{(1,e+ \sum_{i=n}^0 a_iy_i) \mid a_i \in \mathbb{G}_m\}
\end{equation*}
for some $y_i \in \lie{g}(\lambda,i)$ because $\mathbb{K}$ is algebraically closed and $\lie{s} \subseteq \oplus_{i \leqslant 0} \lie{g}(\lambda,i)$. Now $O \cap \overline{X_y}$ is a non-empty open set of the closure $\overline{X_y}$ because it contains $(1,e)$, c.f., \cref{rem:contracting-lin-action}. However as $\overline{X_y}$ is clearly irreducible, as $X_y$ is, we must have $O \cap X_y \neq \emptyset$ because $X_y$ is open in its closure. Thus we have shown that $\pi$ is smooth.

We now consider the final point. Embedding $\bG$ as a closed subgroup of some $\GL_n(\mathbb{K})$ we may easily compute the differential $\mathrm{d}_{(1,x)}\pi : \lie{g} \oplus T_x(\Sigma) \to T_x(\lie{g})$ using the framework of dual numbers, c.f., \cite[AG, 16.2]{borel:1991:linear-algebraic-groups}. Indeed, one readily checks that we have $\mathrm{d}_{(1,x)}\pi(g,y) = y + [g,x]$. Now, by the above we know the differential is surjective so we must have $T_x(\lie{g}) = T_x(\Sigma) + [\lie{g},x]$; so the statement follows from \cref{lem:separable-centraliser}.
\end{proof}

\section{Springer Isomorphisms and Kawanaka Isomorphisms}
\begin{definition}\label{def:kawanaka-iso}
Given any cocharacter $\lambda \in Y(\bG)$ we say an isomorphism $\psi : \bU(\lambda) \to \lie{u}(\lambda)$ is a \emph{Kawanaka isomorphism} if it commutes with the action of the Frobenius endomorphism and furthermore the following hold:
\begin{enumerate}[label=(K\arabic*)]
	\item $\psi(\bU(\lambda,2)) \subseteq \lie{u}(\lambda,2)$
	\item $\psi(uv) - \psi(u) - \psi(v) \in \lie{u}(\lambda,i+1)$ for any $u,v \in \bU(\lambda,i)$ and $i \in \{1,2\}$,
	\item $\psi([u,v]) - c_i[\psi(u),\psi(v)] \in \lie{u}(\lambda,2i+1)$ for any $u,v \in \bU(\lambda,i)$ and $i \in \{1,2\}$ where $c_i \in \mathbb{G}_m$ is a constant not depending on $u$ or $v$.
\end{enumerate}
Note that $[u,v]$ denotes the commutator $uvu^{-1}v^{-1}$ of $u$ and $v$.
\end{definition}

\begin{pa}
Kawanaka isomorphisms will be the crucial ingredient for the definition of generalised Gelfand--Graev representations. In \cite[\S3]{kawanaka:1986:GGGRs-exceptional} Kawanaka gave a general construction for a Kawanaka isomorphism. However, the construction he gives is not in general $\bG$-equivariant so cannot be obtained as the restriction of a Springer isomorphism by \cite[Remark 10]{mcninch:2005:optimal-sl2-homomorphisms}. In this section we wish to show that there always exists a Springer isomorphism whose restriction to $\bU(\lambda)$ is a Kawanaka isomorphism for all cocharacters $\lambda \in Y(\bG)$.
\end{pa}

\begin{lem}\label{lem:Springer-Kawanaka-iso-exceptional}
Assume $\bG$ is an adjoint simple group of exceptional type then there exists a Springer isomorphism $\phi_{\spr} : \mathcal{U} \to \mathcal{N}$ whose restriction to $\bU(\lambda)$ is a Kawanaka isomorphism for every cocharacter $\lambda \in Y(\bG)$.
\end{lem}

\begin{proof}
As $\bG$ is adjoint we have the adjoint representation $\Ad : \bG \to \GL(V)$, with $V = \lie{g}$, is faithful and according to \cite[I, \S5, 5.3]{springer-steinberg:1970:conjugacy-classes} we have
\begin{equation*}
\lie{gl}(V) = \ad(\lie{g}) \oplus \lie{m},
\end{equation*}
where $\lie{m}$ is an $\Ad\bG$-invariant subspace of $\lie{gl}(V)$ containing $\ID_V$. Let $\pi : \lie{gl}(V) \to \ad(\lie{g})$ be the natural projection map then according to Bardsley and Richardson the composition $\phi_{\spr} = \ad^{-1}\circ\pi \circ \Ad$ is a Springer isomorphism whose differential is the identity, see \cite[9.3.4]{bardsley-richardson:1985:Etale-slices-for-algebraic-transformation-groups}. Note that this makes sense because we have $\GL(V) \subseteq \lie{gl}(V)$.

By conjugating we may and will assume that $\bT_0 \leqslant \bB_0$ are contained in $\bP(\lambda)$. For each root $\alpha \in \Phi$ we choose an isomorphism $x_{\alpha} : \mathbb{G}_a \to \bX_{\alpha}$, where $\bX_{\alpha} \leqslant \bG$ is the root subgroup corresponding to $\alpha$, normalised such that $\mathrm{d}x_{\alpha}(t) = te_{\alpha} \in \lie{g}_{\alpha}$ for all $t \in \mathbb{G}_a$, c.f., \cref{pa:chevalley-basis-general}. We claim that for any $\alpha \in \Phi^+$ and $t \in \mathbb{G}_a$ we have
\begin{equation}\label{eq:Ad-root-subgroup}
\Ad x_{\alpha}(t) = \ID_V + t\ad e_{\alpha} + t^2\frac{(\ad e_{\alpha})^2}{2} + t^3\frac{(\ad e_{\alpha})^3}{6},
\end{equation}
in fact $(\ad e_{\alpha})^3 = 0$ unless $\bG$ is of type $\G_2$. If $\bG$ is of type $\E_n$ then this follows from the general argument given in \cite[10.2.7]{springer:2009:linear-algebraic-groups}. If $\bG$ is of type $\G_2$ or $\F_4$ then one can use the implementation of the adjoint representation in \cite{GAP:2014:version-4.7.5} to check that this holds. Indeed, one can check the order of $\ad e_{\alpha}$ and can check that for $j \in \{2,3\}$ the matrix $(\ad e_{\alpha})^j/j!$ is integer valued, the result then follows as in \cite[\S11.3]{carter:1972:simple-groups-lie-type}.

Now $\Ad g(\ad x) = \ad (\Ad g(x))$ for all $g \in \bG$ and $x \in \lie{g}$ so we have $x \in \lie{g}(\lambda,i)$ if and only if $\Ad \lambda(k)(\ad x) = k^i\ad x$ for all $k \in \mathbb{G}_m$. Let us assume now that $x \in \lie{g}(\lambda,i)$ and $y \in \lie{g}(\lambda,j)$. As the action of $\Ad \lambda(k)$ on $\ad x$ is given simply by conjugation we have
\begin{equation*}
\Ad \lambda(k)(\ad x\ad y) = \Ad \lambda(k)(\ad x)\Ad \lambda(k)(\ad y) = k^{i+j}\ad x\ad y.
\end{equation*}
In particular, this implies that
\begin{equation}\label{eq:weight-space-product}
(\ad^{-1}\circ \pi)(\ad x\ad y) \in \lie{g}(\lambda,i+j).
\end{equation}
This could of course simply be 0.

We now wish to show that the restriction of $\phi_{\spr}$ to $\bU(\lambda)$ satisfies the properties (K1) to (K3). For this we fix a total ordering $\alpha_1,\dots,\alpha_m$ on the set of positive roots $\Phi^+$ then, as a variety, we may identify $\bU(\lambda)$ with the product $\prod_{i=1}^m \bX_{\alpha_i}$. In particular, any element $u \in \bU(\lambda)$ may be written uniquely as
\begin{equation*}
u = x_{\alpha_1}(t_1)\cdots x_{\alpha_m}(t_m)
\end{equation*}
for some $t_i \in \mathbb{G}_a$. By \cref{eq:Ad-root-subgroup} we can write $\Ad u = \Ad (x_{\alpha_1}(t_1)\cdots x_{\alpha_m}(t_m)) = \Ad x_{\alpha_1}(t_1)\cdots \Ad x_{\alpha_m}(t_m)$ as
\begin{equation*}
\Ad u = \prod_{i=1}^m \bigg(\ID_V + t_i\ad e_{\alpha_i} + t_i^2\frac{(\ad e_{\alpha_i})^2}{2} + t^3\frac{(\ad e_{\alpha_i})^3}{6}\bigg).
\end{equation*}
Now assume $u \in \bU(\lambda,2)$ then by definition we must have for each $1 \leqslant i \leqslant m$ with $t_i \neq 0$ that $e_{\alpha_i} \in \mathfrak{u}(\lambda,2)$. Hence, expanding the brackets and applying \cref{eq:weight-space-product} we see that (K1) holds. As $u^{-1} = x_{\alpha_m}(-t_m)\cdots x_{\alpha_1}(-t_1)$ we see that any entirely similar argument shows that (K2) and (K3) hold. We leave it to the reader to fill in the details.
\end{proof}

\begin{prop}\label{thm:Springer-Kawanaka-iso}
Assume $\bG$ is a proximate algebraic group then there exists a Springer isomorphism $\phi_{\spr} : \mathcal{U} \to \mathcal{N}$ whose restriction to $\bU(\lambda)$ is a Kawanaka isomorphism for every cocharacter $\lambda \in Y(\bG)$.
\end{prop}

\begin{proof}
Assume $\bG$ is $\SL(V)$, $\Sp(V)$ or $\SO(V)$ then Kawanaka already observed in \cite[1.2]{kawanaka:1985:GGGRs-and-ennola-duality} that such a Springer isomorphism exists. If $\bG$ is $\SL(V)$ then one simply takes the map $f \mapsto f-\ID_V$, for which the statement is easy to deduce. If $\bG$ is $\Sp(V)$ or $\SO(V)$ then one can use the Cayley map $x \mapsto (f-\ID_V)(f+\ID_V)^{-1}$. In fact, if $\bG$ is $\Sp(V)$ or $\SO(V)$ then one could argue as in the proof of \cref{lem:Springer-Kawanaka-iso-exceptional} to deduce the existence of $\phi_{\spr}$ but instead replacing the adjoint representation with the natural representation.

Let $\bG_{\simc}$ be a simple simply connected algebraic group defined over a field of good characteristic. Then there is a natural surjective separable morphism $\pi : \bG_{\simc} \to \bG$ of algebraic groups where $\bG$ is either $\SL(V)$, $\Sp(V)$, $\SO(V)$ or an adjoint exceptional group. According to \cref{lem:Springer-Kawanaka-iso-exceptional} and the above remarks there exists a Springer isomorphism $\phi_{\spr}$ of $\bG$ whose restriction to each $\bU(\lambda) \leqslant \bG$ is a Kawanaka isomorphism. Arguing as in \cref{lem:springer-morphism} we see that this lifts to a Springer isomorphism on $\bG_{\simc}$ which also has this property. One may now argue as in \cref{lem:springer-morphism} to show that such a Springer isomorphism exists for all proximate algebriac groups. We leave the details to the reader.
\end{proof}

\begin{rem}
It seems likely that the properties defined in \cref{def:kawanaka-iso} hold for most Springer isomorphisms. It would be interesting to find a case free proof of \cref{thm:Springer-Kawanaka-iso}.
\end{rem}

\section{Generalised Gelfand--Graev Representations}\label{sec:GGGRs}
\begin{assumption}
From this point forward we will assume that $\bG$ is proximate and we maintain our assumption that $p$ is a good prime for $\bG$. We will denote by:
\begin{itemize}[itemsep=0pt]
	\item $\Ql$ a fixed algebraic closure of the field of $\ell$-adic numbers where $\ell \neq p$ is a prime,
	\item $G = \bG^F$ the $\mathbb{F}_q$-rational structure determined by $F$,
	\item $e \in \mathcal{N}^F$ a fixed nilpotent element and $\lambda \in \mathcal{D}(\bG)$ a cocharacter such that $e \in \lie{g}(\lambda,2)_{\reg}$, c.f., \cref{thm:class-nilpotent-orbits},
	\item $\phi_{\spr} : \mathcal{U} \to \mathcal{N}$ a fixed Springer isomorphism satisfying the property of \cref{thm:Springer-Kawanaka-iso},
	\item $u \in \mathcal{U}^F$ the unique element satisfying $\phi_{\spr}(u) = e$,
	\item $\lie{t}_0$ the Lie algebra $\Lie(\bT_0)$.
\end{itemize}
\end{assumption}

\begin{definition}[Kawanaka, {\cite[3.1.4]{kawanaka:1982:fourier-transforms}}]
An involutive homomorphism ${}^{\dag} : \lie{g} \to \lie{g}$ is called an \emph{$\mathbb{F}_q$-opposition automorphism} if the following holds:
\begin{enumerate}
	\item $\lie{t}_0^{\dag} = \lie{t}_0$,
	\item $e_{\alpha}^{\dag} \in \mathbb{F}_qe_{-\alpha}$ for all $\alpha \in \Phi$ where $e_{\alpha} \in \lie{g}_{\alpha}$ is as in \cref{pa:chevalley-basis}.
\end{enumerate}
\end{definition}

\begin{lem}\label{lem:Fq-opp-auto}
The map defined by $t^{\dag} = -t$ if $t \in \lie{t}_0$ and $e_{\alpha}^{\dag} = -e_{-\alpha}$ for all $\alpha \in \Phi$ is an $\mathbb{F}_q$-opposition automorphism of $\lie{g}$.
\end{lem}

\begin{proof}
We only have to show that ${}^{\dag}$ is a Lie algebra homomorphism. Recall from \cref{pa:chevalley-basis-general} that $\lie{g}_{\der}$ has a Chevalley basis then by \cite[pg.\ 56]{carter:1972:simple-groups-lie-type} we see that the restriction of ${}^{\dag}$ to $\lie{g}_{\der}$ is a homomorphism. However this easily implies that ${}^{\dag}$ is a homomorphism as $\lie{g} = \lie{g}_{\der} + \lie{z(g)}$ and $\lie{z(g)}^{\dag} = \lie{z(g)}$, c.f., \cref{pa:centralisers}.
\end{proof}

\begin{assumption}
We now assume that ${}^{\dag} : \lie{g} \to \lie{g}$ is a fixed $\mathbb{F}_q$-opposition automorphism, which exists by \cref{lem:Fq-opp-auto}.
\end{assumption}

The following won't be needed until the proof of \cref{lem:mod-GGGR-decomp} but it will be convenient to prove this here.

\begin{prop}\label{prop:same-G-orbit}
The elements $e$ and $-e^{\dag}$ are contained in the same $\bG$-orbit.
\end{prop}

\begin{proof}
For any $\alpha \in \Phi$ and $t \in \mathbb{G}_a$ let us denote by $\overline{x}_{\alpha}(t)$ the element $\Ad x_{\alpha}(t)$ then $\{\overline{x}_{\alpha}(t) \mid \alpha \in \Phi, t \in \mathbb{G}_a\}$ is a generating set for $\Ad\bG$; here $x_{\alpha}(t)$ is defined as in the proof of \cref{lem:Springer-Kawanaka-iso-exceptional}. Furthermore, for each $\alpha \in \Phi$ we denote by $\gamma_{\alpha} \in \mathbb{F}_q^{\times}$ the scalar such that $e_{\alpha}^{\dag} = \gamma_{\alpha}e_{-\alpha}$. As in the proof of \cite[Proposition 12.2.3]{carter:1972:simple-groups-lie-type} let $\theta : \Ad\bG \to \Ad\bG$ be the automorphism defined by $\theta(\overline{x}_{\alpha}(t)) = \overline{x}_{-\alpha}(\gamma_{\alpha}t)$ for all $\alpha \in \Phi$ and $t \in \mathbb{G}_a$ then we have
\begin{equation}\label{eq:dag-auto-commute}
\Ad g(x^{\dag}) = \theta(\Ad g)(x)^{\dag}
\end{equation}
for all $g \in \bG$ and $x \in \lie{g}$.

By \cite[2.10, Lemma]{jantzen:2004:nilpotent-orbits} we see that it is sufficient to show that $e$ and $e^{\dag}$ are in the same $\bG$-orbit. Let us denote by $\dot{w}_0 \in N_{\bG}(\bT_0)$ a representative for the longest element $w_0 \in W_{\bG}(\bT_0)$. The action of $-w_0$ on $\Phi$ induces a permutation $\rho : \Phi \to \Phi$ on the roots which is known to satisfy $\langle \alpha,\lambda  \rangle = \langle \rho(\alpha),\lambda\rangle$ for all $\alpha \in \Phi$. To see this it suffices to observe that the weighted Dynkin diagrams of $\A_n$, $\D_{2n+1}$ and $\E_6$ are invariant under the graph automorphism induced by $w_0$, which is easily checked by inspecting \cite[\S13.1]{carter:1993:finite-groups-of-lie-type}. In particular, it follows that $\Ad \dot{w}_0(e^{\dag}) \in \lie{g}(\lambda,2)$, $\theta(\Ad\bL(\lambda)) = \Ad\bL(\lambda)$ and $\Ad\dot{w}_0$ normalises $\Ad\bL(\lambda)$. Combining this with \cref{eq:dag-auto-commute} we see that
\begin{equation*}
\Ad\bL(\lambda)(\Ad \dot{w}_0(e^{\dag})) = \Ad \dot{w}_0(\Ad\bL(\lambda)(e)^{\dag}).
\end{equation*}
As the orbit $\Ad\bL(\lambda)(e)$ is dense in $\lie{g}(\lambda,2)$ we must have $\Ad\bL(\lambda)(\Ad \dot{w}_0(e^{\dag}))$ is dense in $\lie{g}(\lambda,2)$ so $\Ad \dot{w}_0(e^{\dag}) \in \lie{g}(\lambda,2)_{\reg}$. The statement now follows from \cref{thm:class-nilpotent-orbits}.
\end{proof}

\begin{pa}
We now proceed to define GGGRs following \cite[\S3]{kawanaka:1986:GGGRs-exceptional}. Note that all results in this section are due to Kawanaka. Recall that the definition of GGGRs requires the choice of a $\bG$-invariant symmetric bilinear form $\kappa : \lie{g} \times \lie{g} \to \mathbb{K}$. Here, the $\bG$-invariance means $\kappa(\Ad g(x),\Ad g(y)) = \kappa(x,y)$ for all $x,y \in \lie{g}$ and $g \in \bG$. Such a form can be obtained as a trace form
\begin{equation*}
(x,y) \mapsto \Tr(\mathrm{d}\tau(x)\circ\mathrm{d}\tau(y))
\end{equation*}
where $\tau : \bG \to \GL(V)$ is a finite dimensional rational representation. For convenience we recall that such a form satisfies the property
\begin{equation*}
\kappa(x,[y,z]) = \kappa([x,y],z)
\end{equation*}
for all $x,y,z \in \lie{g}$, which we will use without explicit mention. It is important to note that, even with the assumption that $\bG$ is proximate, we \emph{cannot} always choose $\kappa$ to be non-degenerate, see \cite[Proposition 2.5.10]{letellier:2005:fourier-transforms}. However, we can choose it so that it is not too degenerate.
\end{pa}

\begin{lem}\label{lem:perp-of-root-space}
There exists a form $\kappa$ on $\lie{g}$ defined over $\mathbb{F}_q$ such that, for all $\alpha \in \Phi$, we have
\begin{equation}\label{eq:good-perp}
\lie{g}_{\alpha}^{\perp} = \lie{t}_0 \oplus \bigoplus_{\beta \in \Phi \setminus \{-\alpha\}}\lie{g}_{\beta}
\end{equation}
where $\lie{g}_{\alpha}^{\perp} = \{x \in \lie{g} \mid \kappa(x,y) = 0$ for all $y \in \lie{g}_{\alpha}\}$.
\end{lem}

\begin{proof}
Let $\phi : \bG \to \bG_{\ad}$ be an adjoint quotient of $\bG$. As $\bG_{\ad}$ is an adjoint group it is a direct product of simple adjoint groups $\bG_1 \times \cdots \times \bG_r$ so $\lie{g}_{\ad} = \lie{g}_1 \oplus \cdots \oplus \lie{g}_r$ where $\lie{g}_i = \Lie(\bG_i)$. If $\bG_i$ is not of type $\A_n$ then there exists a non-degenerate $\bG_i$-invariant symmetric bilinear form $\kappa_i$ on $\lie{g}_i$ defined over $\mathbb{F}_q$, see \cite[I, 5.3]{springer-steinberg:1970:conjugacy-classes}. Using \cite[2.5.1(2)]{letellier:2005:fourier-transforms} we can deduce that the appropriate version of \cref{eq:good-perp} holds for $\kappa_i$ by dimension counting.

Now assume $\bG_i$ is of type $\A_n$ then $\bG_i \cong \PGL_{n+1}(\mathbb{K})$ for some $n$ and we set $\widetilde{\bG}_i = \GL_{n+1}(\mathbb{K})$ and $\widetilde{\lie{g}}_i = \lie{gl}_{n+1}(\mathbb{K}) = \Lie(\widetilde{\bG}_i)$. As the natural trace form $\widetilde{\kappa}_i(x,y) = \Tr(xy)$ is non-degenerate on $\widetilde{\lie{g}}_i$ and defined over $\mathbb{F}_q$ we have \cref{eq:good-perp} holds in $\widetilde{\lie{g}}_i$. According to \cite[2.3.1]{letellier:2005:fourier-transforms} we have an isomorphism of Lie algebras $\widetilde{\lie{g}}_i \cong \Lie(Z^{\circ}(\widetilde{\bG}_i)) \oplus \lie{g}_i$. Through this isomorphism we may define a $\bG_i$-invariant symmetric bilinear form $\kappa_i$ on $\lie{g}_i$ by restricting $\widetilde{\kappa}_i$ but this is not necessarily non-degenerate. As the image of each root space under this isomorphism must be contained in $\lie{g}_i$ we see that \cref{eq:good-perp} holds for $\kappa_i$.

We now set $\kappa_{\ad} = \kappa_1 + \cdots + \kappa_r$ and define $\kappa$ by setting $\kappa(x,y) = \kappa_{\ad}(\mathrm{d}\phi(x),\mathrm{d}\phi(y))$. Clearly this is $\bG$-invariant and we see that \cref{eq:good-perp} holds by noticing that $\Ker(\mathrm{d}\phi)$ is contained in $\lie{t}_0$.
\end{proof}

\begin{assumption}
We now assume that $\kappa : \lie{g} \times \lie{g} \to \mathbb{K}$ is a fixed $\bG$-invariant symmetric bilinear form, defined over $\mathbb{F}_q$, satisfying \cref{eq:good-perp}.
\end{assumption}

\begin{pa}\label{pa:dual-space}
If $\lie{h}$ is a Lie algebra over $\mathbb{K}$ then we will denote by $\lie{h}^* = \Hom(\lie{h},\mathbb{K})$ the dual space. Assume now that $i \neq 0$ and $\lie{g}_{\alpha} \subseteq \lie{g}(\lambda,i)$ then clearly we have $\lie{g}(\lambda,i)^{\perp} \subseteq \lie{g}_{\alpha}^{\perp}$. In particular, applying \cref{eq:good-perp} we see that $\lie{g}(\lambda,i)^{\perp} \cap \lie{g}(\lambda,-i) = \{0\}$ because $\lie{g}_{\alpha} \subseteq \lie{g}(\lambda,i)$ if and only if $\lie{g}_{-\alpha} \subseteq \lie{g}(\lambda,-i)$. With this we see that the map $x \mapsto \kappa(x,-)$ gives an identification of $\lie{g}(\lambda,-i) = \lie{g}(\lambda,i)^{\dag}$ with the dual space $\lie{g}(\lambda,i)^*$.
\end{pa}

\begin{lem}[Kawnaka]\label{lem:non-degen}
The skew-symmetric bilinear form $\lie{g}(\lambda,1) \times \lie{g}(\lambda,1) \to \mathbb{K}$ given by $(x,y) \mapsto \kappa(e^{\dag},[x,y])$ is non-degenerate.
\end{lem}

\begin{proof}
Assume $x \in \lie{g}(\lambda,1)$ and $\kappa(e^{\dag},[x,y]) = \kappa([e^{\dag},x],y) = 0$ for all $y \in \lie{g}(\lambda,1)$. As $[e^{\dag},x] \in \lie{g}(\lambda, -1)$ we must have $[e^{\dag},x] = 0$ by \cref{pa:dual-space} but this implies $x=0$ because $x \in \lie{c}_{\lie{u}(\lambda)}(e^{\dag}) = \lie{c_g}(e)^{\dag} \cap \lie{u}(\lambda) = \{0\}$, see \cref{thm:class-nilpotent-orbits}.
\end{proof}

\begin{pa}\label{pa:U-lambda-2-hom}
We now choose once and for all a non-trivial additive character $\chi_p : \mathbb{F}_p^+ \to \Ql^{\times}$ of the finite field $\mathbb{F}_p$ viewed as an additive group. If $r \in \mathbb{Z}_{> 1}$ is an integral power of $p$ then we denote by $\mathbb{F}_r \subseteq \mathbb{K}$ the subfield fixed by $x \mapsto x^r$. The choice of $\chi_p$ gives rise to a character $\chi_r : \mathbb{F}_r^+ \to \Ql^{\times}$ by setting $\chi_r = \chi_p\circ \Tr_{\mathbb{F}_r/\mathbb{F}_p}$ where $\Tr_{\mathbb{F}_r/\mathbb{F}_p}$ is the field trace. As $\phi_{\spr}$ satisfies (K2) we obtain a linear character $\varphi_u : U(\lambda,2) \to \Ql^{\times}$ by setting
\begin{equation*}
\varphi_u(x) = \chi_q(\kappa(e^{\dag},\phi_{\spr}(x))).
\end{equation*}
\end{pa}

\begin{pa}
We could now induce the character $\varphi_u$ to $G$ to obtain a character of $G$ but this turns out not to be the right idea. Instead we construct an intermediary subgroup as follows. The non-degeneracy of the form, c.f., \cref{lem:non-degen}, implies that there exists a Lagrangian subspace $\lie{m} = \lie{m}^{\perp} \subseteq \lie{g}(\lambda,1)$ by which we mean that
\begin{equation}\label{eq:form-equals-zero}
\kappa(e^{\dag},[x,y]) = 0
\end{equation}
for all $x,y \in \lie{m}$. This subspace necessarily has dimension $\dim\lie{g}(\lambda,1)/2$. Note that $\lie{m}$ is not necessarily unique so we must choose some such subspace. We then define $\bU(\lambda,1.5)$ to be the variety
\begin{equation*}
\bU(\lambda,1.5) = \{x \in \bU(\lambda) \mid \phi_{\spr}(x) \in \lie{m} + \lie{u}(\lambda,2)\},
\end{equation*}
which is a closed connected $F$-stable subgroup of $\bU(\lambda)$ containing $\bU(\lambda,2)$. This follows from the properties of a Kawanaka isomorphism and the fact that $\kappa$ is defined over $\mathbb{F}_q$.
\end{pa}

\begin{lem}[{}{Kawanaka, \cite[3.1.9]{kawanaka:1986:GGGRs-exceptional}}]
\mbox{}
\begin{enumerate}
	\item $U(\lambda,1.5)$ is a subgroup of $G$ containing $U(\lambda,2)$ such that $[U(\lambda):U(\lambda,1.5)] = [U(\lambda,1.5):U(\lambda,2)]$.
	\item $\varphi_u$ extends to a linear character $\widetilde{\varphi}_u$ of $U(\lambda,1.5)$.
\end{enumerate}
\end{lem}

\begin{proof}
\mbox{}
(i). The first statement is obvious and the second follows from a theorem of Rosenlicht, see \cite[4.2.4]{geck:2003:intro-to-algebraic-geometry}.

(ii). We first observe that $U(\lambda)/\Ker(\varphi_u)$ is abelian which follows from \cref{pa:dual-space}, (K3), \cref{eq:form-equals-zero} and the definition of $U(\lambda,1.5)$. The result now follows from the following general fact: If $H$ is a finite group and $X \leqslant H$ is a subgroup with linear character $\chi \in \Irr(X)$ such that $H/\Ker(\varphi)$ is abelian then $\chi$ extends to $H$. Indeed, passing to the quotient we may assume that $H$ is abelian and $\chi$ is a faithful linear character of $X$, in particular $X$ must be cyclic. We can then write $H$ as a direct product $\widetilde{X} \times Y$ such that $X \leqslant \widetilde{X}$ and $\widetilde{X}$ is cyclic. The result then follows from \cite[Corollary 11.22]{isaacs:2006:character-theory-of-finite-groups}.
\end{proof}

\begin{definition}\label{def:GGGR}
We call the induced representation $\Gamma_u = \Ind_{U(\lambda,1.5)}^G(\widetilde{\varphi}_u)$ the \emph{Generalised Gelfand--Graev Representation} (GGGR) associated with $u$. We identify $\Gamma_u$ with its character.
\end{definition}

\begin{lem}[{}{Kawanaka, \cite[3.1.12]{kawanaka:1986:GGGRs-exceptional}}]\label{lem:ind-from-u2}
We have $\Gamma_u = q^{-\dim\lie{g}(\lambda,1)/2}\Ind_{U(\lambda,2)}^G(\varphi_u)$. In particular, the construction of $\Gamma_u$ does not depend upon the choice of Lagrangian $\lie{m}$ or extension $\widetilde{\varphi}_u$.
\end{lem}

\begin{proof}
First of all let us note that by \cref{eq:commutator-weight-space-grp} we have $U(\lambda,2)$ contains the derived subgroup of $U(\lambda)$ so both $U(\lambda,2)$ and $U(\lambda,1.5)$ are normal subgroups of $U(\lambda)$. Assume $v \in U(\lambda,1.5)$ then by definition we have
\begin{equation*}
\Ind_{U(\lambda,1.5)}^{U(\lambda)}(\widetilde{\varphi}_u)(v) = q^{-\dim\lie{g}(\lambda,1)/2 - \dim\lie{u}(\lambda,2)}\sum_{g \in U(\lambda)} \widetilde{\varphi}_u(gvg^{-1}).
\end{equation*}
Again by \cref{eq:commutator-weight-space-grp} we have $[g,v] \in \bU(\lambda,2)$ so we may rewrite the sum on the right as
\begin{equation*}
\sum_{g \in U(\lambda)} \widetilde{\varphi}_u(gvg^{-1}) = \sum_{g \in U(\lambda)}\widetilde{\varphi}_u(v)\varphi_u([g,v]) = \widetilde{\varphi}_u(v)\sum_{x \in \lie{u}(\lambda)^F}\zeta_v(x).
\end{equation*}
where $\zeta_v$ is the linear character $x \mapsto \chi_q(\kappa(e^{\dag},c[x,\phi_{\spr}(v)]))$ of the abelian group $\lie{u}(\lambda)^F$; here $c$ is a constant as in (K3). Now the character sum is 0 unless $\zeta_v$ is identically 1, in which case it is simply $q^{\dim \lie{u}(\lambda)}$. However, we have $\zeta_v = 1$ if and only if $v \in U(\lambda,2)$ which shows that
\begin{equation*}
\Ind_{U(\lambda,1.5)}^{U(\lambda)}(\widetilde{\varphi}_u)(v) = \begin{cases}
q^{\dim \lie{g}(\lambda,1)/2}\varphi_u(v) &\text{if }v \in U(\lambda,2),\\
0 &\text{if }v \not\in U(\lambda,2).
\end{cases}
\end{equation*}
Applying the exact same argument as above to $\Ind_{U(\lambda,2)}^{U(\lambda)}(\varphi_u)(v)$ we see that
\begin{equation}
\Ind_{U(\lambda,1.5)}^{U(\lambda)}(\widetilde{\varphi}_u) = q^{-\dim \lie{g}(\lambda,1)/2}\Ind_{U(\lambda,2)}^{U(\lambda)}(\varphi_u),
\end{equation}
from which the result follows immediately.
\end{proof}

\begin{rem}
In \cite{kawanaka:1986:GGGRs-exceptional} GGGRs are defined with no assumption on the algebraic group $\bG$. However, it seems to be necessary to have some assumption on $\bG$ to define GGGRs, see for instance the proof of \cref{lem:non-degen}. Note that assumptions similar to $\bG$ being proximate were made in \cite[1.1.1]{kawanaka:1985:GGGRs-and-ennola-duality} but not in \cite{kawanaka:1986:GGGRs-exceptional}. In any case, we have by \cref{prop:pretty-good-exists} that our assumption on $\bG$ is not restrictive when defining GGGRs.
\end{rem}

\section{Fourier Transforms on the Lie Algebra}
\begin{definition}[{}{Letellier, \cite[5.0.14]{letellier:2005:fourier-transforms}}]\label{def:acceptable-prime}
We say the prime $p$ is \emph{acceptable} for $\bG$ if:
\begin{enumerate}
	\item $p$ is a pretty good prime for $\bG$, c.f.\ \cref{def:pretty-good-prime},
	\item $p$ is a very good prime for any Levi subgroup $\bL$ of $\bG$ supporting a cuspidal pair $(S,\mathscr{E})$, in the sense of \cite[2.4]{lusztig:1984:intersection-cohomology-complexes}, such that $S$ contains a unipotent conjugacy class of $\bL$,
	\item there exists a non-degenerate $\bG$-invariant bilinear form on $\lie{g}$.
\end{enumerate}
\end{definition}

\begin{pa}
This definition of an acceptable prime is slightly different than that given in \cite[5.0.14]{letellier:2005:fourier-transforms}. However we note that if $p$ is acceptable in the sense of \cref{def:acceptable-prime} then it is acceptable in the sense of \cite[5.0.14]{letellier:2005:fourier-transforms}. Indeed, we need only observe that if $p$ is an acceptable prime in the sense of \cref{def:acceptable-prime} then: $p$ is a good prime for $\bG$, see \cite[Lemma 2.12(b)]{herpel:2013:on-the-smoothness}, the quotient $X(\bT_0)/\mathbb{Z}\Phi$ has no $p$-torsion and there exists a Springer isomorphism, see \cref{lem:springer-morphism}. To show that these definitions are actually equivalent it would be sufficient to prove the converse to \cref{lem:springer-morphism}, c.f., \cref{rem:conv-spring-iso}. The following observation concerning acceptable primes will also be needed, see \cite[5.0.16]{letellier:2005:fourier-transforms}.
\end{pa}

\begin{lem}\label{lem:p-acceptable}
The following hold:
\begin{enumerate}
	\item $p$ is very good for $\bG$ $\Rightarrow$ $p$ is acceptable for $\bG$,
	\item if $\bG = \GL_n(\mathbb{K})$ then all primes are acceptable for $\bG$.
\end{enumerate}
\end{lem}

\begin{assumption}
From now on we assume that $p$ is an acceptable prime for $\bG$ and consequently that the bilinear form $\kappa$ is non-degenerate. In particular, $\bG$ is a proximate algebraic group, c.f., \cref{lem:p-tor-springer-good}, and the centraliser $C_{\bG}(x)$ is separable for all $x \in \lie{g}$, c.f., \cref{def:pretty-good-prime}.
\end{assumption}

\begin{pa}
Let us denote by $\boldsymbol\Gamma_u : \lie{g}^F \to \Ql$ the $\Ad G$-invariant function obtained as the extension of $\Gamma_u|_{\mathcal{U}^F}\circ \phi_{\spr}^{-1}$ by 0 on $\lie{g}^F\setminus \mathcal{N}^F$. As $\Gamma_u$ is zero outside $\mathcal{U}^F$ we see that $\boldsymbol\Gamma_u$ contains the same information as that of $\Gamma_u$. The upshot of working with $\boldsymbol\Gamma_u$ is that we have the Fourier transform at our disposal, which is defined as follows. For any function $f : \lie{g}^F \to \Ql$ we define the Fourier transform of $f$ to be the function $\mathcal{F}(f) : \lie{g}^F \to \Ql$ given by
\begin{equation*}
\mathcal{F}(f)(y) = \sum_{x \in \lie{g}^F} \chi_q(\kappa(y,x))f(x),
\end{equation*}
where $\chi_q : \mathbb{F}_q^+ \to \Ql$ is as in \cref{pa:U-lambda-2-hom}. We recall the following property of the Fourier transform, see \cite[3.1.9, 3.1.10]{letellier:2005:fourier-transforms}.
\end{pa}

\begin{lem}\label{lem:fourier-inverse}
The Fourier transform admits an inverse $\mathcal{F}^-$ given by $q^{-\dim\lie{g}}(\mathcal{F}\circ\Psi)$ where $\Psi(f)(x) = \chi(-x)$ for all $x \in \lie{g}^F$ and $f : \lie{g}^F \to \Ql$. In other words, we have $(\mathcal{F}^-\circ\mathcal{F})(f) = (\mathcal{F}\circ\mathcal{F}^-)(f) = f$ for all functions $f : \lie{g}^F \to \Ql$.
\end{lem}

\begin{pa}
Following \cite[Proposition 2.5]{lusztig:1992:a-unipotent-support} we would like to obtain an expression for the value of the Fourier transform $\mathcal{F}(\boldsymbol\Gamma_u)$ at an element of $\lie{g}^F$. The argument used in \cite{lusztig:1992:a-unipotent-support} can be applied verbatim to our situation once we know that \cite[Lemma 2.2]{lusztig:1992:a-unipotent-support} holds. In fact we will prove a stronger statement than that of \cite[Lemma 2.2]{lusztig:1992:a-unipotent-support}. The proof of this stronger statement is due to Gan--Ginzburg who considered the corresponding statement over $\mathbb{C}$, see \cite[Lemma 2.1]{gan-ginzburg:2002:quantization-of-slodowy-slices}. To obtain the desired result we will need the following (general) lemma.
\end{pa}

\begin{lem}\label{lem:iso-equivariant-map}
Assume, for $i \in \{1,2\}$, that $\bX_i$ is an irreducible affine variety with a contracting $\mathbb{G}_m$-action such that the fixed point $x_i \in \bX_i$ is non-singular, c.f., \cref{lem:contracting-action}. Let us denote by $T_i$ the tangent space of $\bX_i$ at the point $x_i$. If $\alpha : \bX_1 \to \bX_2$ is a $\mathbb{G}_m$-equivariant morphism such that the differential $\mathrm{d}_{x_1}\alpha : T_1 \to T_2$ is an isomorphism then $\alpha$ is an isomorphism.
\end{lem}

\begin{proof}
As $\alpha$ induces an isomorphism between the tangent spaces of non-singular points we must have the comorphism $\alpha^* : \mathbb{K}[\bX_2] \to \mathbb{K}[\bX_1]$ is injective, c.f., \cite[1.9.1(ii), 4.3.6(i)]{springer:2009:linear-algebraic-groups}. Hence, we're done if we can show that $\alpha^*$ is surjective. Let $\mathfrak{m}_{x_i} \subset \mathbb{K}[\bX_i]$ be the (maximal) vanishing ideal of $x_i$ which is invariant under the induced action of $\mathbb{G}_m$ on $\mathbb{K}[\bX_i]$ defined in \cref{pa:action-on-aff-alg}. We denote by $\gr\mathbb{K}[\bX_i]$ the associated graded ring with respect to $\mathfrak{m}_{x_i}$, i.e.,
\begin{equation*}
\gr\mathbb{K}[\bX_i] = \bigoplus_{j \geqslant 0} \mathfrak{m}_{x_i}^j/\mathfrak{m}_{x_i}^{j+1}
\end{equation*}
where $\mathfrak{m}_{x_i}^0 = \mathbb{K}[\bX_i]$. This clearly also inherits an action of $\mathbb{G}_m$ as does the tangent space $T_i$ and its affine algebra. We thus obtain decompositions of these algebras into weight spaces
\begin{equation*}
\mathbb{K}[T_i] = \bigoplus_{n\geqslant 0} \mathbb{K}[T_i]_n \qquad\qquad \mathbb{K}[\bX_i] = \bigoplus_{n\geqslant 0} \mathbb{K}[\bX_i]_n \qquad\qquad \gr\mathbb{K}[\bX_i] = \bigoplus_{n\geqslant 0} (\gr\mathbb{K}[\bX_i])_n.
\end{equation*}
Let us assume that the following equality holds
\begin{equation}\label{eq:dim-weight-spaces}
\dim \mathbb{K}[\bX_1]_n = \dim \mathbb{K}[T_1]_n = \dim \mathbb{K}[T_2]_n = \dim \mathbb{K}[\bX_2]_n
\end{equation}
for all $n \in \mathbb{Z}$; these are finite dimensional vector spaces by \cref{lem:contracting-action}. Then, as $\alpha^*$ is injective and $\mathbb{G}_m$-equivariant we must have $\alpha^*$ is surjective. Thus we need only prove \cref{eq:dim-weight-spaces}.

By assumption, the point $x_i$ is non-singular so the tangent space $T_i$ coincides with the tangent cone of $\bX_i$ at $x_i$. This implies that $\mathbb{K}[T_i]$ is isomorphic to the associated graded algebra $\gr\mathbb{K}[\bX_i]$ and this isomorphism is in fact $\mathbb{G}_m$-equivariant, see \cite[III, \S3, \S4]{mumford:1999:red-book-of-varieties}. In particular, we have an equality $\dim \mathbb{K}[T_i]_n = \dim (\gr\mathbb{K}[\bX_i])_n$ for all $n \in \mathbb{Z}$. Now as the maximal ideal $\lie{m}_{x_i}$ is $\mathbb{G}_m$-invariant we have $\lie{m}_{x_i} = \bigoplus_{n \in \mathbb{Z}}\lie{m}_{x_i,n}$ with $\lie{m}_{x_i,n} \subseteq \mathbb{K}[\bX_i]_n$. Moreover we clearly have
\begin{equation*}
(\gr\mathbb{K}[\bX_i])_n = \bigoplus_{j \geqslant 0} \mathfrak{m}_{x_i,n}^j/\mathfrak{m}_{x_i,n}^{j+1} = \gr(\mathbb{K}[\bX_i]_n).
\end{equation*}
As $\mathbb{K}[\bX_i]_n$ is a finite dimensional vector space we have $\dim \gr(\mathbb{K}[\bX_i]_n) = \dim \mathbb{K}[\bX_i]_n$. Putting things together we see that \cref{eq:dim-weight-spaces} holds because $\dim \mathbb{K}[T_1]_n = \dim \mathbb{K}[T_2]_n$ for all $n \in \mathbb{Z}$ by assumption.
\end{proof}

\begin{prop}\label{prop:action-map-iso}
Let $\Sigma = -e^{\dag} + \lie{s}$ where $\lie{s} \subseteq \lie{p}(\lambda)$ is as in \cref{cor:existence-transversal-slice} with respect to $-e^{\dag}$. Then the action map $\alpha : \bU(\lambda,2) \times \Sigma \to -e^{\dag} + \lie{p}(\lambda)$ given by $\alpha(u,x) = \Ad u(x)$ is an isomorphism.
\end{prop}

\begin{proof}
We define a linear $\mathbb{G}_m$-action $\beta : \mathbb{G}_m \times \lie{g} \to \lie{g}$ by setting $\beta(k,x) = \rho(k)x$ where $\rho : \mathbb{G}_m \to \GL(\lie{g})$ is defined by $\rho(k)(x) = k^2(\Ad \lambda(k))(x)$. Note that the $\mathbb{G}_m$-action preserves both $\lie{p}(\lambda)$ and $\lie{s}$, c.f., \cref{cor:existence-transversal-slice}, and as $-e^{\dag} \in \lie{g}(\lambda,-2)$ we have $\rho(k)(-e^{\dag}+x) = -e^{\dag} + \rho(k)(x)$ for all $x \in \lie{p}(\lambda)$. In particular, we have $\beta$ restricts to a $\mathbb{G}_m$-action on $-e^{\dag} + \lie{p}(\lambda)$ and $\Sigma$. By \cref{rem:contracting-lin-action} these actions are contractions to $-e^{\dag}$. We also define a $\mathbb{G}_m$-action on $\bU(\lambda,2) \times \Sigma$ by setting
\begin{equation*}
k\cdot (u,x) = (\lambda(k)u\lambda(k^{-1}),\rho(k)(x))
\end{equation*}
for all $k \in \mathbb{G}_m$. It is clear that $\lim_{k \to 0} k\cdot (u,x) = (1,-e^{\dag})$ for all $(u,x) \in \bU(\lambda,2) \times \Sigma$, c.f., \cref{pa:canonical-parabolic-levi}, so this is also a contraction to $(1,-e^{\dag})$ and the action map $\alpha$ is $\mathbb{G}_m$-equivariant with respect to these actions.

Consider the tangent space $T \cong \lie{u}(\lambda,2) \times \lie{s}$ of $\bU(\lambda,2)\times\Sigma$ at the point $(1,-e^{\dag})$ and the tangent space $S \cong \lie{p}(\lambda)$ of $-e^{\dag} + \lie{p}(\lambda)$ at $-e^{\dag} = \alpha(1,-e^{\dag})$. We claim that the differential $\mathrm{d}_{(1,-e^{\dag})}\alpha : T \to S$ is an isomorphism. First, recall from the proof of \cref{cor:existence-transversal-slice} that the differential is given by $\mathrm{d}_{(1,-e^{\dag})}\alpha(u,x) = x+[u,-e^{\dag}]$ for all $(u,x) \in T$. Therefore to prove the map is surjective it suffices to show that
\begin{equation*}
\lie{p}(\lambda) = \lie{s}\oplus[\lie{u}(\lambda,2),-e^{\dag}].
\end{equation*}
Now the map $\ad_{-e^{\dag}} : \lie{u}(\lambda,2) \to \lie{p}(\lambda)$ is injective because $\lie{c}_{\lie{u}(\lambda,2)}(-e^{\dag}) = \{0\}$ and $[\lie{u}(\lambda,2),-e^{\dag}] \cap \lie{s} = \{0\}$ by \cref{cor:existence-transversal-slice,lem:separable-centraliser}. Hence we have $\lie{s}\oplus[\lie{u}(\lambda,2),-e^{\dag}] \subseteq \lie{p}(\lambda)$. However, as $\dim\lie{s} = \dim\lie{c_g}(-e^{\dag}) = \dim\lie{c_g}(e)$ we have
\begin{equation*}
\dim(\lie{s} \oplus [\lie{u}(\lambda,2),-e^{\dag}]) = \dim\lie{u}(\lambda,2) + \dim\lie{c_g}(e) = \dim\lie{p}(\lambda)
\end{equation*}
so we must have an equality. The same dimension counting argument also proves the map $\mathrm{d}_{(1,-e^{\dag})}\alpha$ is injective. The result now follows from \cref{lem:iso-equivariant-map}.
\end{proof}

\begin{assumption}
From now on the subspace $\lie{s} \subseteq \lie{g}$ considered in \cref{prop:action-map-iso} will be fixed and $\Sigma$ will denote the transverse slice $-e^{\dag} + \lie{s}$.
\end{assumption}

\begin{pa}
With this in hand we now obtain \cite[Proposition 2.5]{lusztig:1992:a-unipotent-support} in exactly the same way. We give the argument here for smoothness of the exposition. From the definitions and \cref{lem:ind-from-u2} we have for any $y \in \lie{g}^F$ that
\begin{align*}
\mathcal{F}({\boldsymbol\Gamma}_u)(y) &= q^{\dim\lie{g}(\lambda,1)/2 - \dim\lie{u}(\lambda)} \sum_{\substack{x \in \lie{g}^F, g \in G\\ \Ad g(x) \in \lie{u}(\lambda)^F}} \chi_q(\kappa(y,x) + \kappa(e^{\dag},\Ad g(x)))\\
&= q^{\dim\lie{g}(\lambda,1)/2 - \dim\lie{u}(\lambda)} \sum_{\substack{g \in G\\ x \in \lie{u}(\lambda)^F}}\zeta_{y,g}(x),
\end{align*}
where $\zeta_{y,g}$ is the linear character $x \mapsto \chi_q(\kappa(e^{\dag} + \Ad g(y),x))$ of $\lie{u}(\lambda)^F$ as an abelian group. As $\zeta_{y,g}$ is an irreducible character of a finite abelian group we have the sum of its values is either 0 or $q^{\dim \lie{u}(\lambda)}$ with the latter happening if and only if $\zeta_{y,g}$ is the trivial character. By \cref{pa:dual-space} and the non-degeneracy of $\kappa$ we have the character $\zeta_{y,g}$ is trivial precisely when $e^{\dag} + \Ad g(y) \in \lie{p}(\lambda)^F$, c.f., \cref{lem:perp-of-root-space}, so we have
\begin{equation*}
\mathcal{F}({\boldsymbol\Gamma}_u)(y) = q^{\dim\lie{g}(\lambda,1)/2}|\{(g,x) \in G \times \lie{p}(\lambda)^F \mid \Ad g(y) = -e^{\dag} + x\}|.
\end{equation*}

Applying \cref{prop:action-map-iso} we see that the element $-e^{\dag} + x$ can be written uniquely as $\Ad h(-e^{\dag} + z)$ for some $h \in U(\lambda,2)$ and $z \in \lie{s}^F$, which are necessarily fixed by $F$ through the uniqueness. In particular, we may rewrite the above as
\begin{equation*}
\mathcal{F}({\boldsymbol\Gamma}_u)(y) = q^{\dim\lie{g}(\lambda,1)/2}|\{(g,h,z) \in G \times U(\lambda,2) \times \lie{s}^F \mid \Ad h^{-1}g(y) = -e^{\dag} + z\}|.
\end{equation*}
Thus, changing the variable $h^{-1}g \to g$ we obtain the following.
\end{pa}

\begin{prop}[{}{Lusztig, \cite[Proposition 2.5]{lusztig:1992:a-unipotent-support}}]\label{prop:fourier-GGGR}
For any $y \in \lie{g}^F$ we have
\begin{equation*}
\mathcal{F}({\boldsymbol\Gamma}_u)(y) = q^{r_u}|\{g \in G \mid \Ad g(y) \in \Sigma\}|
\end{equation*}
where
\begin{equation*}
r_u = \dim \lie{g}(\lambda,1)/2 + \dim\lie{u}(\lambda,2) = (\dim\bG - \dim C_{\bG}(u))/2.
\end{equation*}
\end{prop}

\section{Poincar\'e Duality}
\subsection{Notation}
\begin{pa}
For any algebraic variety $\bX$ over $\mathbb{K}$ we will denote by $\mathscr{D}\bX := \mathscr{D}_c^b\bX$ the bounded derived category of $\Ql$-constructible sheaves on $\bX$ defined as in \cite{beilinson-bernstein-deligne:1982:faisceaux-pervers}. Assume now that the Frobenius endomorphism induces a morphism of varieties on $\bX$ then we say $A \in \mathscr{D}\bX$ is $F$-stable if there exists an isomorphism $\phi : F^*A \to A$ in $\mathscr{D}\bX$. With such a choice of isomorphism we will denote by $\chi_{A,\phi} : \bX^F \to \Ql$ the corresponding characteristic function defined by
\begin{equation*}
\chi_{A,\phi}(x) = \sum_{i \in \mathbb{Z}} (-1)^i\Tr(\phi,\mathscr{H}_x^i(A))
\end{equation*}
for any $x \in \bX^F$. We will denote by $\supp(A)$ the set $\{x \in \bX \mid \mathscr{H}_x^i(A) \neq 0$ for some $i \in \mathbb{Z}\}$ which we call the support of $A$.
\end{pa}

\begin{pa}
Assume now that $\bX$ is equidimensional, i.e., all the irreducible components of $\bX$ have the same dimension, then the full subcategory of $\mathscr{D}\bX$ consisting of the perverse sheaves on $\bX$ will be denoted by $\mathscr{M}\bX$. Moreover, let us assume that $\bX \subseteq \bY$ is a subvariety of $\bY$. We will naturally consider any complex $A \in \mathscr{D}\bX$ as a complex on $\bY$ through extension by 0 and we will sometimes do this without explicit mention. Finally assume $\bX$ is a smooth open dense subset of $\overline{\bX}$ and that $\mathscr{L}$ is a local system on $\bX$. We will denote by $\IC(\overline{\bX},\mathscr{L})[\dim\bX] \in \mathscr{M}\overline{\bX}$ the corresponding intersection cohomology complex determined by $\mathscr{L}$.
\end{pa}

\begin{pa}\label{pa:cuspidal-data}
For any connected reductive algebraic group $\bH$ we denote by $\mathcal{V}_{\bH} = \mathcal{V}_{\bH}^{\nil}$ the set of all pairs $\iota = (\mathcal{O}_{\iota},\mathscr{E}_{\iota})$ consisting of a nilpotent $\bH$-orbit $\mathcal{O}_{\iota} \subseteq \Lie(\bH)$ and an irreducible $\bL$-equivariant local system $\mathscr{E}_{\iota}$ on $\mathcal{O}_{\iota}$ taken up to isomorphism. We will write $\mathcal{V}_{\bH}^0 \subseteq \mathcal{V}_{\bH}$ for the subset consisting of all pairs $\iota$ such that $\mathscr{E}_{\iota}$ is a cuspidal local system in the sense of \cite[Definition 2.4]{lusztig:1984:intersection-cohomology-complexes}, see also \cite[5.1.59]{letellier:2005:fourier-transforms}. We call the elements of $\mathcal{V}_{\bH}^0$ cuspidal pairs.
\end{pa}

\subsection{Induction Diagram}
\begin{pa}\label{pa:induction-diagram}
Let $\bL \leqslant \bG$ be the Levi complement of a parabolic subgroup $\bP \leqslant \bG$ and let $\lie{l} = \Lie(\bL) \subseteq \lie{p} = \Lie(\bP)$ be the corresponding Lie algebras. If $\lie{z(l)}$ is the centre of the Lie algebra $\lie{l}$ then we define
\begin{equation*}
\lie{z(l)}_{\reg} = \{t \in \lie{z(l)} \mid \lie{c_g}(t) = \lie{l}\},
\end{equation*}
which is an open subset of $\lie{z(l)}$. Note that this is non-empty because of our standing assumption that $p$ is an acceptable prime, c.f., \cite[2.6.13(i)]{letellier:2005:fourier-transforms}. We will assume that $\iota_0 = (\mathcal{O}_0,\mathscr{E}_0) \in \mathcal{V}_{\bL}^0$ is a cuspidal pair. Setting $\Pi = \mathcal{O}_0 + \lie{z(l)}$ and $\Pi_{\reg} = \mathcal{O}_0 + \lie{z(l)}_{\reg}$ we have, as in \cite[3.1]{lusztig:1992:a-unipotent-support}, the following induction diagram
\begin{center}
\begin{tikzpicture}
\matrix (m) [matrix of math nodes, row sep=4em, column sep=3em, text height=1.5ex, text depth=0.25ex]
{\mathcal{O}_0 & \hat{Y} & \widetilde{Y} & Y\\};

\path [>=stealth,->]	(m-1-2) edge node[above]{$\alpha$} (m-1-1)
		(m-1-2) edge node[above]{$\beta$} (m-1-3)
		(m-1-3) edge node[above]{$\gamma$} (m-1-4);
\end{tikzpicture}
\end{center}
where
\begin{gather*}
\begin{aligned}
Y &= \bigcup_{g \in \bG} \Ad g(\Pi_{\reg})\\
\hat{Y} &= \{(x,g) \in \lie{g} \times \bG \mid \Ad g^{-1}(x) \in \Pi_{\reg}\}\\
\tilde{Y} &= \{(x,g\bL) \in \lie{g} \times (\bG/\bL) \mid \Ad g^{-1}(x) \in \Pi_{\reg}\}
\end{aligned}\\
\begin{aligned}
\alpha(x,g) &= \pi_0(\Ad g^{-1}(x)) \quad&\quad \beta(x,g) &= (x,g\bL) \quad&\quad \gamma(x,g\bL) &= x
\end{aligned}
\end{gather*}
and $\pi_0 : \Pi \to \mathcal{O}_0$ is the natural projection. As discussed in \cite[3.1]{lusztig:1992:a-unipotent-support} and \cite[pg.\ 75]{letellier:2005:fourier-transforms} from this data one constructs a local system $\mathscr{L}$ on $Y$ and a corresponding intersection cohomology complex $K_0 = \IC(\overline{Y},\mathscr{L})[\dim Y]$, which is denoted by $\ind_{\Sigma}^{\bG}(\mathscr{E}_0)$ in \cite{letellier:2005:fourier-transforms}. Let us also note that we have $\dim Y = \dim \bG/\bL + \dim \Pi$, c.f., \cite[5.1.28]{letellier:2005:fourier-transforms}.
\end{pa}

\begin{assumption}
From now on we assume that the parabolic $\bP$, Levi subgroup $\bL$ and the cuspidal pair $\iota_0 = (\mathcal{O}_0,\mathscr{E}_0) \in \mathcal{V}_{\bL}^0$ are fixed. Recall also that $\Sigma$ is a fixed transverse slice $-e^{\dag} + \mathfrak{s}$.
\end{assumption}

\begin{lem}\label{lem:fibres-pure-dim}
The fibres of the smooth morphism $\pi : \bG \times \Sigma \to \lie{g}$ given by $\pi(g,x) = \Ad g(x)$ are of pure dimension equal to $\dim\lie{c_g}(e)$.
\end{lem}

\begin{proof}
Firstly let us note that $\bG$, $\lie{g}$ and $\Sigma$ are irreducible varieties then applying \cite[III, \S9, 9.6]{hartshorne:1977:algebraic-geometry} we have, for any $y \in \lie{g}$, that every irreducible component of the fibre $\pi^{-1}(y)$ has dimension $\dim \Sigma = \dim \lie{c_g}(e)$ which gives us the first part.
\end{proof}

\begin{pa}
Let $\lie{u_p}$ be the Lie algebra of the unipotent radical of $\bP$. Recall from \cite{lusztig-spaltenstein:1979:induced-unipotent-classes} that there exists a unique nilpotent orbit $\mathcal{O} \in \mathcal{N}/\bG$ such that the intersection $\mathcal{O} \cap (\mathcal{O}_0 + \lie{u_p})$ is open and dense in $\mathcal{O}_0 + \lie{u_p}$, see also \cite[7.1.1]{collingwood-mcgovern:1993:nilpotent-orbits}. We denote the orbit $\mathcal{O}$ by $\Ind_{\bL \subseteq \bP}^{\bG}(\mathcal{O}_0)$ and call it the induced nilpotent orbit.
\end{pa}

\begin{prop}[Lusztig]\label{lem:bumper-facts}
The intersection $\overline{Y} \cap \mathcal{N}$ is the closure of the induced orbit $\Ind_{\bL\subseteq\bP}^{\bG}(\mathcal{O}_0)$. In particular, it is irreducible of dimension $\dim\bG/\bL + \dim \mathcal{O}_0$.
\end{prop}

\begin{proof}
For this we argue as in the proof of \cite[Proposition 7.2]{lusztig:1984:intersection-cohomology-complexes} using the analogous statements for the Lie algebra proven in \cite{letellier:2005:fourier-transforms}. Consider the variety
\begin{equation*}
X = \{(x,g\bP) \in \lie{g} \times (\bG/\bP) \mid \Ad g^{-1}(x) \in \overline{\Pi} + \lie{u_p}\}.
\end{equation*}
If $\phi : X \to \lie{g}$ is the morphism defined by $\phi(x,g\bP) = x$ then we have $\phi(X) = \overline{Y}$, see \cite[5.1.30(ii)]{letellier:2005:fourier-transforms}. Therefore we have $\phi(X_{\mathcal{N}}) = \overline{Y}\cap \mathcal{N}$ where
\begin{equation*}
X_{\mathcal{N}} = \{(x,g\bP) \in \lie{g} \times (\bG/\bP) \mid \Ad g^{-1}(x) \in \overline{\mathcal{O}_0} + \lie{u_p}\}
\end{equation*}
because $\overline{\Pi} \cap \mathcal{N} = \overline{\mathcal{O}_0}$. The argument in \cite[5.1.10]{letellier:2005:fourier-transforms} shows that $X_{\mathcal{N}}$ is irreducible, hence so is $\overline{Y}\cap\mathcal{N}$. As $\overline{Y}\cap\mathcal{N}$ is invariant under the $\Ad\bG$-action we see it is a union $\mathcal{O}_1\sqcup\cdots\sqcup\mathcal{O}_k$ of distinct nilpotent orbits $\mathcal{O}_i \in \mathcal{N}/\bG$. Moreover as $\overline{Y}\cap\mathcal{N}$ is closed we have $\overline{Y}\cap\mathcal{N} = \overline{\mathcal{O}_1}\cup\cdots\cup\overline{\mathcal{O}_k}$ but as $\overline{Y}\cap\mathcal{N}$ and each orbit closure is irreducible we must have $\overline{Y}\cap\mathcal{N} = \overline{\mathcal{O}_i}$ for a unique $1 \leqslant i \leqslant k$. The exact same proof as \cite[7.3(a)]{lusztig:1984:intersection-cohomology-complexes} shows that $\mathcal{O}_i$ is the closure of the induced orbit $\Ind_{\bL \subseteq \bP}^{\bG}(\mathcal{O}_0)$. We thus have $\dim \overline{Y}\cap\mathcal{N} = \dim \Ind_{\bL \subseteq \bP}^{\bG}(\mathcal{O}_0)$ and this latter dimension is easily seen to be $\dim\bG/\bL + \dim \mathcal{O}_0$, c.f., \cite[7.1.4(i)]{collingwood-mcgovern:1993:nilpotent-orbits}.
\end{proof}

\begin{lem}
Assume $-e^{\dag} \in \overline{Y}$ then the following hold:
\begin{enumerate}
	\item The restriction $\sigma : \bG \times (\overline{Y}\cap\mathcal{N} \cap \Sigma) \to \overline{Y}\cap\mathcal{N}$ of the action map $\pi : \bG \times \Sigma \to \lie{g}$ is a smooth morphism with all fibres of pure dimension equal to $\dim\lie{c_g}(e)$.
	\item $\dim (\overline{Y}\cap\mathcal{N} \cap \Sigma) = -\dim\bL + \dim\mathcal{O}_0 + \dim \lie{c_g}(e)$.
\end{enumerate}
\end{lem}

\begin{proof}
(i). We have $\bG \times (\overline{Y}\cap\mathcal{N}\cap\Sigma)$ is the preimage of $\overline{Y}\cap\mathcal{N}$ under $\pi$. As $\sigma$ is thus obtained by base change with respect to the closed embedding $\overline{Y}\cap\mathcal{N} \to \lie{g}$ we have the result follows from \cite[III, 10.1(b)]{hartshorne:1977:algebraic-geometry} and \cref{lem:fibres-pure-dim}.

(ii). As $\sigma$ is smooth we may again apply \cite[III, \S9, 9.6]{hartshorne:1977:algebraic-geometry} to deduce that
\begin{equation*}
\dim(\overline{Y}\cap\mathcal{N}\cap \Sigma) = \dim \overline{Y} \cap \mathcal{N} + \dim \lie{c_g}(e) - \dim \bG.
\end{equation*}
The statement now follows from \cref{lem:bumper-facts}.
\end{proof}

\begin{pa}
According to \cite[6.6.1]{lusztig:1984:intersection-cohomology-complexes} we have the restriction $K_0[-\dim\lie{z(l)}]|_{\overline{Y}\cap\mathcal{N}}$ is a perverse sheaf; note that $\dim \overline{Y}\cap\mathcal{N} = \dim Y - \dim\lie{z(l)}$. Let us take on the assumption that $-e^{\dag} \in \overline{Y}$ and let $\sigma$ be as in \cref{lem:bumper-facts} then according to \cite[4.2.4]{beilinson-bernstein-deligne:1982:faisceaux-pervers} we have
\begin{equation}\label{eq:pullback-perverse}
\sigma^*(K_0[d]|_{\overline{Y}\cap\mathcal{N}}) \in \mathscr{D}(\overline{Y}\cap\mathcal{N}\cap\Sigma)
\end{equation}
is also a perverse sheaf where $d = \dim\lie{c_g}(e)-\dim\lie{z(l)}$. Now $\bG$ acts on $Y$, via $\Ad$, on $\hat{Y}$ and $\widetilde{Y}$, via $\Ad$ on the first coordinate and left translation on the second, and finally on $\Sigma$ trivially. As discussed in \cite[pg.\ 75]{letellier:2005:fourier-transforms} we see that, with respect to these actions, $\mathscr{L}$ is a $\bG$-equivariant local system and $K_0$ is a $\bG$-equivariant perverse sheaf where $\mathscr{L}$ and $K_0$ are as in \cref{pa:induction-diagram}. According to \cite[3.3]{lusztig:1992:a-unipotent-support} as the pullback in \cref{eq:pullback-perverse} is a perverse sheaf we have $\widetilde{K}_0[d_0] \in \mathscr{D}(\overline{Y}\cap\mathcal{N}\cap\Sigma)$ is also a perverse sheaf where $\widetilde{K}_0 \in \mathscr{D}(\overline{Y}\cap\mathcal{N}\cap\Sigma)$ is the restriction of $K_0$ to $\{1\} \times (\overline{Y}\cap\mathcal{N}\cap\Sigma)$ and $d_0 = \dim\lie{c_g}(e) - \dim \lie{z(l)} - \dim\bG$. Furthermore, if $\widetilde{K}_0^{\vee}$ is the restriction of $K_0^{\vee} =\IC(\overline{Y},\mathscr{L}^{\vee})[\dim Y]$, where $\mathscr{L}^{\vee}$ is the local system dual to $\mathscr{L}$, then we have $\widetilde{K}_0^{\vee}[d_0]$ is also a perverse sheaf which is naturally the Verdier dual of $\widetilde{K}_0[d_0]$. We then have a canonical non-singular pairing (Poincar\'e duality)
\begin{equation}\label{eq:poincare-duality}
\bH^j(\overline{Y}\cap\mathcal{N}\cap\Sigma,\widetilde{K}_0) \otimes \bH_c^{2d-j}(\overline{Y}\cap\mathcal{N}\cap\Sigma,\widetilde{K}_0^{\vee}) \to \Ql(-d_0).
\end{equation}
between the corresponding hypercohomology groups.
\end{pa}

\begin{pa}
If $\rho$ is the contracting $\mathbb{G}_m$-action on $\lie{g}$ defined in the proof of \cref{prop:action-map-iso} then we let $\mathbb{G}_m$ act on $Y$, via $\rho(k)$, on $\hat{Y}$ and $\widetilde{Y}$, via $\rho(k)$ on the first coordinate and $\Ad\lambda(k^{-1})$ on the second, and finally on $\mathcal{O}_0$ via $k\cdot x = k^{-2}x$. By \cite[2.10, Lemma]{jantzen:2004:nilpotent-orbits} and \cref{cor:existence-transversal-slice} the action of $\mathbb{G}_m$ by $\rho$ leaves invariant the space $\overline{Y}\cap\mathcal{N}\cap\Sigma$. As in \cite[3.4]{lusztig:1992:a-unipotent-support} we may use this contracting action to deduce that the canonical homomorphism
\begin{equation}
\bH^j(\overline{Y}\cap\mathcal{N}\cap\Sigma,\widetilde{K}_0) \to \mathscr{H}_{-e^{\dag}}^j(\widetilde{K}_0)
\end{equation}
is an isomorphism. In particular, the pairing of \cref{eq:poincare-duality} becomes a pairing
\begin{equation}\label{eq:poincare-duality-II}
\mathscr{H}_{-e^{\dag}}^j(\widetilde{K}_0) \otimes \bH_c^{2d-j}(\overline{Y}\cap\mathcal{N}\cap\Sigma,\widetilde{K}_0^{\vee}) \to \Ql(-d_0).
\end{equation}
\end{pa}

\section{Blocks}
\begin{pa}
Let us denote by $\mathcal{A}$ the endomorphism algebra $\End_{\mathscr{D}\overline{Y}}(K_0)$ of the complex $K_0$; note this is naturally isomorphic to the endomorphism algebra $\End_{\mathscr{D}\mathcal{O}_0}(\mathscr{E}_0)$. The following is a combination of \cite[Theorem 9.2(a)]{lusztig:1984:intersection-cohomology-complexes} and \cite[2.4]{lusztig:1986:on-the-character-values}, see also \cite[Proposition 5.3.6]{letellier:2005:fourier-transforms}.
\end{pa}

\begin{prop}\label{prop:algebra-basis}
There exists a set of basis elements $\{\theta_w \mid w \in W_{\bG}(\bL)\}$ for $\mathcal{A}$ such that the $\Ql$-linear extension of the map $w \mapsto \theta_w$ defines a $\Ql$-algebra isomorphism $\Ql W_{\bG}(\bL) \cong \mathcal{A}$.
\end{prop}

\begin{pa}\label{pa:gen-spring-corr}
Combining \cref{prop:algebra-basis} with \cite[4.2(b)]{lusztig:1992:a-unipotent-support}, see also \cite[4.3]{taylor:2014:evaluating-characteristic-functions}, we have a canonical isomorphism
\begin{equation}\label{eq:decomp-of-induced}
K_0 \cong \bigoplus_{E \in \Irr(W_{\bG}(\bL))} (E \otimes K_E),
\end{equation}
where $K_E = \Hom_{\mathcal{A}}(E,K_0) \in \mathscr{M}\overline{Y}$ and we identify $\Irr(\mathcal{A})$ and $\Irr(W_{\bG}(\bL))$ through the isomorphism in \cref{prop:algebra-basis}. Again from \cite[4.2]{lusztig:1992:a-unipotent-support}, see also \cite[Theorem 6.5]{lusztig:1984:intersection-cohomology-complexes}, there exists a unique pair $\iota \in \mathcal{V}_{\bG}$ such that
\begin{equation}\label{eq:K-restriction}
K_E|_{\overline{Y}\cap\mathcal{N}} \cong \IC(\overline{\mathcal{O}}_{\iota},\mathscr{E}_{\iota})[\dim\mathcal{O}_{\iota} + \dim\lie{z(l)}].
\end{equation}
The map $E \to K_E \to \iota$ then gives us an injective map $\Irr(W_{\bG}(\bL)) \to \mathcal{V}_{\bG}$. We will denote the perverse sheaf $K_E$ by $K_{\iota}$.
\end{pa}

\begin{pa}
Let us denote by $\widetilde{\mathcal{W}}_{\bG}$ the set of all pairs $(\bM,\upsilon)$ consisting of a Levi complement $\bM \leqslant \bG$ of a parabolic subgroup of $\bG$ and a cuspidal pair $\upsilon \in \mathcal{V}_{\bM}^0$. We have $\bG$ acts naturally on $\widetilde{\mathcal{W}}_{\bG}$ by conjugation and we denote by $[\bM,\upsilon]$ the orbit containing $(\bM,\upsilon)$. We also denote by $\mathcal{W}_{\bG}$ the set of all such orbits. Using the map in \cref{pa:gen-spring-corr} we can associate to every pair $(\bM,\upsilon) \in \widetilde{\mathcal{W}}_{\bG}$ a subset $\mathscr{I}[\bM,\upsilon] \subseteq \mathcal{V}_{\bG}$ which depends only on the orbit $[\bM,\upsilon]$. We then have a disjoint union
\begin{equation*}
\mathcal{V}_{\bG} = \bigsqcup_{[\bM,\upsilon] \in \mathcal{M}_{\bG}} \mathscr{I}[\bM,\upsilon]
\end{equation*}
and we call $\mathscr{I}[\bM,\upsilon]$ a \emph{block} of $\mathcal{V}_{\bG}$. If $\upsilon \in \mathcal{V}_{\bM}^0$ is cuspidal then the relative Weyl group $W_{\bG}(\bM) = N_{\bG}(\bM)/\bM$ is a Coxeter group and \cref{pa:gen-spring-corr} yields a bijection
\begin{equation}\label{eq:gen-spring-cor}
\mathscr{I}[\bM,\upsilon] \longleftrightarrow \Irr(W_{\bG}(\bM))
\end{equation}
for all $[\bM,\upsilon] \in \mathcal{W}_{\bG}$, which we denote by $\iota \mapsto E_{\iota}$.
\end{pa}

\begin{pa}
The Frobenius endomorphism $F$ acts naturally on the set of pairs $\mathcal{V}_{\bG}$ via
\begin{equation*}
\iota \mapsto F^{-1}(\iota) = (F^{-1}(\mathcal{O}_{\iota}),F^*\mathscr{E}_{\iota}).
\end{equation*}
We say $\iota$ is $F$-stable if $F(\mathcal{O}_{\iota}) = \mathcal{O}_{\iota}$ and there exists an isomorphism $F^*\mathscr{E}_{\iota} \to \mathscr{E}_{\iota}$ and we denote by $\mathcal{V}_{\bG}^F \subseteq \mathcal{V}_{\bG}$ the subset of all $F$-stable pairs. The Frobenius acts also on the set $\widetilde{\mathcal{W}}_{\bG}$, hence also on $\mathcal{W}_{\bG}$, via $(\bM,\upsilon) \mapsto (F^{-1}(\bM),F^{-1}(\upsilon))$. We say the pair $(\bM,\upsilon)$, resp., the orbit $[\bM,\upsilon]$, is $F$-stable if $F^{-1}(\bM) = \bM$ and $F^{-1}(\upsilon) = \upsilon$, resp., $[\bM,\upsilon] = [F^{-1}(\bM),F^{-1}(\upsilon)]$, and we denote by $\widetilde{\mathcal{W}}_{\bG}^F$, resp., $\mathcal{W}_{\bG}^F$, the subset of all $F$-stable pairs, resp., orbits. Note that for any orbit $[\bM,\upsilon] \in \mathcal{W}_{\bG}^F$ we have $[\bM,\upsilon] \cap \widetilde{\mathcal{W}}_{\bG}^F \neq \emptyset$ and the map $\mathcal{V}_{\bG} \to \mathcal{W}_{\bG}$ is compatible with the corresponding actions of $F$ so that we have an induced map $\mathcal{V}_{\bG}^F \to \mathcal{W}_{\bG}^F$. Furthermore if $(\bM,\upsilon) \in \mathcal{W}_{\bG}^F$ then we have the bijection in \cref{eq:gen-spring-cor} is also compatible with $F$, namely it restricts to a bijection
\begin{equation}\label{eq:gen-spring-cor-F-fixed}
\mathscr{I}[\bM,\upsilon]^F \longleftrightarrow \Irr(W_{\bG}(\bM))^F.
\end{equation}
\end{pa}

\section{Deligne--Fourier Transform}
\begin{pa}
We define the Deligne--Fourier transform of a perverse sheaf as follows, see \cite[]{lusztig:1992:a-unipotent-support}. Let $\mathbb{G}_a$ act transitively on the affine line $\mathbb{A}^1 = \mathbb{K}$ via $a\cdot t = t + a - a^p$. Now the character $\chi_p : \mathbb{F}_p^+ \to \Ql$ fixed in \cref{pa:U-lambda-2-hom} gives rise to a $\mathbb{G}_a$-equivariant local system of rank 1 on $\mathbb{A}^1$ which we denote by $\mathscr{L}_{\chi_p}$. More precisely this is the Artin--Schreier sheaf associated to $\chi_p$, defined as in \cite[1.2.1]{laumon:1987:transformation-de-fourier} or \cite[5.1.57]{letellier:2005:fourier-transforms}. Its inverse image under the non-degenerate form $\kappa$ is then a local system of rank 1 on $\lie{g} \times \lie{g}$. With this we define the Fourier transform $\mathcal{F}(A)$ of any complex $A \in \mathscr{D}(\lie{g})$ by setting
\begin{equation}
\mathcal{F}(A) = (\pr_2)_!(\pr_1^*(A) \otimes \kappa^*\mathscr{L}_{\chi_p})
\end{equation}
where $\pr_i : \lie{g} \times \lie{g} \to \lie{g}$ is the projection onto the $i$th factor. Note that the Deligne--Fourier transform of complexes and the Fourier transform of functions are related in the following way, see \cite[5.2.3]{letellier:2005:fourier-transforms}.
\end{pa}

\begin{lem}\label{lem:deligne-fourier-fourier}
Assume $A \in \mathscr{D}(\lie{g})$ is an $F$-stable complex and $\phi : F^*A \to A$ is an isomorphism then we have
\begin{equation*}
\chi_{\mathcal{F}(A),\mathcal{F}(\phi)} = \mathcal{F}(\chi_{A,\phi}),
\end{equation*}
where $\mathcal{F}(\phi) : \mathcal{F}(F^*A) \to \mathcal{F}(A)$ is the induced isomorphism.
\end{lem}

\begin{rem}\label{rem:diff-defs}
We will frequently cite results from the work of Letellier \cite{letellier:2005:fourier-transforms} throughout this article. We point out here that our definitions of the Deligne--Fourier transform and Fourier transform do not coincide with those of \cite{letellier:2005:fourier-transforms} but that they differ only up to shifts and a scalar. The correction is easily seen by comparing \cref{lem:deligne-fourier-fourier} and \cite[5.2.3]{letellier:2005:fourier-transforms}.
\end{rem}

\begin{pa}\label{pa:G-induced-pair}
Let $A_0 \in \mathscr{D}(\lie{g})$ be the complex obtained as the extension by 0 of the complex $K_0[-\dim\lie{z(l)}]|_{\overline{Y}\cap\mathcal{N}}$. According to the proof of \cite[Proposition 6.2.9]{letellier:2005:fourier-transforms}, see also \cite[\S8]{lusztig:1987:fourier-transforms-on-a-semisimple-lie-algebra}, we have an isomorphism
\begin{equation*}
\mathcal{F}(K_0) \to A_0.
\end{equation*}
Comparing \cref{eq:decomp-of-induced,eq:K-restriction} we see that for each $\iota \in \mathscr{I}[\bL,\iota_0]$ there exists a unique pair $\hat{\iota} \in \mathscr{I}[\bL,\iota_0]$ such that $\mathcal{F}(K_{\iota}) \cong \IC(\overline{\mathcal{O}_{\hat{\iota}}},\mathscr{E}_{\hat{\iota}})$ up to some shift. In particular the map $\hat{\phantom{x}} : \mathscr{I}[\bL,\iota_0] \to \mathscr{I}[\bL,\iota_0]$ given by $\iota \mapsto \hat{\iota}$ is a bijection. We will need the following property of this bijection and its corollary.
\end{pa}

\begin{lem}\label{lem:bijection-not-id}
Assume $\mathscr{I}[\bL,\iota_0] \subseteq \mathcal{V}_{\bG}$ is a block containing precisely two elements, i.e., $\bL$ is the Levi complement of a maximal parabolic subgroup of $\bG$, then the bijection $\hat{\phantom{x}}$ is not the identity.
\end{lem}

\begin{cor}[{}{Lusztig, \cite[5.5]{lusztig:1992:a-unipotent-support}}]\label{lem:bijection-tensor-sign}
Under the bijection in \cref{eq:gen-spring-cor} we have the map $\iota \mapsto \hat{\iota}$ corresponds to the natural bijection
\begin{equation*}
E \mapsto E\otimes\sgn
\end{equation*}
where $\sgn \in \Irr(W_{\bG}(\bM))$ is the sign representation.
\end{cor}

\begin{pa}
\Cref{lem:bijection-not-id} can be proved as a consequence of the validity of \cref{thm:main-theorem-nilpotent} for a split Frobenius endomorphism $F$ just as in \cite[7.7]{lusztig:1992:a-unipotent-support}. Note that the statement of \cref{lem:bijection-not-id} does not depend upon the $\mathbb{F}_q$-structure, so we may use any Frobenius to prove it. We will not use \cref{lem:bijection-not-id} in any step towards the proof of \cref{thm:main-theorem-nilpotent} for split Frobenius endomorphisms. However, we will use the validity of \cref{lem:bijection-not-id}, more specifically \cref{lem:bijection-tensor-sign}, to prove \cref{thm:main-theorem-nilpotent} for twisted Frobenius endomorphisms. This is a fine needle to thread but the reader is free to check that the logic is sound.
\end{pa}

\begin{pa}
Reading carefully the rest of \cite[\S5]{lusztig:1992:a-unipotent-support} one sees that the entire discussion holds unchanged. One need only note that the results cited in \cite{lusztig:1987:fourier-transforms-on-a-semisimple-lie-algebra} are proved in \cite{letellier:2005:fourier-transforms} under our weaker assumption that $p$ is an acceptable prime for $\bG$, see in particular \cite[5.2.9, 5.2.10]{letellier:2005:fourier-transforms}.
\end{pa}

\section{\texorpdfstring{$\mathbb{F}_q$-Structures}{Fq-Structures}}
\begin{assumption}
From now on we assume that the orbit $[\bL,\iota_0]$ of $(\bL,\iota_0)$ is $F$-stable. Furthermore, we assume that $(\bL,\iota_0) \in \widetilde{\mathcal{W}}_{\bG}^F$ is $F$-stable and that the parabolic subgroup $\bP$ is $F$-stable. Finally we assume chosen an involutive automorphism $\overline{\phantom{x}} : \Ql \to \Ql$ which maps any root of unity to its inverse.
\end{assumption}

\subsection{Isomorphisms}

\begin{pa}\label{pa:iso-defining-prop}
We will now follow the path of \cite[\S6]{lusztig:1992:a-unipotent-support}. However, unlike \cite{lusztig:1992:a-unipotent-support} we will not assume that $F$ is a split Frobenius endomorphism. We will denote by $\varphi_0 : F^*\mathscr{E}_0 \to \mathscr{E}_0$ a fixed isomorphism chosen such that the induced isomorphism $(\mathscr{E}_0)_x \to (\mathscr{E}_0)_x$ at the stalk of any element $x \in \mathcal{O}_0^F$ has finite order. Recalling the notation from \cref{pa:induction-diagram} we have the choice of $\varphi_0$ naturally determines an isomorphism $\phi_0 : F^*K_0 \to K_0$ and hence a corresponding automorphism $\phi_0 : \mathscr{H}_x^i(K_0) \to\mathscr{H}_x^i(K_0)$ for any $x \in \overline{Y}^F$ and $i \in \mathbb{Z}$, see \cite[6.12]{taylor:2014:evaluating-characteristic-functions}. Taking the contragradient of $\varphi_0$ we obtain an isomorphism $\varphi_0^{\vee} : F^*\mathscr{E}_0^{\vee} \to \mathscr{E}_0^{\vee}$ which in turn induces isomorphisms $\phi_0^{\vee} : F^*K_0^{\vee} \to K_0^{\vee}$ and $\widetilde{\phi}_0^{\vee} : F^*\widetilde{K}_0^{\vee} \to \widetilde{K}_0^{\vee}$. These in turn induce automorphisms $\mathscr{H}_x^i(K_0^{\vee})$ and $\bH_c^i(\overline{Y}\cap\mathcal{N}\cap \Sigma,\widetilde{K}_0^{\vee})$ for any $x \in \overline{Y}^F$ and $i \in \mathbb{Z}$.
\end{pa}

\begin{pa}\label{pa:iso-preferred-extension}
The choice of parabolic subgroup $\bP$ naturally determines a set of Coxeter generators $\mathbb{S}$ of the relative Weyl group $W_{\bG}(\bL)$, see \cite[Proposition 1.12]{bonnafe:2004:actions-of-rel-Weyl-grps-I}. Furthermore, as both $\bL$ and $\bP$ are chosen to be $F$-stable we have $F$ induces an automorphism $W_{\bG}(\bL) \to W_{\bG}(\bL)$ which stabilises $\mathbb{S}$. For any irreducible $W_{\bG}(\bL)$-module $E$ we denote by $E_F$ the $\Ql W_{\bG}(\bL)$-module obtained from $E$ by twisting with $F^{-1}$. Now the module $E$ is $F$-stable if and only if there exists an isomorphism $E \to E_F$ of $\Ql W_{\bG}(\bL)$-modules. Furthermore, any choice of such an isomorphism naturally determines a $\Ql\widetilde{W}_{\bG}(\bL)$-module structure on $E$ where $\widetilde{W}_{\bG}(\bL) = W_{\bG}(\bL) \rtimes \langle F \rangle$.

Let $\iota \in \mathscr{I}[\bL,\iota_0]^F$ be an $F$-stable pair in the block then we wish to choose an isomorphism $\phi_{\iota} : F^*K_{\iota} \to K_{\iota}$. Recall that $E_{\iota} \in \Irr(W_{\bG}(\bL))^F$ denotes the corresponding simple module under \cref{eq:gen-spring-cor-F-fixed}. Following the construction in \cite[6.13]{taylor:2014:evaluating-characteristic-functions} we see that choosing an isomorphism $\phi_{\iota}$ is equivalent to choosing an isomorphism $\psi_{\iota} : E \to E_F$ of $\Ql W_{\bG}(\bL)$-modules. Thus we define $\phi_{\iota}$ by requiring that the action of $F^{-1}$, through $\psi_{\iota}$, makes $E_{\iota}$ Lusztig's preferred extension of $E_{\iota}$ as a $\Ql \widetilde{W}_{\bG}(\bL)$-module, c.f., \cite[17.2]{lusztig:1986:character-sheaves-IV}.
\end{pa}

\begin{pa}
From the isomorphism \cref{eq:K-restriction} we recover the local system $\mathscr{E}_{\iota}$ via the isomorphism
\begin{equation*}
\mathscr{H}^{-\dim\mathcal{O}_{\iota} - \dim \lie{z(l)}}(K_{\iota})|_{\mathcal{O}_{\iota}} \cong \mathscr{E}_{\iota}.
\end{equation*}
Through this we see that $\phi_{\iota}$ determines an isomorphism $F^*\mathscr{E}_{\iota} \to \mathscr{E}_{\iota}$ which is of the form $q^{b_{\iota}/2}\varphi_{\iota}$ where
\begin{equation}
\begin{aligned}
b_{\iota} &= -\dim\mathcal{O}_{\iota} - \dim\lie{z(l)} + \dim Y\\
&= (\dim\bG - \dim\mathcal{O}_{\iota}) - (\dim\bL - \dim\mathcal{O}_0)
\end{aligned}
\end{equation}
and $\varphi_{\iota} : F^*\mathscr{E}_{\iota} \to \mathscr{E}_{\iota}$ is an isomorphism which induces an automorphism $\varphi_{\iota} : (\mathscr{E}_{\iota})_x \to (\mathscr{E}_{\iota})_x$ of finite order for any $x \in \mathcal{O}_{\iota}^F$, see \cite[24.2.4]{lusztig:1984:intersection-cohomology-complexes}. For convenience we will also define the value
\begin{equation}
a_{\iota} = -\dim\mathcal{O}_{\iota} - \dim Z^{\circ}(\bL)
\end{equation}
for any $\iota \in \mathscr{I}[\bL,\iota_0]$.
\end{pa}

\begin{pa}\label{pa:twisted-levis}
For $w \in W_{\bG}(\bL)$ we consider an $F$-stable Levi subgroup $\bL_w = g\bL g^{-1}$, where $g \in \bG$ is such that $g^{-1}F(g) = \dot{w}^{-1} \in N_{\bG}(\bL)$ is a representative of $w^{-1} \in W_{\bG}(\bL)$. Conjugating the cuspidal pair $\iota_0$ we also obtain a corresponding $F$-stable cuspidal pair $(\mathcal{O}_w,\mathscr{E}_w) \in \mathcal{V}_{\bL_w}^0$ where
\begin{equation*}
\mathcal{O}_w = {}^g\mathcal{O}_0 \qquad\qquad \mathscr{E}_w = (\Ad g^{-1})^*\mathscr{E}_0.
\end{equation*}
We now obtain complexes $K_w$, $\widetilde{K}_w$ and $K_w^{\vee}$ from $(\bL_w,\mathcal{O}_w,\mathscr{E}_w)$ just as $K_0$, $\widetilde{K}_0$ and $K_0^{\vee}$ were obtained from $(\bL,\mathcal{O}_0,\mathscr{E}_0)$. As in \cite[6.9]{taylor:2014:evaluating-characteristic-functions} using Lusztig's basis element $\theta_w$ of $\End(\mathscr{E}_0)$, c.f., \cref{prop:algebra-basis}, we see that the fixed isomorphism $\varphi_0 : F^*\mathscr{E}_0 \to \mathscr{E}_0$ determines an isomorphism $\varphi_w : F^*\mathscr{E}_w \to \mathscr{E}_w$. In turn, this induces isomorphisms
\begin{equation*}
\phi_w : F^*K_w \to K_w \qquad \phi_w^{\vee} : F^*K_w^{\vee} \to K_w^{\vee} \qquad \widetilde{\phi}_w : F^*\widetilde{K}_w \to \widetilde{K}_w,
\end{equation*}
where $\phi_w^{\vee}$ is the contragradient of $\phi_w$, see \cite[6.12]{taylor:2014:evaluating-characteristic-functions}.
\end{pa}

\begin{rem}\label{rem:digne-lehrer-michel}
In light of \cref{lem:bijection-tensor-sign,pa:iso-preferred-extension} we would like to recall the following property of Lusztig's preferred extension as observed by Digne--Lehrer--Michel in \cite[Remark 3.6]{digne-lehrer-michel:2003:space-of-unipotently-supported}. For any pair $\iota \in \mathscr{I}[\bL,\iota_0]^F$ there exists a unique sign $\varepsilon_{\iota} \in \{\pm 1\}$ such that
\begin{equation*}
\Tr(wF,E_{\hat{\iota}}) = \varepsilon_{\iota}\sgn(w)\Tr(wF,E_{\iota})
\end{equation*}
for all $w \in W_{\bG}(\bL)$. Note that here we are using the characterisation of the map $\iota\mapsto \hat{\iota}$ given in \cref{lem:bijection-tensor-sign}.
\end{rem}

\subsection{Lusztig's Algorithm}
\begin{pa}
We now define functions $Y_{\iota}, X_{\iota}, \widetilde{X}_{\iota} : \lie{g}^F \to \Ql$ by setting
\begin{align*}
Y_{\iota}(y) &= \Tr(\varphi_{\iota},(\mathscr{E}_{\iota})_y)\\
X_{\iota}(x) &= \sum_{i \in \mathbb{Z}}(-1)^i\Tr(\phi_{\iota}, \mathscr{H}_x^i(K_{\iota}))\\
\widetilde{X}_{\iota}(x) &= \sum_{i \in \mathbb{Z}}(-1)^i\Tr(\phi_{\iota}^{-1}, \mathscr{H}_x^i(K_{\iota}))
\end{align*}
if $y \in \mathcal{O}_{\iota}^F$, $x \in \overline{\mathcal{O}_{\iota}}^F$ and 0 otherwise. For any $x \in \lie{g}$ we denote by $A_{\bG}(x)$ the component group $C_{\bG}(x)/C_{\bG}^{\circ}(x)$ of the centraliser. With this notation we have the following lemma.
\end{pa}

\begin{lem}\label{lem:orthogonality-relations}
Let $y = y_1,\dots,y_m$ be a set of representatives for the orbits of $G$ acting on $\mathcal{O}_{\iota}^F$ then the following orthogonality relations hold:
\begin{equation*}
\begin{aligned}
\sum_{i=1}^m [A_{\bG}(y) : A_{\bG}(y_i)^F]Y_{\iota}(y_i)\overline{Y_{\iota'}(y_i)} &= |A_{\bG}(y)|\delta_{\iota',\iota}\\
\sum_{\substack{\iota' \in \mathcal{V}_{\bG}\\\mathcal{O}_{\iota'} = \mathcal{O}_{\iota}}} Y_{\iota'}(y_i)\overline{Y_{\iota'}(y_j)} &= |A_{\bG}(y_i)^F|\delta_{i,j}.
\end{aligned}
\end{equation*}
Here $\delta_{\iota',\iota}$ and $\delta_{i,j}$ denote the Kronecker delta.
\end{lem}

\begin{proof}
Let $y \in \mathcal{O}_{\iota}^F$ be a fixed class representative and denote by $H^1(F,A_{\bG}(y))$ the $F$-conjugacy classes of the component group. Then we may realise the set of representatives $\{y_1,\dots,y_m\}$ as the set of $y_a = gyg^{-1}$ for every $a \in H^1(F,A_{\bG}(y))$, where $g^{-1}F(g) \in C_{\bG}(y)$ is a representative of $a \in A_{\bG}(y)$. From the definition we see that $A_{\bG}(y_a)^F$ is naturally isomorphic to the $F$-centraliser $C_{A_{\bG}(y),F}(a) = \{b \in A_{\bG}(y) \mid b^{-1}aF(b) = a\}$. Let us denote by $\chi_{\iota} \in \Irr(A_{\bG}(u))$ the irreducible character corresponding to the local system $\mathscr{E}_{\iota}$. Then $\chi_{\iota}$ is $F$-stable, as $\mathscr{E}_{\iota}$ is $F$-stable, and we can choose an extension $\widetilde{\chi}_{\iota}$ to the semidirect product $A_{\bG}(y) \rtimes \langle F \rangle$. According to \cite[1.3]{shoji:2006:generalized-green-functions-I} there exists an isomorphism $\psi_{\iota} : F^*\mathscr{E}_{\iota} \to \mathscr{E}_{\iota}$ and a scalar $\xi_{\iota} \in \Ql^{\times}$ such that $\varphi_{\iota} = \xi_{\iota}\psi_{\iota}$ and
\begin{equation*}
\widetilde{\chi}_{\iota}(aF) = \Tr(\psi_{\iota}, (\mathscr{E}_{\iota})_{y_a})
\end{equation*}
for all $a \in A_{\bG}(y)$. In fact, the isomorphism $\psi_{\iota}$ induces a finite order automorphism on the stalk $(\mathscr{E}_{\iota})_y$. As $\varphi_{\iota}$ also has this property we must have $\xi_{\iota}^m = 1$ for some $m \geqslant 1$, hence $\xi_{\iota}$ is a root of unity. In particular, we have
\begin{equation*}
Y_{\iota}(y_a)\overline{Y_{\iota'}(y_a)} = \xi_{\iota}\xi_{\iota'}^{-1}\widetilde{\chi}_{\iota}(aF)\overline{\widetilde{\chi}_{\iota'}(aF)}
\end{equation*}
for all $a \in A_{\bG}(y)$ and so the result follows from the orthogonality relations of cosets, see \cite[II, Corollaire 2.10]{digne-michel:1985:fonctions-L-des-varietes} and \cite[8.14]{isaacs:2006:character-theory-of-finite-groups}.
\end{proof}

\begin{pa}\label{pa:basis-Ys}
The set of functions $\{Y_{\iota} \mid \iota \in \mathcal{V}_{\bG}^F\}$ forms a basis for the space $\Centn(\lie{g}^F)$ of $\Ad G$-invariant functions $\lie{g}^F \to \Ql$ which are supported on $\mathcal{N}^F$, see \cite[24.2.7]{lusztig:1986:character-sheaves-V}. In particular, for any two pairs $\iota,\iota' \in \mathcal{V}_{\bG}^F$ there exists a scalar $P_{\iota',\iota} \in \Ql$ such that
\begin{equation*}
X_{\iota} = \sum_{\iota'\in \mathcal{V}_{\bG}^F} P_{\iota',\iota}Y_{\iota'}.
\end{equation*}
By the definition of the functions $X_{\iota}$ and $Y_{\iota'}$ we see that
\begin{equation}\label{eq:P-conditon}
P_{\iota',\iota} = \begin{cases}
1 &\text{if }\iota' = \iota,\\
0 &\text{if }\mathcal{O}_{\iota'} \not\subseteq \overline{\mathcal{O}_{\iota}} \text{ or if }\mathcal{O}_{\iota'} = \mathcal{O}_{\iota} \text{ and }\iota\neq\iota'.
\end{cases}
\end{equation}
for any $\iota',\iota \in \mathcal{V}_{\bG}^F$, see also \cite[24.2.10, 24.2.11]{lusztig:1984:intersection-cohomology-complexes}. We now also define scalars $\lambda_{\iota',\iota} \in \Ql$ by setting
\begin{equation*}
\lambda_{\iota,\iota'} = \lambda_{\iota',\iota} = \sum_{y \in \mathcal{N}^F} Y_{\iota}(y)\overline{Y_{\iota'}(y)},
\end{equation*}
which are integers by \cref{lem:orthogonality-relations}. Moreover, from the definition of the functions $Y_{\iota}$ it is clear that we have
\begin{equation}\label{eq:Lambda-condition}
\lambda_{\iota',\iota} = 0\text{ if }\mathcal{O}_{\iota'} \neq \mathcal{O}_{\iota}.
\end{equation}
\end{pa}

\begin{pa}
For any $w \in W_{\bG}(\bL)$ let $\bL_w$ be as in \cref{pa:twisted-levis} then for any $\iota',\iota \in \mathscr{I}[\bL,\iota_0]^F$ we set
\begin{equation*}
\omega_{\iota,\iota'} = \omega_{\iota',\iota} = q^{-\dim\bG - (a_{\iota}+a_{\iota'})/2}\frac{1}{|W_{\bG}(\bL)|}\sum_{w \in W_{\bG}(\bL)} \Tr(wF,E_{\iota})\Tr(wF,E_{\iota'})\frac{|\bG^F|}{|Z^{\circ}(\bL_w)^F|}.
\end{equation*}
Here $Z^{\circ}(\bL_w)$ denotes the connected centre of $\bL_w$ and $F$ is acting on $E_{\iota}$ via the isomorphism $\psi_{\iota}^{-1}$, similarly for $E_{\iota'}$, c.f., \cref{pa:iso-preferred-extension}. If either $\iota$ or $\iota'$ are not contained in the block $\mathscr{I}[\bL,\iota_0]$ then we set $\omega_{\iota,\iota'} = 0$. Our definition of the term $\omega_{\iota,\iota'}$ is slightly different to that given in \cite[24.3.4]{lusztig:1986:character-sheaves-V}. However, one sees that these definitions are equivalent by \cite[24.2.1]{lusztig:1986:character-sheaves-V}. With this we have the following theorem of Lusztig.
\end{pa}

\begin{thm}[{}{Lusztig, \cite[Theorem 24.4]{lusztig:1986:character-sheaves-V}}]\label{thm:gen-green-funcs}
Let us denote by $P$, $\Lambda$ and $\Omega$ the matrices $(P_{\iota',\iota})$, $(\lambda_{\iota',\iota})$ and $(\omega_{\iota',\iota})$ respectively. Then the entries of the matrices $P$ and $\Lambda$ are the unique solution to the system of equations given by $P\Lambda P^T = \Omega$, \cref{eq:P-conditon} and \cref{eq:Lambda-condition}. Furthermore, we have:
\begin{enumerate}
	\item $P$, $\Lambda$ and $\Omega$ are non-singular integer valued matrices,
	\item $P_{\iota',\iota} = \lambda_{\iota',\iota} = 0$ if $\iota',\iota \in \mathcal{V}_{\bG}^F$ lie in different blocks.
\end{enumerate}
\end{thm}

\begin{pa}
Recall that on the character group $X(\bT_0)$ of our chosen maximal torus the Frobenius endomorphism factors as $q\tau$ where $\tau$ is a finite order automorphism, c.f., \cref{thm:isogeny-theorem}. For any $n \geqslant 1$ the endomorphism $q^n\tau$ of $X(\bT_0)$ lifts to a Frobenius endomorphism $F' : \bG \to \bG$ which determines an $\mathbb{F}_{q^n}$-rational structure of $\bG$, c.f., \cref{thm:isogeny-theorem}. Replacing $q$ by $q^n$ in the entries for $P$ we see that we obtain the corresponding matrix for $\bG^{F'}$. In this way we may view the entries of $P$ as polynomials in a single variable, say $\mathbf{q}$. We now denote by $P^{\star} = (P_{\iota',\iota}^{\star})$ the rational valued matrix obtained from $P$ by evaluating $\mathbf{q}$ at $q^{-1}$. This matrix is such that
\begin{equation}
\widetilde{X}_{\iota}(y) = \sum_{\iota' \in \mathcal{V}_{\bG}^F} P_{\iota',\iota}^{\star}\overline{Y_{\iota'}(y)}
\end{equation}
for all $\iota \in \mathcal{V}_{\bG}^F$ and $y \in \mathcal{N}^F$.
\end{pa}

\begin{pa}
Let us denote by $Q = (Q_{\iota',\iota})$, $\widetilde{\Lambda} = (\widetilde{\lambda}_{\iota',\iota})$ and $\widetilde{\Omega} = (\widetilde{\omega}_{\iota',\iota})$ the inverse matrices to $P$, $\Lambda$ and $\Omega$ respectively, c.f., \cref{thm:gen-green-funcs}. It is clear that we have
\begin{gather}
Q\widetilde{\Lambda}Q^T = \widetilde{\Omega}\label{eq:inv-mat-prod}\\
Y_{\iota} = \sum_{\iota' \in \mathcal{V}_{\bG}^F} Q_{\iota',\iota}X_{\iota'}
\end{gather}
for all $\iota \in \mathcal{V}_{\bG}^F$. Using the coset orthogonality relations for finite groups we deduce that
\begin{equation}\label{eq:omega-inv}
\widetilde{\omega}_{\iota,\iota'} =  q^{\dim\bG + (a_{\iota}+a_{\iota'})/2}\frac{1}{|W_{\bG}(\bL)|}\sum_{w \in W_{\bG}(\bL)} \Tr(wF,E_{\iota})\Tr(wF,E_{\iota'})\frac{|Z^{\circ}(\bL_w)^F|}{|\bG^F|},
\end{equation}
if $\iota$, $\iota' \in \mathscr{I}[\bL,\iota_0]^F$, see also \cite[Lemma 5.1]{digne-lehrer-michel:2003:space-of-unipotently-supported}. Clearly $\widetilde{\omega}_{\iota,\iota'} = 0$ if $\iota,\iota' \in \mathcal{V}_{\bG}^F$ are not contained in the same block.
\end{pa}

\begin{assumption}
From this point forward we assume that $-e^{\dag} \in \overline{Y}$.
\end{assumption}

\subsection{Formula for the Fourier Transform of the GGGR}
\begin{pa}
Let $\bL_w$ for some $w \in W_{\bG}(\bL)$ be as in \cref{pa:twisted-levis} then by the argument in \cite[10.6, 25.6.3]{lusztig:1985:character-sheaves}, see also \cite[6.15]{taylor:2014:evaluating-characteristic-functions} and \cite[6.9]{lusztig:1992:a-unipotent-support}, we have the following equalities
\begin{align}
\chi_{K_w,\phi_w}(y) = \chi_{K_0,\theta_w\circ\phi_0}(y) &= \sum_{\iota \in \mathscr{I}[\bL,\iota_0]^F} q^{b_{\iota}}\Tr(wF,E_{\iota})X_{\iota}(y),\label{eq:twist-X}\\
\chi_{K_w,\phi_w^{-1}}(y) = \chi_{K_0,\theta_w\circ\phi_0^{-1}}(y) &= \sum_{\iota \in \mathscr{I}[\bL,\iota_0]^F} q^{-b_{\iota}}\Tr(wF,E_{\iota})\widetilde{X}_{\iota}(y),\label{eq:twist-X-inv}\\
\chi_{K_w,\phi_w^{\vee}}(y) = \chi_{K_0,\theta_w\circ\phi_0^{\vee}}(y) &= \sum_{\iota \in \mathscr{I}[\bL,\iota_0]^F} q^{b_{\iota}}\Tr(wF,E_{\iota})\overline{X_{\iota}(y)},\label{eq:twist-X-vee}
\end{align}
for any $y \in (\overline{Y}\cap\mathcal{N})^F$.
\end{pa}

\begin{pa}
According to \cite[6.10(a)]{lusztig:1992:a-unipotent-support} we have, by Grothendieck's trace formula, that
\begin{equation}\label{eq:grothendieck-trace}
\sum_{i \in \mathbb{Z}}(-1)^i\Tr(\widetilde{\phi}_w^{\vee},\bH_c^i(\overline{Y}\cap\mathcal{N}\cap \Sigma,\widetilde{K}_w^{\vee})) = \sum_{x \in (\overline{Y}\cap\mathcal{N}\cap\Sigma)^F}\sum_{i \in \mathbb{Z}}(-1)^i\Tr(\phi_w^{\vee},\mathscr{H}_x^i(K_w^{\vee})).
\end{equation}
Using \cref{eq:poincare-duality-II} we get
\begin{equation}\label{eq:application-poincare}
\Tr(\widetilde{\phi}_w^{\vee},\bH_c^{2d_{\iota_0} - i}(\overline{Y}\cap\mathcal{N}\cap \Sigma,\widetilde{K}_w^{\vee})) = q^{d_{\iota_0}}\Tr(\phi_w^{-1},\mathscr{H}_{-e^{\dag}}^i(K_w))
\end{equation}
where $d_{\iota_0}$ is the integer in (iii) of \cref{lem:bumper-facts}. Now, combining \cref{eq:grothendieck-trace,eq:application-poincare} we obtain
\begin{equation*}
q^{d_{\iota_0}}\sum_{i \in \mathbb{Z}}(-1)^i\Tr(\phi_w^{-1},\mathscr{H}_{-e^{\dag}}^i(K_w)) = \sum_{x \in (\overline{Y}\cap\mathcal{N}\cap\Sigma)^F}\sum_{i \in \mathbb{Z}}(-1)^i\Tr(\phi_w^{\vee},\mathscr{H}_x^i(K_w^{\vee}))
\end{equation*}
and applying \cref{eq:twist-X-inv,eq:twist-X-vee} to this equality we obtain
\begin{equation*}
\sum_{\iota \in \mathscr{I}[\bL,\iota_0]^F} q^{d_{\iota_0}-b_{\iota}/2}\widetilde{X}_{\iota}(-e^{\dag})f_{\iota} = \sum_{x \in (\overline{Y}\cap\mathcal{N}\cap\Sigma)^F}\sum_{\iota \in \mathscr{I}[\bL,\iota_0]^F} q^{b_{\iota}/2}\overline{X_{\iota}(x)}f_{\iota},
\end{equation*}
where $f_{\iota} : W_{\bG}(\bL) \to \Ql$ is defined by $f_{\iota}(w) = \Tr(wF,E_{\iota})$ for all $w \in W_{\bG}(\bL)$. Note that the set of functions $\{f_{\iota} \mid \iota \in \mathscr{I}[\bL,\iota_0]^F\}$ is linearly independent hence we get
\begin{equation*}
q^{d_{\iota_0}-b_{\iota}/2}\widetilde{X}_{\iota}(-e^{\dag}) = \sum_{x \in (\overline{Y}\cap\mathcal{N}\cap\Sigma)^F}q^{b_{\iota}/2}\overline{X_{\iota}(x)}.
\end{equation*}
Rewriting the $X_{\iota}$'s in terms of the $Y_{\iota}$'s and conjugating by $\overline{\phantom{x}}$ we have
\begin{equation}\label{eq:sum-Ys}
q^{d_{\iota_0}-b_{\iota}}\sum_{\iota' \in \mathcal{V}_{\bG}^F}P_{\iota',\iota}^{\star}Y_{\iota'}(-e^{\dag}) = \sum_{x \in (\overline{Y}\cap\mathcal{N}\cap\Sigma)^F}\sum_{\iota' \in \mathcal{V}_{\bG}^F}P_{\iota',\iota}Y_{\iota'}(x).
\end{equation}
Note that $P_{\iota',\iota}^{\star}$ and $P_{\iota',\iota}$ are both rational and neither sum on the right hand side depends on $\iota$.

Finally, as $Q$ is the inverse to $P$ we have $\sum_{\iota \in \mathscr{I}[\bL,\iota_0]^F} P_{\iota',\iota}Q_{\iota,\iota''} = \delta_{\iota',\iota''}$ (the Kronecker delta) by \cref{eq:P-conditon} for any $\iota' \in \mathcal{V}_{\bG}^F$ and $\iota'' \in \mathscr{I}[\bL,\iota_0]^F$. Thus, multiplying both sides of \cref{eq:sum-Ys} by $\sum_{\iota \in \mathscr{I}[\bL,\iota_0]^F} Q_{\iota,\iota''}$ we obtain
\begin{equation}\label{eq:sum-over-double-inter}
\sum_{x \in (\overline{Y}\cap\mathcal{N}\cap\Sigma)^F}Y_{\iota''}(x) = \sum_{\iota,\iota' \in \mathscr{I}[\bL,\iota_0]^F}q^{d_{\iota_0}-b_{\iota}}P_{\iota',\iota}^{\star}Q_{\iota,\iota''}Y_{\iota'}(-e^{\dag}),
\end{equation}
which holds for any $\iota'' \in \mathscr{I}[\bL,\iota_0]^F$. As is pointed out in \cite[6.9]{lusztig:1992:a-unipotent-support} we have used here the assumption that $-e^{\dag} \in \overline{Y}$ but in fact this equality holds regardless as both sides are zero when $-e^{\dag} \not\in \overline{Y}$.
\end{pa}

\begin{pa}
By \cref{pa:basis-Ys} we see there exist unique scalars $\alpha_{\iota_1}, \beta_{\iota_1} \in \Ql$, for $\iota_1 \in \mathcal{V}_{\bG}^F$, such that
\begin{equation*}
\mathcal{F}(\boldsymbol\Gamma_u)|_{\mathcal{N}^F} = \sum_{\iota_1 \in \mathcal{V}_{\bG}^F} \alpha_{\iota_1}Y_{\iota_1} = \sum_{\iota_1 \in \mathcal{V}_{\bG}^F} \beta_{\iota_1}X_{\iota_1}.
\end{equation*}
Our goal is to now try and determine these scalars. Multiplying $\mathcal{F}(\boldsymbol\Gamma_u)|_{\mathcal{N}^F}$ by the complex conjugate of $Y_{\iota_2}$ and summing over $\mathcal{N}^F$, then inverting the matrix $\Lambda$, we get
\begin{equation*}
\alpha_{\iota_1} = \sum_{\iota_2 \in \mathscr{I}[\bL,\iota_0]^F}\widetilde{\lambda}_{\iota_2,\iota_1}\sum_{x \in \mathcal{N}^F}\widehat{\boldsymbol\Gamma}_u(x)\overline{Y_{\iota_2}(x)}
\end{equation*}
for any $\iota_1 \in \mathcal{V}_{\bG}^F$, c.f., \cref{pa:basis-Ys}. Applying \cref{prop:fourier-GGGR} we may rewrite this as
\begin{align*}
\alpha_{\iota_1} &= \sum_{\iota_2 \in \mathscr{I}[\bL,\iota_0]^F}\widetilde{\lambda}_{\iota_2,\iota_1}\sum_{x \in \mathcal{N}^F}q^{r_u}|\{g \in G \mid \Ad g(x) \in \Sigma\}|\overline{Y_{\iota_2}(x)}\\
&= q^{r_u}|G|\sum_{\iota_2 \in \mathscr{I}[\bL,\iota_0]^F}\widetilde{\lambda}_{\iota_2,\iota_1}\sum_{x \in \mathcal{N}^F\cap\Sigma}\overline{Y_{\iota_2}(x)}
\intertext{where $r_u$ is as in \cref{prop:fourier-GGGR}. Now if $\widetilde{\lambda}_{\iota_2,\iota_1} \neq 0$ then we must have $\iota_1$ and $\iota_2$ are in the same block but if this is the case then $Y_{\iota_2}(x) \neq 0$ implies $x \in \overline{Y}$. So the right most sum can be taken over $(\overline{Y} \cap \mathcal{N} \cap \Sigma)^F$. Thus, using \cref{eq:sum-over-double-inter} we get}
\alpha_{\iota_1} &= |G|\sum_{\iota_2,\iota,\iota' \in \mathscr{I}[\bL,\iota_0]^F}q^{d_{\iota_0} - b_{\iota} + r_u}\widetilde{\lambda}_{\iota_2,\iota_1}P_{\iota',\iota}^{\star}Q_{\iota,\iota_2}\overline{Y_{\iota'}(-e^{\dag})}
\end{align*}

Note that the scalars $\alpha_*$ and $\beta_*$ are related via the equation $\beta_{\iota_1} = \sum_{\iota_1' \in \mathscr{I}[\bL,\iota_0]^F} Q_{\iota_1,\iota_1'}\alpha_{\iota_1'}$ so applying this to the above expression for $\alpha_{\iota_1'}$ we obtain
\begin{align*}
\beta_{\iota_1} &= |G|\sum_{\iota,\iota' \in \mathscr{I}[\bL,\iota_0]^F}q^{d_{\iota_0} - b_{\iota} + r_u}P_{\iota',\iota}^{\star}\overline{Y_{\iota'}(-e^{\dag})}\sum_{\iota_1',\iota_2 \in \mathscr{I}[\bL,\iota_0]^F}Q_{\iota,\iota_2}\widetilde{\lambda}_{\iota_2,\iota_1'}Q_{\iota_1,\iota_1'}\\
&= |G|\sum_{\iota,\iota' \in \mathscr{I}[\bL,\iota_0]^F}q^{d_{\iota_0} - b_{\iota} + r_u}P_{\iota',\iota}^{\star}\widetilde{\omega}_{\iota,\iota_1}\overline{Y_{\iota'}(-e^{\dag})}
\end{align*}
where in the second equality we have used \cref{eq:inv-mat-prod}. Finally, using \cref{eq:omega-inv} we have the following.
\end{pa}

\begin{prop}\label{prop:fourier-inverse-decomp}
We define $\widehat{\boldsymbol\Gamma}_{u,\mathscr{I}[\bL,\iota_0]} : \mathcal{N}^F \to \Ql$ to be the function
\begin{equation*}
\sum_{\iota,\iota',\iota_1 \in \mathscr{I}[\bL,\iota_0]^F} q^{f(\iota,\iota_1)}\frac{1}{|W_{\bG}(\bL)|}\sum_{w \in W_{\bG}(\bL)}\Tr(wF,E_{\iota})\Tr(wF,E_{\iota_1})|Z^{\circ}(\bL_w)^F|P_{\iota',\iota}^{\star}\overline{Y_{\iota'}(-e^{\dag})}X_{\iota_1}
\end{equation*}
where
\begin{align*}
f(\iota,\iota_1) &= d_{\iota_0} - b_{\iota} + r_u + \dim\bG + (a_{\iota}+a_{\iota_1})/2\\
&= \dim\bG - \dim Z^{\circ}(\bL) + (\dim\mathcal{O}_{\iota} - \dim\mathcal{O}_{\iota_1} - \dim \mathcal{O}_u)/2.
\end{align*}
Then we have
\begin{equation*}
\mathcal{F}(\boldsymbol\Gamma_u)|_{\mathcal{N}^F} = \sum_{[\bL,\iota_0] \in \mathcal{W}_{\bG}^F} \widehat{\boldsymbol\Gamma}_{u,\mathscr{I}[\bL,\iota_0]},
\end{equation*}
where the sum is taken over all $F$-stable blocks.
\end{prop}

\section{\texorpdfstring{A Decomposition of $\Gamma_u$}{A Decomposition of GGGRs}}
\begin{pa}
By the definition of the bijection $\hat{\phantom{x}} : \mathscr{I}[\bL,\iota_0] \to \mathscr{I}[\bL,\iota_0]$ and \cite[5.2.3]{letellier:2005:fourier-transforms} we see that for any $\iota \in \mathscr{I}[\bL,\iota_0]$ there exists a scalar $c_{\iota} \in \Ql^{\times}$ such that
\begin{equation}\label{eq:scalar-hat-bijection}
\mathcal{F}(X_{\iota})|_{\mathcal{N}^F} = c_{\iota}X_{\hat{\iota}}|_{\mathcal{N}^F}.
\end{equation}
For any function $f : \lie{g}^F \to \Ql$ we will denote by $f^* : \lie{g}^F \to \Ql$ the function obtained as the extension by 0 of $f|_{\mathcal{N}^F}$. With this in hand we have the following lemma.
\end{pa}

\begin{lem}\label{lem:fourier-induced-complex}
There exists a sign $\nu \in \{\pm1\}$ such that $\Psi(\chi_{K_0,\phi_0}) = \nu\chi_{K_0,\phi_0}$, where $\Psi$ is as in \cref{lem:fourier-inverse} and
\begin{equation*}
\mathcal{F}(\chi_{K_0,\phi_0}) = \gamma q^{(\dim\bG+\dim Z^{\circ}(\bL))/2}\chi_{K_0,\phi_0}^*
\end{equation*}
where $\gamma \in \Ql^{\times}$ is such that $\gamma^2 = \nu$, hence $\gamma^4=1$. The constant $\gamma$ does not depend upon the choice of isomorphism $\phi_0$ and we have $\gamma = \nu = 1$ if $\bL$ is a torus. Furthermore, assuming the conclusion of \cref{lem:bijection-not-id} holds then we have
\begin{equation}\label{eq:fourier-transform-equality}
\mathcal{F}(\chi_{K_0,\theta_w\circ\phi_0}) = \gamma \sgn(w)q^{(\dim\bG+\dim Z^{\circ}(\bL))/2} \chi_{K_0,\theta_w\circ\phi_0}^*
\end{equation}
for any $w \in W_{\bG}(\bL)$.
\end{lem}

\begin{proof}
All parts follow from the proof of \cite[6.2.9]{letellier:2005:fourier-transforms} together with \cite[4.4.6, 5.2.3, 5.2.8]{letellier:2005:fourier-transforms} and \cite[6.2.8, 6.2.12, 6.2.15]{letellier:2005:fourier-transforms}. One only has to take into consideration that our definitions of the Fourier and Deligne--Fourier transforms are slightly different to those used in \cite{letellier:2005:fourier-transforms}, c.f., \cref{rem:diff-defs}.
\end{proof}

\begin{prop}\label{prop:fourth-root-unity}
Assume either that $F$ is split or that the conclusion of \cref{lem:bijection-not-id} holds then there exists a fourth root of unity $\zeta \in \Ql^{\times}$ such that
\begin{equation*}
c_{\iota} = \varepsilon_{\iota}\zeta q^{(\dim\bG-\dim Z^{\circ}(\bL) - \dim\mathcal{O}_{\hat{\iota}} + \dim\mathcal{O}_{\iota})/2},
\end{equation*}
where $\varepsilon_{\iota}$ is as in \cref{rem:digne-lehrer-michel}. In particular, we have $\zeta_{\mathscr{I}} = \zeta$ depends only on the block $\mathscr{I} = \mathscr{I}[\bL,\iota_0]$ and not on $\iota$ itself. Furthermore, if $\bL$ is a torus then $\zeta_{\mathscr{I}} = 1$.
\end{prop}

\begin{proof}
We are going to use \cref{lem:fourier-induced-complex} to determine the scalar $c_{\iota}$. Applying $\mathcal{F}^{-1}\circ \Psi$, c.f., \cref{lem:fourier-inverse}, to the equality \cref{eq:fourier-transform-equality} we have
\begin{equation}\label{eq:fourier-transform-equality-II}
\mathcal{F}(\chi_{K_0,\theta_w\circ\phi_0}^*) = \nu\gamma^{-1} \sgn(w)q^{(\dim\bG-\dim Z^{\circ}(\bL))/2}\chi_{K_0,\theta_w\circ\phi_0}.
\end{equation}
Now let us consider the restriction of the equality \cref{eq:fourier-transform-equality-II} to $(\overline{Y}\cap\mathcal{N})^F$ then applying \cref{eq:twist-X} we obtain
\begin{equation}\label{eq:fourier-transform-equality-III}
\sum_{\iota \in \mathscr{I}[\bL,\iota_0]^F} q^{b_{\iota}/2}\Tr(wF,E_{\iota})\mathcal{F}(X_{\iota})(y) = \nu\gamma^{-1} q^{(\dim\bG-\dim Z^{\circ}(\bL))/2}\sum_{\iota \in \mathscr{I}[\bL,\iota_0]^F} q^{b_{\iota}/2}\sgn(w)\Tr(wF,E_{\iota})X_{\iota}(y),
\end{equation}
for all $w \in W_{\bG}(\bL)$ and $y \in (\overline{Y}\cap\mathcal{N})^F$. Using \cref{eq:scalar-hat-bijection} and the change of variable $\iota \mapsto \hat{\iota}$ this becomes
\begin{equation}\label{eq:fourier-transform-equality-IV}
\sum_{\iota \in \mathscr{I}[\bL,\iota_0]^F} q^{b_{\iota}/2}\Tr(wF,E_{\iota})c_{\iota}X_{\hat{\iota}}(y) = \nu\gamma^{-1} q^{(\dim\bG-\dim Z^{\circ}(\bL))/2}\sum_{\iota \in \mathscr{I}[\bL,\iota_0]^F} q^{b_{\hat{\iota}}/2}\sgn(w)\Tr(wF,E_{\hat{\iota}})X_{\hat{\iota}}(y).
\end{equation}
The set of functions $\{X_{\hat{\iota}}|_{(\overline{Y}\cap\mathcal{N})^F} \mid \iota \in \mathscr{I}[\bL,\iota_0]^F\}$ is linearly independent. Thus, using \cref{rem:digne-lehrer-michel} we get that
\begin{equation*}
q^{b_{\iota}/2}\Tr(wF,E_{\iota})c_{\iota} = \varepsilon_{\iota}\nu\gamma^{-1} q^{(b_{\hat{\iota}} + \dim\bG-\dim Z^{\circ}(\bL))/2}\Tr(wF,E_{\iota})
\end{equation*}
Both sides of this equation may be 0 as $\Tr(wF,E_{\iota})$ could be 0. However, choosing a $w \in W_{\bG}(\bL)$ such that $\Tr(wF,E_{\iota}) \neq 0$ we deduce the result with $\zeta = \nu\gamma^{-1}$. Note that it is immediately clear from \cref{lem:fourier-induced-complex} that $\gamma = 1$ if $\bL$ is a torus.

Now assume $F$ is split then $F$ acts trivially on $W_{\bG}(\bL)$ and so it acts as the identity on the representation $E_{\iota}$. In this case we may take $w = 1$ in the above argument as we will have $\Tr(wF,E_{\iota}) = \dim(E_{\iota}) \neq 0$. When $w = 1$ we have the results of \cref{lem:fourier-induced-complex} hold without assuming \cref{lem:bijection-not-id} holds, so the statement follows.
\end{proof}

\begin{pa}\label{pa:bound-on-p}
We point out that Digne--Lehrer--Michel both stated and indicated how to prove \cref{prop:fourth-root-unity} in the proof of \cite[Proposition 6.1]{digne-lehrer-michel:2003:space-of-unipotently-supported}. We now come to the following main result of this article, which is due to Lusztig when $p$ is large enough so that $\exp$ and $\log$ define inverse bijections between $\mathcal{N}$ and $\mathcal{U}$.
\end{pa}

\begin{thm}\label{thm:main-theorem-nilpotent}
Let $\boldsymbol\Gamma_{u,\mathscr{I}[\bL,\iota_0]} : \lie{g}^F \to \Ql$ denote the function
\begin{equation*}
\zeta_{\mathscr{I}[\bL,\iota_0]}^{-1}\sum_{\iota,\iota',\iota_1 \in \mathscr{I}[\bL,\iota_0]^F} q^{f'(\iota,\iota_1)}\frac{1}{|W_{\bG}(\bL)|}\sum_{w \in W_{\bG}(\bL)}\Tr(wF,E_{\iota})\Tr(wF,E_{\hat{\iota}_1})|Z^{\circ}(\bL_w)^F|P_{\iota',\iota}^{\star}\overline{Y_{\iota'}(-e^{\dag})}\varepsilon_{\iota_1}X_{\iota_1}
\end{equation*}
where
\begin{align*}
f'(\iota,\iota_1) &= f(\iota,\iota_1) -  (\dim\bG-\dim Z^{\circ}(\bL) - \dim\mathcal{O}_{\iota_1} + \dim\mathcal{O}_{\iota_1})/2\\
&= (\dim\bG - \dim Z^{\circ}(\bL) + \dim\mathcal{O}_{\iota} - \dim\mathcal{O}_{\iota_1} - \dim \mathcal{O}_u)/2.
\end{align*}
Then we have
\begin{equation}\label{eq:decomp-gggr-nil}
\boldsymbol\Gamma_u = \sum_{[\bL,\iota_0] \in \mathcal{W}_{\bG}^F} \boldsymbol\Gamma_{u,\mathscr{I}[\bL,\iota_0]},
\end{equation}
where the sum is taken over all $F$-stable blocks.
\end{thm}

\begin{proof}
If $F$ is split then this is proved in exactly the same way as \cite[Theorem 7.3]{lusztig:1992:a-unipotent-support}. Once the theorem is deduced when $F$ is split we may prove \cref{lem:bijection-not-id} as in \cite[Theorem 7.7]{lusztig:1992:a-unipotent-support}. Finally, now that \cref{lem:bijection-not-id} is proved we have \cref{prop:fourth-root-unity} holds so we may again apply the proof in \cite[Theorem 7.3]{lusztig:1992:a-unipotent-support} to deduce the theorem when $F$ is twisted.
\end{proof}

\begin{pa}
We now wish to transfer \cref{thm:main-theorem-nilpotent} to a statement about the GGGR $\Gamma_u$. Let $\mathcal{V}_{\bG}^{\uni}$ denote the set of pairs $(\mathcal{O},\mathscr{E})$ consisting of a unipotent conjugacy class of $\bG$ and an irreducible local system on $\mathcal{O}$. It is easy to see that the Springer isomorphism induces a bijection $\mathcal{V}_{\bG}^{\nil} \to \mathcal{V}_{\bG}^{\uni}$ given by $(\mathcal{O},\mathscr{E}) \mapsto (\phi_{\spr}^{-1}(\mathcal{O}),\phi_{\spr}^*\mathscr{E})$. Assume now that $A \in \mathscr{D}\mathcal{N}$ is an $F$-stable complex and let $\gamma : F^*A \to A$ be an isomorphism. We then have $\phi_{\spr}^*A \in \mathscr{D}\mathcal{U}$ is a complex and $\phi_{\spr}^*\gamma$ defines an isomorphism
\begin{equation*}
F^*\phi_{\spr}^*A = \phi_{\spr}^*F^*A \to \phi_{\spr}^*A
\end{equation*}
because $F$ and $\phi_{\spr}$ commute. This has the following effect at the level of characteristic functions
\begin{equation*}
\chi_{\phi_{\spr}^*A,\phi_{\spr}^*\gamma} = \chi_{A,\gamma}\circ\phi_{\spr}.
\end{equation*}
Applying this to the functions $X_{\iota}$ and $Y_{\iota}$ we may easily translate the statement in \cref{thm:main-theorem-nilpotent} to the following statement about $\Gamma_u$. For ease, we simply write $X_{\iota}$ for the unipotently supported class function corresponding to the nilpotently supported function.
\end{pa}

\begin{thm}\label{thm:main-theorem-unipotent}
Recall that $\bG$ is any connected reductive algebraic group with Frobenius endomorphism $F : \bG \to \bG$ and $p$ is an acceptable prime for $\bG$. Let $\Gamma_{u,\mathscr{I}} : G \to \Ql$ denote the function
\begin{equation*}
\zeta_{\mathscr{I}}^{-1}\sum_{\iota,\iota',\iota_1 \in \mathscr{I}[\bL,\iota_0]^F} q^{f'(\iota,\iota_1)}\frac{1}{|W_{\bG}(\bL)|}\sum_{w \in W_{\bG}(\bL)}\Tr(wF,E_{\iota})\Tr(wF,E_{\hat{\iota}_1})|Z^{\circ}(\bL_w)^F|P_{\iota',\iota}^{\star}\overline{Y_{\iota'}(u^*)}\varepsilon_{\iota_1}X_{\iota_1}
\end{equation*}
where $f'(\iota,\iota_1)$ is as in \cref{thm:main-theorem-nilpotent} and $u^* = \phi_{\spr}(-e^{\dag})$. Then
\begin{equation*}
\Gamma_u = \sum_{\mathscr{I}} \Gamma_{u,\mathscr{I}},
\end{equation*}
where the sum is taken over all $F$-stable blocks.
\end{thm}

\begin{pa}
In \cite{lusztig:1992:a-unipotent-support} Lusztig slightly modified the GGGRs to obtain a new basis for the space $\Centu(\bG^F)$ of unipotently supported class functions, which tends to be more convenient than the GGGRs themselves. This is done as follows. Let $\mathcal{O} \subseteq \mathcal{U}$ be an $F$-stable unipotent conjugacy class of $\bG$ then we denote by $\{u_1,\dots,u_r\} \subseteq \mathcal{O}^F$ a set of representatives for the $\bG^F$-classes contained in $\mathcal{O}^F$. We then define
\begin{equation}\label{eq:mod-GGGR}
\Gamma_{\iota^*} = \sum_{i=1}^r [A_{\bG}(u_i) : A_{\bG}(u_i)^F]Y_{\iota^{\star}}(u_i)\Gamma_{u_i},
\end{equation}
for any $\iota^* \in \mathcal{V}_{\bG}^{\uni}$ with $\mathcal{O}_{\iota^*} = \mathcal{O}$. We now get the following expression for $\Gamma_{\iota^*}$.
\end{pa}

\begin{lem}\label{lem:mod-GGGR-decomp}
Recall the assumptions of \cref{thm:main-theorem-unipotent}. Let $\mathscr{I} \subseteq \mathcal{V}_{\bG}^{\uni}$ be the block containing $\iota^*$ then we have
\begin{equation*}
\Gamma_{\iota^*} = \zeta_{\mathscr{I}}^{-1}\sum_{\iota,\iota_1 \in \mathscr{I}^F} q^{f'(\iota,\iota_1)}\frac{|A_{\bG}(u)|}{|W_{\bG}(\bL)|}\sum_{w \in W_{\bG}(\bL)}\Tr(wF,E_{\iota})\Tr(wF,E_{\hat{\iota}_1})|Z^{\circ}(\bL_w)^F|P_{\iota^*,\iota}^{\star}\varepsilon_{\iota_1}X_{\iota_1}
\end{equation*}
where $f'(\iota,\iota_1)$ is as in \cref{thm:main-theorem-nilpotent} and $u \in \mathcal{O}_{\iota^*}$.
\end{lem}

\begin{proof}
Using \cref{thm:main-theorem-unipotent} and the definition of $\Gamma_{\iota^*}$ we see this follows from \cref{prop:same-G-orbit}, the definition of the functions $Y_{\iota}$ and the orthogonality relations given in \cref{lem:orthogonality-relations}.
\end{proof}

\section{Weyl Groups and Unipotent Classes}\label{sec:weyl-uni}
\begin{assumption}
We now drop the assumptions introduced from \cref{sec:GGGRs} onwards. From this point forward, unless stated otherwise, we simply assume that $p$ is a good prime for $\bG$.
\end{assumption}

\begin{pa}\label{pa:dual-quadruple}
We will denote by $\mathcal{S}(\bT_0)$ the set of (isomorphism classes of) tame local systems on $\bT_0$. Assume that $\mathscr{L} \in \mathcal{S}(\bT_0)$ is such a local system then we denote by $W_{\bG}(\mathscr{L})$ the quotient group $N_{\bG}(\bT_0,\mathscr{L})/\bT_0$ where
\begin{equation*}
N_{\bG}(\bT_0,\mathscr{L}) = \{x \in N_{\bG}(\bT_0) \mid (\Inn x)^*(\mathscr{L}) \cong \mathscr{L}\}.
\end{equation*}
For any root $\alpha \in \Phi$ we denote by $s_{\alpha} \in W_{\bG}(\bT_0)$ the reflection of $\alpha$ in the natural action of $W_{\bG}(\bT_0)$ on $X(\bT_0) \otimes_{\mathbb{Z}} \mathbb{R}$. Following \cite[2.3]{lusztig:1985:character-sheaves-I} we set
\begin{equation*}
\Phi_{\mathscr{L}} = \{\alpha \in \Phi \mid s_{\alpha} \in W_{\bG}(\mathscr{L})\}.
\end{equation*}
This is a root system, which is not necessarily additively closed in $\Phi$, and $W_{\bG}^{\circ}(\mathscr{L}) = \langle s_{\alpha} \mid \alpha \in \Phi_{\mathscr{L}} \rangle$ is the corresponding reflection group. In fact, this is a large normal subgroup of $W_{\bG}(\mathscr{L})$, which can of course be the whole of $W_{\bG}(\mathscr{L})$. Setting $\Phi_{\mathscr{L}}^+ = \Phi_{\mathscr{L}} \cap \Phi^+$ we get a system of positive roots for $\Phi_{\mathscr{L}}$ which determines a corresponding set of Coxeter generators $\mathbb{T}_{\mathscr{L}}$ for $W_{\bG}^{\circ}(\mathscr{L})$. Denoting by $\mathcal{A}_{\bG}(\mathscr{L})$ the group $\{w \in W_{\bG}(\mathscr{L}) \mid w\mathbb{T}_{\mathscr{L}}w^{-1} = \mathbb{T}_{\mathscr{L}}\}$ we obtain a semidirect product decomposition $W_{\bG}(\mathscr{L}) = W_{\bG}^{\circ}(\mathscr{L}) \rtimes \mathcal{A}_{\bG}(\mathscr{L})$. Let us now set $W = W_{\bG}(\bT_0)$, $H = W_{\bG}(\mathscr{L})$, $H^{\circ} = W_{\bG}^{\circ}(\mathscr{L})$ and $\Omega = \mathcal{A}_{\bG}(\mathscr{L})$ then following \cite[16.3]{lusztig:1985:character-sheaves} we make the following definition.
\end{pa}

\begin{definition}
A subset $\mathfrak{C} \subseteq H$ is called a \emph{two-sided cell} if there exists a two-sided cell $\mathfrak{C}^{\circ} \subseteq H^{\circ}$, defined as in \cite[\S5.1]{lusztig:1984:characters-of-reductive-groups}, such that $\mathfrak{C} = \Omega\mathfrak{C}^{\circ}\Omega$.
\end{definition}

\begin{pa}\label{pa:two-sided-cells-families}
The two-sided cells of $H$ are naturally in bijection with the $\Omega$-orbits of the two-sided cells of $H^{\circ}$. In particular, we have $H$ is a disjoint union of its two-sided cells. Using the leading coefficients of representations of the corresponding extended Hecke algebra Lusztig associates to each two-sided cell $\mathfrak{C} \subseteq H$ a subset of irreducible representations $\Irr(H\mid \mathfrak{C})$ of $H$, see \cite[10.4]{lusztig:1992:a-unipotent-support}. We will call these subsets \emph{families}. In each family he identifies a unique irreducible representation $E_{\mathfrak{C}} \in \Irr(H \mid \mathfrak{C})$, called a \emph{special representation}, which has the following property. Let $\{\mathfrak{C}_1^{\circ},\dots,\mathfrak{C}_r^{\circ}\}$ be the $\Omega$-orbit of two-sided cells of $H^{\circ}$ corresponding to $\mathfrak{C}$ then
\begin{equation}\label{eq:restrict-special}
\Res_{H^{\circ}}^H(E_{\mathfrak{C}}) = E_{\mathfrak{C}_1^{\circ}} + \cdots + E_{\mathfrak{C}_r^{\circ}},
\end{equation}
where $E_{\mathfrak{C}_i^{\circ}} \in \Irr(H^{\circ})$ is the unique special representation corresponding to $\mathfrak{C}_i^{\circ}$, in the sense of \cite[(4.1.9)]{lusztig:1984:characters-of-reductive-groups}.
\end{pa}

\begin{pa}
Let us denote by $\sgn$ the sign representation of $W$ and also its restriction to any subgroup of $W$. Now, given any two-sided cell $\mathfrak{C}$ of $H$ we denote by $\mathfrak{C}^{\dag}$ the unique two-sided cell of $H$ satisfying the condition:
\begin{itemize}
	\item For any representation $E \in \Irr(H \mid \mathfrak{C})$ we have $E \otimes \sgn \in \Irr(H \mid \mathfrak{C}^{\dag})$.
\end{itemize}
The map $\mathfrak{C} \to \mathfrak{C}^{\dag}$ defines a permutation on the set of two-sided cells of $H$.
\end{pa}

\begin{pa}
Let $V$ be the natural module for $W$ then for each representation $E \in \Irr(H^{\circ})$ we denote by $b_E$ the minimal $i \in \mathbb{Z}$ such that $E$ occurs in the $i$th symmetric power $S^i(U)$, where $U = V/\Fix_{H^{\circ}}(V)$ and $\Fix_{H^{\circ}}(V) = \{v \in V \mid h\cdot v = v$ for all $h \in H\}$. Note that taking $H^{\circ} = W$ in this definition we again obtain an integer $b_E$ for every $E \in \Irr(W)$. With this in hand we have the Lusztig--Macdonald--Spaltenstein induction map $j_{H^{\circ}}^W : \Irr(H^{\circ}) \to \Cent(W)$, which is defined in the following way. Let $E \in \Irr(H^{\circ})$ be an irreducible representation then we set
\begin{equation*}
j_{H^{\circ}}^W(E) = \sum_{\substack{E' \in \Irr(W)\\ b_E = b_{E'}}} \langle \Ind_{H^{\circ}}^W(E), E' \rangle E'.
\end{equation*}
\end{pa}

\begin{pa}
According to \cite[1.3]{lusztig:2009:unipotent-classes-and-special-Weyl} we have for each two sided cell $\mathfrak{C}^{\circ} \subseteq H^{\circ}$ that the $j$-induced representation $j_{H^{\circ}}^W(E_{\mathfrak{C}^{\circ}})$ is irreducible. Furthermore, by \cite[1.5(a)]{lusztig:2009:unipotent-classes-and-special-Weyl} we see that $j_{H^{\circ}}^W(E_{\mathfrak{C}^{\circ}})$ corresponds under the Springer correspondence to some $E_{\iota} \in \Irr(W)$ with $\iota = (\mathcal{O},\Ql) \in \mathscr{I}[\bT_0,\Ql] \subseteq \mathcal{V}_{\bG}^{\uni}$. We denote by $\mathcal{O}_{H^{\circ},\mathfrak{C}^{\circ}}$ the unipotent class $\mathcal{O}$. Now assume $\mathfrak{C} \subseteq H$ is a two-sided cell of $H$ and $\mathfrak{C}^{\circ} \subseteq H^{\circ}$ is a two-sided cell such that $\mathfrak{C} = \Omega\mathfrak{C}^{\circ}\Omega$ then we denote by $\mathcal{O}_{\mathscr{L},\mathfrak{C}}$ the class $\mathcal{O}_{H^{\circ},\mathfrak{C}^{\circ}}$; recall that $H = W_{\bG}(\mathscr{L})$. As in \cite{lusztig:1992:a-unipotent-support} we have defined a map
\begin{equation}\label{eq:cells-to-classes}
\begin{aligned}
\{\text{two-sided cells of }H\} &\to \{\text{unipotent conjugacy classes of }\bG\}\\
\mathfrak{C} &\mapsto \mathcal{O}_{\mathscr{L},\mathfrak{C}}.
\end{aligned}
\end{equation}
Note that, by \cref{eq:restrict-special}, we see that this map is well defined as $\mathcal{O}_{\mathscr{L},\mathfrak{C}}$ does not depend upon the choice of two-sided cell $\mathfrak{C}^{\circ} \subseteq H^{\circ}$ satisfying $\mathfrak{C} = \Omega\mathfrak{C}^{\circ}\Omega$.
\end{pa}

\begin{assumption}
We now assume that $(\bG^{\star},\bT_0^{\star},F^{\star})$ is a fixed triple dual to $(\bG,\bT_0,F)$.
\end{assumption}

\begin{pa}\label{pa:dual-tori-iso}
As in \cite[18.A]{bonnafe:2006:sln} we may choose an isomorphism $\bT_0^{\star} \cong \mathcal{S}(\bT_0)$ which respects the actions of $W_{\bG}(\bT_0)$, $W_{\bG^{\star}}(\bT_0^{\star})$, $F$ and $F^{\star}$. Recall that duality induces an anti-isomorphism ${}^{\star} : W_{\bG}(\bT_0) \to W_{\bG^{\star}}(\bT_0^{\star})$. Now if $\mathscr{L} \in \mathcal{S}(\bT_0)$ corresponds to $s \in \bT_0^{\star}$ under our chosen isomorphism then we have $W_{\bG}(\mathscr{L})^{\star} = W_{\bG^{\star}}(s)$ and $W_{\bG}^{\circ}(\mathscr{L})^{\star} = W_{\bG^{\star}}^{\circ}(s)$ where
\begin{equation*}
W_{\bG^{\star}}(s) = N_{\bG^{\star}}(\bT_0^{\star},s)/\bT_0^{\star}\qquad\text{and}\qquad W_{\bG^{\star}}^{\circ}(s) = N_{C_{\bG^{\star}}^{\circ}(s)}(\bT_0^{\star})/\bT_0^{\star}.
\end{equation*}
This anti-isomorphism also gives an identification of the 2-sided cells of these groups. In particular, if $\mathfrak{C} \subseteq W_{\bG^{\star}}(s)$ is a 2-sided cell we denote by $\mathcal{O}_{s,\mathfrak{C}}$ the class $\mathcal{O}_{\mathscr{L},\mathfrak{D}}$, where $\mathfrak{D} \subseteq W_{\bG}(\mathscr{L})$ is the unique 2-sided cell such that $\mathfrak{D}^{\star} = \mathfrak{C}$. We thus get a map
\begin{equation*}
\begin{aligned}
\{\text{two-sided cells of }W_{\bG^{\star}}(s)\} &\to \{\text{unipotent conjugacy classes of }\bG\}\\
\mathfrak{C} &\mapsto \mathcal{O}_{s,\mathfrak{C}}.
\end{aligned}
\end{equation*}
If $Z(\bG)$ is connected then $C_{\bG^{\star}}(s)$ is connected and this map coincides with the map defined in \cite[\S13.3]{lusztig:1984:characters-of-reductive-groups}.
\end{pa}

\section{Unipotent Supports for Character Sheaves}
\begin{pa}
For any local system $\mathscr{L} \in \mathcal{S}(\bT_0)$ Lusztig has defined in \cite[2.10]{lusztig:1985:character-sheaves} a set of (isomorphism classes of) perverse sheaves $\widehat{\bG}_{\mathscr{L}}$ on $\bG$. The definition of this set depends only on the $W_{\bG}(\bT_0)$-orbit of $\mathscr{L}$. With this we have
\begin{equation*}
\widehat{\bG} = \bigsqcup_{\mathscr{L} \in \mathcal{S}(\bT_0)/W_{\bG}(\bT_0)} \widehat{\bG}_{\mathscr{L}}
\end{equation*}
is the set of character sheaves of $\bG$, where the union is over the $W_{\bG}(\bT_0)$-orbits of tame local systems on $\bT_0$. The set $\widehat{\bG}_{\mathscr{L}}$ is then further partitioned into \emph{families}
\begin{equation*}
\widehat{\bG}_{\mathscr{L}} = \bigsqcup_{\mathfrak{C} \subseteq W_{\bG}(\mathscr{L})} \widehat{\bG}_{\mathscr{L},\mathfrak{C}}
\end{equation*}
where the union runs over all the two-sided cells of $W_{\bG}(\mathscr{L})$, see \cite[Corollary 16.7]{lusztig:1985:character-sheaves}.
\end{pa}

\begin{definition}\label{def:unipotent-support-char-sheaf}
For any conjugacy class $\mathcal{C}$ of $\bG$ let $\mathcal{C}_{\uni}$ be the set of all unipotent elements occuring in the Jordan decomposition of some element in $\mathcal{C}$; it is a unipotent conjugacy class of $\bG$. If $A \in \widehat{\bG}_{\mathscr{L},\mathfrak{C}}$ is a character sheaf then we say $\mathcal{O} \in \mathcal{U}/\bG$ is a \emph{unipotent support} of $A$ if following properties hold:
\begin{enumerate}
	\item For any conjugacy class $\mathcal{C}$ of $\bG$ we have $A|_{\mathcal{C}} \neq 0$ implies that either $\dim\mathcal{C}_{\uni} < \dim\mathcal{O}$ or $\mathcal{C}_{\uni} = \mathcal{O}$.
	\item There exists a conjugacy class $\mathcal{C}$ of $\bG$ and a character sheaf $A' \in \widehat{\bG}_{\mathscr{L},\mathfrak{C}}$ such that $\mathcal{C}_{\uni} = \mathcal{O}$ and $A'|_{\mathcal{C}} \neq 0$.
\end{enumerate}
\end{definition}

\begin{pa}
It is clear that every character sheaf has a unipotent support. What we would like to now show is that every character sheaf has a unique unipotent support. Before doing so we recall the following result of Shoji.
\end{pa}

\begin{thm}[{}{Shoji, \cite[Theorem 4.2]{shoji:1996:on-the-computation}}]\label{thm:shoji-gen-green-funcs}
Assume that $p$ is good for $\bG$ and that $Z(\bG)$ is connected then the results of \cite{lusztig:1990:green-functions-and-character-sheaves} are true without restriction on $q$.
\end{thm}

\begin{pa}
We will denote by $q_0(\bG) \geqslant 1$ a constant, as in \cite{lusztig:1990:green-functions-and-character-sheaves}, such that if $q > q_0(\bG)$ the results of \cite{lusztig:1990:green-functions-and-character-sheaves} are true. By \cref{thm:shoji-gen-green-funcs} we may set $q_0(\bG) = 1$ if $Z(\bG)$ is connected. Now, combining \cref{thm:main-theorem-nilpotent} with a careful reading of the remaining parts of \cite{lusztig:1992:a-unipotent-support} we see that the following is true.
\end{pa}

\begin{cor}\label{cor:all-Lusztig-results}
Assume that $p$ is an acceptable prime for $\bG$ and $q > q_0(\bG)$, recall that this means no assumption on $q$ if $Z(\bG)$ is connected. Then after replacing $\log$ and $\exp$ by $\phi_{\spr}$ and $\phi_{\spr}^{-1}$ we see that all the results of \cite[\S8 -- \S11]{lusztig:1992:a-unipotent-support} are true in our more general situation.
\end{cor}

\begin{pa}
With this in hand we may prove the following theorem, which is due to Lusztig assuming that $p$ is sufficiently large and Aubert in certain special cases when $p$ is good, see \cite{lusztig:1992:a-unipotent-support} and \cite{aubert:2003:characters-sheaves-and-generalized}.
\end{pa}

\begin{thm}\label{thm:unip-support-char-sheaves}
Assume $p$ is a good prime for $\bG$ then every character sheaf $A \in \widehat{\bG}$ has a unique unipotent support denoted by $\mathcal{O}_A$. Furthermore, if $A$ is contained in the family $\widehat{\bG}_{\mathscr{L},\mathfrak{C}}$ then $\mathcal{O}_A = \mathcal{O}_{\mathscr{L},\mathfrak{C}}$.
\end{thm}

\begin{pa}\label{pa:strategy}
Our strategy for proving \cref{thm:unip-support-char-sheaves} will be to reduce the problem to the case where $\bG$ is simple and simply connected. If $\bG$ is not of type $\A_n$ then any good prime is an acceptable prime for $\bG$. In particular, by \cref{cor:all-Lusztig-results} we have \cref{thm:unip-support-char-sheaves} follows from \cite[Theorem 10.7]{lusztig:1992:a-unipotent-support} in this case. However, if $\bG$ is of type $\A_n$ then we require special arguments to obtain \cref{thm:unip-support-char-sheaves}. Our idea for this case is to characterise whether a character sheaf vanishes at a conjugacy class in terms of combinatorial data which is, in a suitable sense, independent of $p$. We then use the validity of the result in the case where $p$ is large to deduce the case for a general prime.
\end{pa}

\begin{prop}\label{prop:unipotent-supports-SLn}
Assume $\bG$ is simple and simply connected of type $\A_n$ then \cref{thm:unip-support-char-sheaves} holds for $\bG$.
\end{prop}

\begin{proof}
For any element $z \in Z(\bG)$ let $t_z : \bG \to \bG$ be the morphism given by $t_z(g) = zg$. Now given any character sheaf $A \in \widehat{\bG}_{\mathscr{L},\mathfrak{C}}$ there exists an element $z \in Z(\bG)$ and a character sheaf $A' \in \widehat{\bG}_{\mathscr{L},\mathfrak{C}}$ such that $A = t_z^*A'$ and $\supp(A') \subseteq \mathcal{U}$, see \cite[\S17.17]{lusztig:1986:character-sheaves-IV}, the proof of \cite[Proposition 18.5]{lusztig:1986:character-sheaves-IV} and \cite[2.9]{lusztig:1986:on-the-character-values}. In particular, we have
\begin{equation*}
A|_{\mathcal{C}} = A'|_{z\mathcal{C}}
\end{equation*}
for any conjugacy class $\mathcal{C} \subseteq \bG$. Clearly if $A|_{\mathcal{C}} \neq 0$ then we must have $\mathcal{C} = z^{-1}\mathcal{C}_{\uni}$. From this we see that a unipotent conjugacy class is a unipotent support of $A'$ if and only if it is a unipotent support of $A$. Thus we may and will assume that $A$ is such that $\supp(A) \subseteq \mathcal{U}$.

Now, there exists a parabolic subgroup $\bQ \leqslant \bG$ and Levi complement $\bM \leqslant \bQ$ such that $\bM$ supports a cuspidal character sheaf $A_0 \in \widehat{\bM}$ and $A$ is a summand of the induced complex $\ind_{\bM \subseteq \bQ}^{\bG}(A_0)$, see \cite[Theorem 4.4(a)]{lusztig:1985:character-sheaves-I}. Replacing $A$ by an isomorphic character sheaf we may assume that $\bM \leqslant \bQ$ are standard in the sense that $\bT_0 \leqslant \bB_0 \leqslant \bQ$ and $\bT_0 \leqslant \bM$. Furthermore, we must have that $A_0 \in \widehat{\bM}_{\mathscr{L}}$ and $\supp(A_0) \subseteq \mathcal{U} \cap \bM$, see \cite[Proposition 4.8(b)]{lusztig:1985:character-sheaves-I} and \cite[2.9]{lusztig:1986:on-the-character-values}. In fact, we have
\begin{equation}\label{eq:cuspidal-char-sheaf}
A_0 = \IC(\overline{\mathcal{O}_0}Z^{\circ}(\bM),\mathscr{E}_0 \boxtimes \mathscr{F})[\dim\mathcal{O}_0 + \dim Z^{\circ}(\bM)]
\end{equation}
where $\iota_0 = (\mathcal{O}_0,\mathscr{E}_0) \in \mathcal{V}_{\bM}^{\uni}$ is a cuspidal pair and $\mathscr{F} \in \mathcal{S}(Z^{\circ}(\bM))$ is a tame local system.

Replacing $\mathscr{F}$ by $\Ql$ in \cref{eq:cuspidal-char-sheaf} we obtain a new cuspidal unipotently supported character sheaf which we denote by $A_1$. According to \cite[2.4]{lusztig:1986:on-the-character-values} there is a canonical parameterisation of the simple summands of $\ind_{\bM \subseteq \bQ}^{\bG}(A_0)$, resp., $\ind_{\bM \subseteq \bQ}^{\bG}(A_1)$, by the set of simple modules of $\Ql W_{\bG}(\bM,\mathscr{F})$, resp., $\Ql W_{\bG}(\bM)$. Note these groups are defined using the obvious generalisation of the construction in \cref{pa:dual-quadruple}. Let us assume $A = A_E$ is parameterised by $E \in \Irr(W_{\bG}(\bM,\mathscr{F}))$ then according to \cite[2.4]{lusztig:1986:on-the-character-values} we have
\begin{equation}\label{eq:char-sheaf-restrict}
A_E|_{\mathcal{U}} \cong \bigoplus_{\iota \in \mathscr{I}[\bM,\iota_0]} (K_{\iota}|_{\mathcal{U}})^{\oplus m_{E,\iota}}.
\end{equation}
Here $E_{\iota} \in \Irr(W_{\bG}(\bM))$ is the simple module corresponding to $\iota$ under the generalised Springer correspondence, $K_{\iota}$ is the summand of $\ind_{\bM \subseteq \bQ}^{\bG}(A_1)$ parameterised by $E_{\iota}$ and $m_{E,\iota} = \langle \Ind_{W_{\bG}(\bM,\mathscr{F})}^{W_{\bG}(\bM)}(E), E_{\iota} \rangle$. The $K_{\iota}$ described here are the group analogues of the perverse sheaves described in \cref{pa:gen-spring-corr}.

Let us denote by $\bG'$ a simple simply connected algebraic group again of type $\A_n$, the same $n$, defined over an algebraic closure of the finite field $\mathbb{F}_{p'}$ where $p' > n$ and $p' \neq \ell$, i.e., $p'$ is a very good prime for $\bG'$. We fix a maximal torus and Borel subgroup $\bT_0' \leqslant \bB_0' \leqslant \bG'$ then we assume $\bM' \leqslant \bQ' \leqslant \bG'$ are the standard parabolic subgroup and Levi complement naturally in correspondence with $\bM \leqslant \bQ$, i.e., they are determined by the same set of simple roots. Note that by \cref{lem:p-acceptable}, \cref{cor:all-Lusztig-results} and \cite[Theorem 10.7]{lusztig:1992:a-unipotent-support} we have \cref{thm:unip-support-char-sheaves} holds for $\bG'$.

Recall from \cite[\S5]{lusztig-spaltenstein:1985:on-the-generalised-springer-correspondence} that the component group $\mathcal{Z}(\bM) = Z(\bM)/Z^{\circ}(\bM)$ of the centre of $\bM$ acts on the local system $\mathscr{E}_0$ by a faithful irreducible character $\varphi \in \Irr(\mathcal{Z}(\bM))$. By this we mean that $\mathcal{Z}(\bM)$ acts on each stalk of the local system by multiplication with $\varphi$. This irreducible character characterises the cuspidal pair on $\bM$. The component group $\mathcal{Z}(\bM')$ is isomorphic to $\mathcal{Z}(\bM)$ and fixing an isomorphism we may identify the sets of irreducible characters $\Irr(\mathcal{Z}(\bM))$ and $\Irr(\mathcal{Z}(\bM'))$. We denote by $\varphi' \in \Irr(\mathcal{Z}(\bM'))$ the faithful irreducible character corresponding to $\varphi$ in this way.

As for $\bM$ we see that $\bM'$ then admits a cuspidal pair $\iota_0' = (\mathcal{O}_0',\mathscr{E}_0') \in \mathcal{V}_{\bM'}^{\uni}$ such that $\mathcal{Z}(\bM')$ acts on $\mathscr{E}_0'$ by the irreducible character $\varphi'$. Let $m \in \mathbb{Z}_{\geqslant 1}$ be the minimum integer such that $\mathscr{F}^{\otimes m} \cong \Ql$ then choosing $p' > m$, as well as maintaining our previous assumptions, we can find a tame local system $\mathscr{F}' \in \mathcal{S}(Z^{\circ}(\bM'))$ so that we have $W_{\bG}(\bM,\mathscr{F}) \cong W_{\bG}(\bM,\mathscr{F}')$. Using this data in \cref{eq:cuspidal-char-sheaf} we thus obtain a cuspidal character sheaf $A_0' \in \widehat{\bM}'$. The relative Weyl groups $W_{\bG}(\bM)$ and $W_{\bG}(\bM')$ are isomorphic so choosing an isomorphism we may identify the sets of irreducible characters. Assume $E' \in \Irr(W_{\bG}(\bM'))$ corresponds to $E \in \Irr(W_{\bG}(\bM))$ in this way then we have a corresponding summand $A_{E'}$ of the induced complex $\ind_{\bM'\subseteq \bQ'}^{\bG'}(A_0')$. Note that, as above, we may assume $A_0' \in \widehat{\bG}_{\mathscr{L}',\mathfrak{C}'}$ where $\mathscr{L}' \in \mathcal{S}(\bT_0')$ is a tame local system such that $W_{\bG}(\bT_0,\mathscr{L}) \cong W_{\bG'}(\bT_0',\mathscr{L}')$ and the two-sided cell $\mathfrak{C}'$ is identified with $\mathfrak{C}$.

We are now in a position to prove the proposition. We will identify the unipotent classes of $\bG$ and $\bG'$ in the obvious way. Let $\widetilde{\mathcal{O}}_A$ be a unipotent conjugacy class such that $A_E|_{\widetilde{\mathcal{O}}_A} \neq 0$ and $\widetilde{\mathcal{O}}_A$ has maximal dimension amongst all classes with this property. By \cref{eq:char-sheaf-restrict}, we see that any such class is obtained as a class $\mathcal{O}_{\iota}$ where $\iota \in \mathscr{I}[\bM,\iota_0]$ satisfies the property
\begin{equation*}
m_{E,\iota'} \neq 0 \Rightarrow \dim \mathcal{O}_{\iota'} \leqslant \dim \mathcal{O}_{\iota}.
\end{equation*}
In particular, we see that the restriction of this character sheaf to $\widetilde{\mathcal{O}}_A$ is characterised in terms of data which is, in a suitable sense, independent of $p$. Now applying this argument in $\bG'$ we see that $A_{E'}|_{\widetilde{\mathcal{O}}_A} \neq 0$. As \cref{thm:unip-support-char-sheaves} holds in $\bG'$ we can deduce that either $\dim \widetilde{\mathcal{O}}_A < \dim\mathcal{O}_{\mathscr{L},\mathfrak{C}}$ or $\widetilde{\mathcal{O}}_A = \mathcal{O}_{\mathscr{L},\mathfrak{C}}$.

To finish the argument it remains to find a character sheaf in $\widehat{\bG}_{\mathscr{L},\mathfrak{C}}$ whose restriction to $\mathcal{O}_{\mathscr{L},\mathfrak{C}}$ is non-zero. In the proof of \cite[Theorem 10.7]{lusztig:1992:a-unipotent-support} Lusztig constructs such a character sheaf for $\bG'$ which is obtained as above with $\bM' = \bT_0'$ and $\iota_0' = (\{1\},\Ql)$. This part of the generalised Springer correspondence exists in all characteristics. Thus, applying the same style of argument as before we see that such a character sheaf exists in $\widehat{\bG}_{\mathscr{L},\mathfrak{C}}$.
\end{proof}

\begin{proof}[of \cref{thm:unip-support-char-sheaves}]
We will now consider a series of reduction arguments as in \cite[\S17]{lusztig:1985:character-sheaves}. Let $\pi : \bG \to \overline{\bG} = \bG/Z^{\circ}(\bG)$ and $\sigma : \bG \to \bG/\bG_{\der}$ be the canonical quotient maps. We will denote by $\overline{\bT}_0$ the image of $\bT_0$ by $\pi$. Recall that any local system $\mathscr{L} \in \mathcal{S}(\bT_0)$ is of the form $\mathscr{F} \otimes \mathscr{E}$ where $\mathscr{F}$ is the inverse image under $\pi$ of a local system $\mathscr{F}' \in \mathcal{S}(\overline{\bT}_0)$ and $\mathscr{E}$ is the inverse image under $\sigma$ of a local system $\mathscr{E}_0 \in \mathcal{S}(\bG/\bG_{\der})$. Note that here we consider $\mathscr{E}$ as a local system on $\bT_0$ by restriction.

Assume $A \in \widehat{\bG}$ is contained in the series $\widehat{\bG}_{\mathscr{L}}$ then $A$ may be written as $\pi^*(\overline{A}) \otimes \mathscr{E}$ with $\overline{A} \in \widehat{\overline{\bG}}_{\mathscr{F}'}$. Note that $\pi$ induces a natural isomorphism $W_{\bG}(\bT_0,\mathscr{L}) \to W_{\overline{\bG}}(\overline{\bT}_0,\mathscr{F}')$, in particular we may identify the two-sided cells of these groups. In this way we see that $A \in \widehat{\bG}_{\mathscr{L},\mathfrak{C}}$ if and only if $\overline{A} \in \widehat{\overline{\bG}}_{\mathscr{F}',\mathfrak{C}}$. From this description it is clear that we have
\begin{equation*}
A|_{\mathcal{C}} \neq 0 \Leftrightarrow \overline{A}|_{\pi(\mathcal{C})} \neq 0
\end{equation*}
for any conjugacy class $\mathcal{C}$ of $\bG$. In particular, we may clearly assume that $\bG$ is semisimple.

If $\bG$ is semisimple then there exists a simply connected cover $\pi : \bG_{\simc} \to \bG$, which we assume fixed. Let $\bT_{\simc} \leqslant \bG_{\simc}$ be the unique maximal torus satisfying $\pi(\bT_{\simc}) = \bT_0$. Moreover, assume $A \in \widehat{\bG}_{\mathscr{L}}$ then we set $\mathscr{L}' \in \mathcal{S}(\bT_{\simc})$ to be the inverse image of $\mathscr{L}$ under $\pi$. Note that, identifying $W_{\bG}(\bT_0)$ and $W_{\bG_{\simc}}(\bT_{\simc})$ under the natural isomorphism we have $W_{\bG}^{\circ}(\mathscr{L}) = W_{\bG_{\simc}}^{\circ}(\mathscr{L}')$ and $W_{\bG}(\mathscr{L}) \subseteq W_{\bG_{\simc}}(\mathscr{L}')$. Furthermore, for any character sheaf $A \in \widehat{\bG}_{\mathscr{L},\mathfrak{C}}$ we have $\pi^*A = A_1 \oplus \cdots \oplus A_r$ is a direct sum of character sheaves with $A_i \in \widehat{\bG}_{\simc,\mathscr{L}',\mathfrak{C}'}$ where $\mathfrak{C}'\subseteq W_{\bG_{\simc}}(\mathscr{L}')$ is the unique two-sided cell containing $\mathfrak{C}$. From this it is clear that we may assume $\bG$ is simply connected.

Finally assume $\bG$ is simply connected then by \cite[\S17.11]{lusztig:1986:character-sheaves-IV} we may assume that $\bG$ is simple and simply connected. The result now follows from \cref{prop:unipotent-supports-SLn} and the remarks in \cref{pa:strategy}.
\end{proof}

\section{Wave Front Sets for Irreducible Characters}

\subsection{Families of Irreducible Characters}
\begin{pa}\label{pa:series-char-sheaves}
Recall that a family of character sheaves $\widehat{\bG}_{\mathscr{L},\mathfrak{C}}$ contains an $F$-stable character sheaf if and only if the $W_{\bG}(\bT_0)$-orbit of $(\mathscr{L},\mathfrak{C})$ is $F$-stable. Assume this is the case then as in \cite[11.1]{lusztig:1992:a-unipotent-support} we define a corresponding set of irreducible characters
\begin{equation*}
\mathcal{E}(G,\mathscr{L},\mathfrak{C}) = \{\rho \in \Irr(G) \mid \langle \chi_{A,\phi}, \rho \rangle \neq 0\text{ for some }A \in \widehat{\bG}_{\mathscr{L},\mathfrak{C}}^F\text{ with }\supp (A) = \bG\}.
\end{equation*}
Clearly this definition is independent of the choice of isomorphism $F^*A \to A$. We now wish to consider the relationship between this set and the usual notion of a Lusztig series.
\end{pa}

\begin{pa}\label{pa:loc-sys-irr-chars-correspondence}
Assume $(\bT,\mathscr{L})$ is a pair consisting of an $F$-stable maximal torus $\bT \leqslant \bG$ and a tame $F$-stable local system $\mathscr{L} \in \mathcal{S}(\bT)^F$. To this pair we have a corresponding $F$-stable complex $K_{\bT}^{\mathscr{L}} \in \mathscr{D}\bG$ defined as in \cite[I, 1.7]{shoji:1995:character-sheaves-and-almost-characters}; this is simply the complex obtained by inducing $\mathscr{L}$ to $\bG$. There is a unique isomorphism $\varphi : F^*\mathscr{L} \to \mathscr{L}$ such that the induced isomorphism over the stalk of the identity is the identity. Let $\chi_{\mathscr{L}}$ be the resulting characteristic function then we have $\chi_{\mathscr{L}}(1)$ is a positive integer. With this we have a bijection
\begin{equation}\label{eq:loc-sys-irr-char}
\begin{aligned}
\mathcal{S}(\bT)^F &\to \Irr(\bT^F)\\
\mathscr{L} &\mapsto \chi_{\mathscr{L}}.
\end{aligned}
\end{equation}
The isomorphism $\varphi$ chosen above naturally induces an isomorphism $\phi : F^*K_{\bT}^{\mathscr{L}} \to K_{\bT}^{\mathscr{L}}$. We will denote by $\chi_{K_{\bT}^{\mathscr{L}}}$ the resulting characteristic function determined by $\phi$. By \cite[I, Corollary 2.3]{shoji:1995:character-sheaves-and-almost-characters} we then have
\begin{equation}\label{eq:DL-chars}
\chi_{K_{\bT}^{\mathscr{L}}} = (-1)^{\dim\bT}R_{\bT}^{\bG}(\chi_{\mathscr{L}}),
\end{equation}
where $R_{\bT}^{\bG}(\chi_{\mathscr{L}})$ is the corresponding Deligne--Lusztig virtual character, as defined in \cite{deligne-lusztig:1976:representations-of-reductive-groups}.
\end{pa}

\begin{pa}\label{pa:char-func-of-char-sheaf}
Now assume the series $\widehat{\bG}_{\mathscr{L},\mathfrak{C}}$ contains an $F$-stable character sheaf then the $W_{\bG}(\bT_0)$-orbit of $(\mathscr{L},\mathfrak{C})$ is $F$-stable and so the set
\begin{equation*}
Z_{\bG}(\mathscr{L}) = \{n \in N_{\bG}(\bT_0) \mid (\Inn n)^*F^*\mathscr{L} \cong \mathscr{L}\}/\bT_0 \subseteq W_{\bG}(\bT_0)
\end{equation*}
is non-empty. It is easy to see that $Z_{\bG}(\mathscr{L})$ is a coset of $W_{\bG}(\mathscr{L}) \leqslant W_{\bG}(\bT_0)$ and hence is a union of cosets of $W_{\bG}^{\circ}(\mathscr{L})$ in $W_{\bG}(\bT_0)$. Now for any element $x \in W_{\bG}(\bT_0)$ we assume fixed a representative $\dot{x} \in N_{\bG}(\bT_0)$ of $x$ and an element $g_x \in \bG$ such that $g_x^{-1}F(g_x) = F(\dot{x})$. Moreover we denote by $\bT_x$ the $F$-stable maximal torus $g_x\bT_0g_x^{-1}$. An easy calculation shows that the local system $\mathscr{L}_x := (\Inn g_x^{-1})^*\mathscr{L} \in \mathcal{S}(\bT_x)$ is $F$-stable if and only if $x \in Z_{\bG}(\mathscr{L})$. Hence, if $x \in Z_{\bG}(\mathscr{L})$ then the pair $(\bT_x,\mathscr{L}_x)$ gives rise to an $F$-stable complex $K_{\bT_x}^{\mathscr{L}_x}$ as in \cref{pa:loc-sys-irr-chars-correspondence}.

By \cite[Lemma 1.9(i)]{lusztig:1984:characters-of-reductive-groups} every coset of $W_{\bG}^{\circ}(\mathscr{L})$ in $Z_{\bG}(\mathscr{L})$ contains a unique element which stabilises the set $\Phi_{\mathscr{L}}^+$ of positive roots, c.f., \cref{pa:dual-quadruple}. We assume that $w_1 \in Z_{\bG}(\mathscr{L})$ is chosen to have this property. Now, if $A \in \widehat{\bG}_{\mathscr{L},\mathfrak{C}}^F$ is an $F$-stable character sheaf such that $\supp (A) = \bG$ then, up to isomorphism, we must have $A$ is a constituent of the complex $K_{\bT_{w_1}}^{\mathscr{L}_{w_1}}$, c.f., \cite[2.9]{lusztig:1986:on-the-character-values} and \cite[10.5]{lusztig:1985:character-sheaves}. Using the conjugation isomorphism $\Inn g_{w_1}$ we may identify $A$ with a summand of the $Fw_1$-stable complex $K_{\bT_0}^{\mathscr{L}}$ where $Fw_1$ denotes the Frobenius endomorphism $F\circ\Inn \dot{w}_1$ of $\bG$.

The endomorphism algebra $\End(K_{\bT_0}^{\mathscr{L}})$ is isomorphic to the group algebra $\Ql W_{\bG}(\mathscr{L})$ and so $A$ is indexed by an irreducible representation $E \in \Irr(W_{\bG}(\mathscr{L}))^{Fw_1}$. If we choose an extension $\widetilde{E}$ of $E$ to the semidirect product $W_{\bG}(\mathscr{L}) \rtimes \langle Fw_1 \rangle$ then this determines an isomorphism $\phi_A : (Fw_1)^*A \to A$, c.f., \cite[6.13]{taylor:2014:evaluating-characteristic-functions}. In \cite[6.15]{taylor:2014:evaluating-characteristic-functions}, see also \cite[10.6]{lusztig:1985:character-sheaves}, we have defined for any $w \in W_{\bG}(\mathscr{L})$ an isomorphism $\varphi_w : (Fw_1)^*\mathscr{L}_w \to \mathscr{L}_w$ from the canonical choice of isomorphism $\varphi : (Fw_1)^*\mathscr{L} \to \mathscr{L}$ made in \cref{pa:loc-sys-irr-chars-correspondence}. By \cite[Corollary 6.9]{bonnafe:2004:actions-of-rel-Weyl-grps-I} we see that $\varphi_w$ is again the canonical isomorphism considered in \cref{pa:loc-sys-irr-chars-correspondence}. Thus, \cite[10.4.5, 10.6.1]{lusztig:1985:character-sheaves} and \cite[I, 5.17.1]{shoji:1995:character-sheaves-and-almost-characters} together with \cref{eq:DL-chars} shows that
\begin{equation}\label{eq:decom-char-sheaf}
(-1)^{\dim\bT_0}\chi_{A,\phi_A} = \frac{1}{|W_{\bG}(\mathscr{L})|}\sum_{w \in W_{\bG}(\mathscr{L})} \Tr(Fw_1w,\widetilde{E})R_{\bT_{w_1w}}^{\bG}(\chi_{\mathscr{L}_{w_1w}}).
\end{equation}
If $Z(\bG)$ is connected then $W_{\bG}(\mathscr{L})$ is a Weyl group and we may assume that the extension $\widetilde{E}$ is defined over $\mathbb{Q}$, c.f., \cite[3.2]{lusztig:1984:characters-of-reductive-groups}. In this case the function in \cref{eq:decom-char-sheaf} is nothing other than the almost character defined by Lusztig in \cite[3.7.1]{lusztig:1984:characters-of-reductive-groups}.
\end{pa}

\begin{pa}\label{pa:sets-in-dual-grp}
Recall the dual triple $(\bG^{\star},\bT_0^{\star},F^{\star})$ fixed in \cref{sec:weyl-uni} and assume that $s \in \bT_0^{\star}$ corresponds to $\mathscr{L} \in \mathcal{S}(\bT_0)$ under the isomorphism in \cref{pa:dual-tori-iso}. As the $W_{\bG}(\bT_0)$-orbit of $(\mathscr{L},\mathfrak{C})$ is $F$-stable we must have the $W_{\bG^{\star}}(\bT_0^{\star})$-orbit of $(s,\mathfrak{C}^{\star})$ is $F^{\star}$-stable. Under the anti-isomorphism ${}^{\star} : W_{\bG}(\mathscr{L}) \to W_{\bG^{\star}}(s)$, c.f., \cref{pa:dual-tori-iso}, the set $Z_{\bG}(\mathscr{L})$ is identified with
\begin{equation*}
Z_{\bG^{\star}}(s) = \{n \in N_{\bG^{\star}}(\bT_0^{\star}) \mid {}^nF^{\star}(s) = s\}/\bT_0^{\star} \subseteq W_{\bG^{\star}}(\bT_0^{\star})
\end{equation*}
and the automorphism $Fw_1$ is identified with the automorphism $(w_1^{\star}F^{\star})^{-1}$. In particular, identifying $w_1$ with $w_1^{\star}$ this gives us an identification of the semidirect product $W_{\bG}(\mathscr{L}) \rtimes \langle Fw_1\rangle$ with $W_{\bG^{\star}}(s) \rtimes \langle w_1F^{\star}\rangle$. Following \cref{eq:decom-char-sheaf} we define for any extension $\widetilde{E}$ of $E \in \Irr(W_{\bG^{\star}}(s))^{w_1F^{\star}}$ a class function
\begin{equation*}
\mathcal{R}_{\bT_0^{\star}}^{\bG}(\widetilde{E},s) = \frac{1}{|W_{\bG^{\star}}(s)|}\sum_{w \in W_{\bG^{\star}}(s)} \Tr(ww_1F^{\star},\widetilde{E})R_{\bT_{w_1w}^{\star}}^{\bG}(s)
\end{equation*}
where $\bT_{w_1w}^{\star}$ is a torus dual to $\bT_{w_1w}$. Furthermore we define a set
\begin{equation*}
\mathcal{E}(G,s,\mathfrak{C}) = \{\rho \in \Irr(G) \mid \langle \mathcal{R}_{\bT_0^{\star}}^{\bG}(\widetilde{E},s),\rho \rangle \neq 0\text{ for some }E \in \Irr(W_{\bG^{\star}}(s)\mid\mathfrak{C})^{w_1F^{\star}}\}
\end{equation*}
where the extension $\widetilde{E}$ is chosen arbitrarily.

Now, as the $\bG^{\star}$-conjugacy class of $s$ is $F^{\star}$-stable we have the corresponding geometric Lusztig series $\mathcal{E}(G,s)$ is defined, see \cite[11.A]{bonnafe:2006:sln}. From the definitions and \cref{eq:decom-char-sheaf} we see that if $\mathfrak{D} \subseteq W_{\bG}(\mathscr{L})$ is the unique two-sided cell such that $\mathfrak{D}^{\star} = \mathfrak{C}$ then we have
\begin{equation*}
\mathcal{E}(G,\mathscr{L},\mathfrak{D}) = \mathcal{E}(G,s,\mathfrak{C}) \subseteq \mathcal{E}(G,s).
\end{equation*}
We now claim, as in \cite[11.1]{lusztig:1992:a-unipotent-support}, that we have a partition
\begin{equation*}
\mathcal{E}(G,s) = \bigsqcup_{\mathfrak{C} \subseteq W_{\bG^{\star}}(s)}\mathcal{E}(G,s,\mathfrak{C}) = \bigsqcup_{\mathfrak{C} \subseteq W_{\bG}(\mathscr{L})}\mathcal{E}(G,\mathscr{L},\mathfrak{C}),
\end{equation*}
where the first, resp., second, union is taken over all $F^{\star}$-stable, resp., $F$-stable, two-sided cells. If $Z(\bG)$ is connected then this follows from the disjointness theorem of Lusztig, c.f., \cite[6.17]{lusztig:1984:characters-of-reductive-groups}, together with the remarks at the end of \cref{pa:char-func-of-char-sheaf}. In particular, the sets $\mathcal{E}(G,\mathscr{L},\mathfrak{C})$ are nothing other than the families of characters considered in \cite{lusztig:1984:characters-of-reductive-groups}, c.f., \cite[Theorem 5.25]{lusztig:1984:characters-of-reductive-groups}.

Let us now deal with the case where $Z(\bG)$ is disconnected. Denote by $\mathcal{A}_{\bG^{\star}}(s)$ the image of $\mathcal{A}_{\bG}(\mathscr{L})$ under the anti-isomorphism ${}^{\star}$, c.f., \cref{pa:dual-quadruple}, then we have $W_{\bG^{\star}}(s) = W_{\bG^{\star}}^{\circ}(s) \rtimes \mathcal{A}_{\bG^{\star}}(s)$. For each $a \in \mathcal{A}_{\bG^{\star}}(s)$ we then define a function
\begin{equation*}
\mathcal{R}_{\bT_0^{\star}}^{\bG}(\widetilde{E},s,a) = \frac{1}{|W_{\bG^{\star}}^{\circ}(s)|}\sum_{w \in W_{\bG^{\star}}^{\circ}(s)} \Tr(waw_1F^{\star},\widetilde{E})R_{\bT_{w_1wa}^{\star}}^{\bG}(s)
\end{equation*}
with $\widetilde{E}$ as above. It is then clear that we have
\begin{equation*}
\mathcal{R}_{\bT_0^{\star}}^{\bG}(\widetilde{E},s) = \frac{1}{|\mathcal{A}_{\bG^{\star}}(s)|}\sum_{a \in \mathcal{A}_{\bG^{\star}}(s)} \mathcal{R}_{\bT_0^{\star}}^{\bG}(\widetilde{E},s,a).
\end{equation*}
Now assume $\iota : \bG \to \widetilde{\bG}$ is a regular embedding, as in \cite[\S7]{lusztig:1988:reductive-groups-with-a-disconnected-centre}, with $\iota^{\star} : \widetilde{\bG}^{\star} \to \bG^{\star}$ an induced surjective morphism between dual groups. Let $\widetilde{\bT}_0^{\star}$ be the preimage of $\bT_0^{\star}$ under $\iota^{\star}$ then this is a maximal torus of $\bG^{\star}$. As in \cite[2.3]{digne-michel:1990:lusztigs-parametrization} we choose an element $\widetilde{s} \in \widetilde{\bT}_0^{\star}$ such that $Z_{\widetilde{\bG}^{\star}}(\widetilde{s}) = W_{\bG^{\star}}^{\circ}(s)w_1$. Using the results in \cite[\S2, \S5]{digne-michel:1990:lusztigs-parametrization} we see that we may realise the functions $\mathcal{R}_{\bT_0^{\star}}^{\bG}(\widetilde{E},s,a)$ as the restriction of functions $\mathcal{R}_{\widetilde{\bT}_0^{\star}}^{\widetilde{\bG}}(\widetilde{E},\widetilde{s}z)$ for some $z \in \Ker(\iota^{\star})$; see \cite[2.5, 2.7]{digne-michel:1990:lusztigs-parametrization}. In particular, if $\mathfrak{C}$ is of the form $\mathcal{A}_{\bG^{\star}}(s)\mathfrak{C}^{\circ}\mathcal{A}_{\bG^{\star}}(s)$ with $\mathfrak{C}^{\circ} \subseteq W_{\bG^{\star}}^{\circ}(s)$ a two-sided cell then one may verify that we have
\begin{equation*}
\mathcal{E}(G,s,\mathfrak{C}) = \{\rho \in \Irr(G) \mid \langle \Res_G^{\widetilde{G}}(\widetilde{\rho}),\rho\rangle \neq 0\text{ for some }\widetilde{\rho} \in \mathcal{E}(\widetilde{G},\widetilde{s},\mathfrak{C}^{\circ})\},
\end{equation*}
where we identify $W_{\widetilde{\bG}^{\star}}(\widetilde{s})$ with $W_{\bG^{\star}}^{\circ}(s)$. We will not go into more details here but instead refer the reader to \cite[\S2, \S5]{digne-michel:1990:lusztigs-parametrization} and \cite[Chapitre 3]{bonnafe:2006:sln}.
\end{pa}

\subsection{Wave Front Sets}
\begin{definition}
Assume $\rho \in \Irr(G)$ is an irreducible character of $G$ and $\mathcal{O}$ is an $F$-stable unipotent conjugacy class of $\bG$. We say $\mathcal{O}$ is a \emph{wave front set} for $\rho$ if $\langle \Gamma_u,\rho\rangle \neq 0$ for some $u \in \mathcal{O}^F$ and $\mathcal{O}$ has maximal dimension amongst all unipotent classes with this property.
\end{definition}

\begin{pa}
The following result was conjectured to hold by Kawanaka in \cite[Conjecture 3.3.3]{kawanaka:1985:GGGRs-and-ennola-duality}. In \cite[Theorem 11.2]{lusztig:1992:a-unipotent-support} Lusztig proved Kawanaka's conjecture under the assumption that $p$ and $q$ are sufficiently large. Here we give the general case where $p$ is a good prime, thus completing the proof of Kawanaka's conjecture. Note that our proof uses in an essential way the ideas and techniques used by Lusztig in \cite[Theorem 11.2]{lusztig:1992:a-unipotent-support}.
\end{pa}

\begin{thm}\label{thm:existence-wave-front-sets}
Assume $p$ is a good prime for $\bG$ then every irreducible character $\rho \in \Irr(G)$ has a unique wave front set denoted by $\mathcal{O}_{\rho}^*$. Furthermore, if $\rho$ is contained in the series $\mathcal{E}(G,s,\mathfrak{C})$ then $\mathcal{O}_{\rho}^* = \mathcal{O}_{s,\mathfrak{C}^{\dag}}$, c.f., \cref{pa:dual-tori-iso}.
\end{thm}

\begin{pa}\label{pa:isotypic-morphism}
Before proving the theorem we will consider the following two reduction steps, which are similar to those used in \cite{geck:1996:on-the-average-values}. Note that the second reduction is only required because \cref{thm:main-theorem-unipotent} does not necessarily hold in good characteristic when $Z(\bG)$ is connected. In the following lemmas we will implicitly assume that $p$ is a good prime and the following fact. Assume $\varphi : \bG \to \bH$ is an isotypic morphism between connected reductive algebraic groups. In other words, the image of $\varphi$ contains the derived subgroup of $\bH$ and the kernel is contained in the centre of $\bG$. Then $\varphi$ induces a bijection between the unipotent conjugacy classes of $\bG$ and $\bH$ (see for instance \cite[Proposition 5.1.1]{carter:1993:finite-groups-of-lie-type}).
\end{pa}

\begin{lem}\label{lem:regular-embedding}
Assume $\bG \hookrightarrow \widetilde{\bG}$ is a regular embedding into a group with connected centre then \cref{thm:existence-wave-front-sets} holds for $G$ if and only if it holds for $\widetilde{G}$. Moreover, if \cref{thm:existence-wave-front-sets} holds, then for any irreducible characters $\widetilde{\rho} \in \Irr(\widetilde{G})$ and $\rho \in \Irr(G)$ satisfying $\langle \Res_G^{\widetilde{G}}(\widetilde{\rho}),\rho\rangle \neq 0$ we have $\mathcal{O}_{\widetilde{\rho}}^* = \mathcal{O}_{\rho}^*$.
\end{lem}

\begin{proof}
Let $u \in G$ be a unipotent element then by the definition of the GGGR it is clear to see that we have
\begin{equation*}
\Gamma_u^{\widetilde{G}} = \Ind_G^{\widetilde{G}}(\Gamma_u^G)
\end{equation*}
and all GGGRs of $\widetilde{G}$ are obtained in this way. Assume $\rho \in \Irr(\widetilde{G})$ is an irreducible character then according to \cite[Proposition 5.1]{lusztig:1988:reductive-groups-with-a-disconnected-centre} we have $\Res^{\widetilde{G}}_G(\rho) = \rho_1 + \cdots + \rho_r$ for some irreducible characters $\rho_i \in \Irr(G)$. In particular, by Frobenius reciprocity we have
\begin{equation*}
\langle\Gamma_u^{\widetilde{G}},\rho\rangle_{\widetilde{G}} = \sum_{i=1}^r \langle \Gamma_u^G,\rho_i\rangle_G.
\end{equation*}
As the right hand side is a sum of non-negative integers we have $\langle\Gamma_u^{\widetilde{G}},\rho\rangle_{\widetilde{G}} \neq 0$ if and only if $\langle \Gamma_u^G,\rho_i\rangle_G \neq 0$ for some $i \in \{1,\dots,r\}$. This proves that \cref{thm:existence-wave-front-sets} holds in $G$ if and only if it holds in $\widetilde{G}$ by the discussion in \cref{pa:sets-in-dual-grp} and the definition of the class $\mathcal{O}_{s,\mathfrak{C}^{\dag}}$. The last statement is clear.
\end{proof}

\begin{lem}\label{lem:connected-simply-connected}
Assume $\bG$ has a connected centre and let $\pi : \widetilde{\bG} \to \bG$ be a surjective isotypic morphism defined over $\mathbb{F}_q$ such that: $\Ker(\pi)$ is connected, $Z(\widetilde{\bG})$ is connected and the derived subgroup of $\widetilde{\bG}$ is simply connected. Then \cref{thm:existence-wave-front-sets} holds for $G$ if it holds for $\widetilde{G}$.
\end{lem}

\begin{proof}
Note that the map $\pi$ descends to a surjective map $\pi : \widetilde{G} \to G$ of $F$-fixed points as the kernel of $\pi$ is connected. We denote by $\Inf_G^{\widetilde{G}} : \Cent(G) \to \Cent(\widetilde{G})$ the inflation map given by $\Inf_G^{\widetilde{G}}(f) = f\circ\pi$. As this map induces an isometry onto its image we have
\begin{equation*}
\langle \Gamma_u^G,\chi \rangle_G = \langle \Inf_G^{\widetilde{G}}(\Gamma_u^G),\Inf_G^{\widetilde{G}}(\chi) \rangle_{\widetilde{G}}
\end{equation*}
for all unipotent elements $u \in G$ and irreducible characters $\chi \in \Irr(G)$. Let us identify $u \in G$ with the unique unipotent element in the preimage $\pi^{-1}(u)$. Then we want to show that the inflation $\Inf_G^{\widetilde{G}}(\Gamma_u^G)$ is a summand of $\Gamma_u^{\widetilde{G}}$.

Recall the notation of \cref{def:GGGR} then we may identify the subgroup $U(\lambda,1) \leqslant G$ with the corresponding subgroup of $\widetilde{G}$. Denoting by $\widetilde{U}(\lambda,1) \leqslant \widetilde{G}$ the subgroup $U(\lambda,1)\Ker(\pi)^F$ (a direct product) we have the GGGR corresponding to $u$ in $\widetilde{G}$ is given by
\begin{equation*}
\Gamma_u^{\widetilde{G}} = \Ind_{U(\lambda,1)}^{\widetilde{G}}(\widetilde{\varphi}_u) = (\Ind_{\widetilde{U}(\lambda,1)}^{\widetilde{G}}\circ\Ind_{U(\lambda,1)}^{\widetilde{U}(\lambda,1)})(\widetilde{\varphi}_u).
\end{equation*}
There is a unique irreducible constituent of $\Ind_{U(\lambda,1)}^{\widetilde{U}(\lambda,1)}(\widetilde{\varphi}_u)$ whose kernel contains $\Ker(\pi)^F$ and inducing this to $\widetilde{G}$ we obtain the inflation $\Inf_G^{\widetilde{G}}(\Gamma_u^G)$ as a summand of $\Gamma_u^{\widetilde{G}}$. In particular, this shows that we have
\begin{equation*}
\langle \Gamma_u^G,\chi \rangle_G = \langle \Gamma_u^{\widetilde{G}}, \Inf_G^{\widetilde{G}}(\chi) \rangle_{\widetilde{G}}
\end{equation*}
as $\Inf_G^{\widetilde{G}}(\Gamma_u^G)$ contains all the irreducible constituents of $\Gamma_u^{\widetilde{G}}$ with $\Ker(\pi)^F$ in their kernel. From this the result follows immediately.
\end{proof}

\begin{proof}[of \cref{thm:existence-wave-front-sets}]
Let us assume that $\bG$ is $\GL_n(\mathbb{K})$ or that $Z(\bG)$ is connected and the derived subgroup $\bG_{\der}$ is simple not of type $\A$. In this situation both \cref{thm:main-theorem-unipotent} and the results of \cite{lusztig:1990:green-functions-and-character-sheaves} are available to us (assuming that $p$ is a good prime) and the theorem can be proved in exactly the same way as \cite[Theorem 11.2]{lusztig:1992:a-unipotent-support}. We will not repeat the argument here.

Assume now that $\bG$ is simple and simply connected then we may choose a regular embedding $\bG \hookrightarrow \widetilde{\bG}$ such that $\widetilde{\bG}$ is $\GL_n(\mathbb{K})$ if $\bG$ is of type $\A_{n-1}$. By the previous case and \cref{lem:regular-embedding} we see the theorem holds for $G$.

Now assume $\bG$ is semisimple and simply connected then we may write $\bG$ as a direct product $\bG^{(1)}\times\cdots \times \bG^{(r)}$ where each $\bG^{(i)}$ is a direct product of simple groups permuted transitively by $F$. Clearly if the result holds for each $\bG^{(i)}$ then it holds for $\bG$ so we may assume that $\bG = \bG^{(1)} = \bG_1\times \cdots \times \bG_r$, where each $\bG_j$ is a simple group. However, in this situation we have $\bG^F \cong \bG_1^{F^r}$ so the result follows from the previous case. Thus the theorem holds for semisimple simply connected groups.

Assume now that $\bG$ has a connected centre and simply connected derived subgroup $\bG_{\der}$. Applying \cref{lem:regular-embedding} and the previous case to the natural regular embedding $\bG_{\der} \hookrightarrow \bG$ we get that the theorem holds for $G$.

Finally, assume $\bG$ is any group with a connected centre then we may find a surjective morphism $\widetilde{\bG} \to \bG$ as in \cref{lem:connected-simply-connected}, c.f., \cite[\S 8.8]{lusztig:1984:characters-of-reductive-groups}. In particular, the theorem holds for $G$ by \cref{lem:connected-simply-connected} and the previous case. Finally, assume $\bG$ is arbitrary then choosing a regular embedding $\bG \hookrightarrow \widetilde{\bG}$ we deduce the theorem from \cref{lem:regular-embedding} and the previous case.
\end{proof}

\begin{pa}
Recall that if $p$ is a good prime for $\bG$ then Geck, using Lusztig's result \cite[Theorem 11.2]{lusztig:1992:a-unipotent-support}, has shown that every irreducible character $\rho \in \Irr(G)$ has a unique unipotent support $\mathcal{O}_{\rho}$; see \cite[Theorem 1.4]{geck:1996:on-the-average-values}. For any irreducible character $\rho \in \Irr(G)$ we will denote by $\rho^* \in \Irr(G)$ the dual character $\pm D_G(\rho)$, where $D_G(\rho)$ is the Alvis--Curtis dual of $\rho$. We then have the following relationship between unipotent supports and wave front sets which appears as \cite[Theorem 11.2]{lusztig:1992:a-unipotent-support} in large characteristic.
\end{pa}

\begin{lem}\label{lem:duality-unip-supp-wave-front}
Assume $p$ is a good prime for $\bG$ then for any irreducible character $\rho \in \Irr(G)$ we have $\mathcal{O}_{\rho^*} = \mathcal{O}_{\rho}^*$.
\end{lem}

\begin{proof}
We first assume that $Z(\bG)$ is connected. Assume $\rho$ is contained in the series $\mathcal{E}(G,s,\mathfrak{C})$ then by the proof of \cite[(8.5.12)]{lusztig:1984:characters-of-reductive-groups} we have the dual character $\rho^*$ is contained in the series $\mathcal{E}(G,s,\mathfrak{C}^{\dag})$. In particular, we have $\mathcal{O}_{\rho^*} = \mathcal{O}_{s,\mathfrak{C}^{\dag}}$ by \cite[\S5.4]{geck:1996:on-the-average-values} so the result follows in this case from \cref{thm:existence-wave-front-sets}.

Now assume $Z(\bG)$ is disconnected and let $\bG \hookrightarrow \widetilde{\bG}$ be a regular embedding. If $\rho \in \Irr(G)$ is an irreducible character then choose an irreducible character $\widetilde{\rho} \in \Irr(\widetilde{G})$ such that $\rho$ is a constituent of $\Res_G^{\widetilde{G}}(\widetilde{\rho})$; we then have $\rho^*$ is a constituent of $\Res_G^{\widetilde{G}}(\widetilde{\rho}^*)$ by \cite[Corollary 5.3]{taylor:2013:on-unipotent-supports}. Now from the proof of \cite[Lemma 5.1]{geck:1996:on-the-average-values} and \cref{lem:regular-embedding} we see that $\mathcal{O}_{\rho^*} = \mathcal{O}_{\widetilde{\rho}^*}$ and $\mathcal{O}_{\rho}^* = \mathcal{O}_{\widetilde{\rho}}^*$. Hence the result follows from the previous case.
\end{proof}

\section{Closing Remarks}
\begin{pa}
In this final section we gather two important results from the literature concerning GGGRs. These results were proved assuming that the results of \cite{lusztig:1992:a-unipotent-support} hold. It is our purpose to show that these results now hold assuming only that $p$ is a good prime. We have chosen these results as they are relevant for \cite[Conjecture 2.1]{geck:2012:remarks-on-modular-representations}, which is a geometric refinement of \cite[Conjecture 3.4]{geck-hiss:1997:modular-representations} concerning the unitriangularity of the decomposition matrix; see \cite[Remark 2.4]{geck:2012:remarks-on-modular-representations}. The first result we consider, as mentioned in \cref{pa:geometric-condition}, is a geometric refinement of the condition in \cref{WF:2}.
\end{pa}

\begin{prop}[Achar--Aubert]\label{prop:geometric-refinement}
Assume $p$ is a good prime for $G$. Then for any irreducible character $\rho \in \Irr(G)$ and any unipotent element $u \in \mathcal{U}^F$ we have
\begin{equation*}
\langle \Gamma_u,\rho\rangle \neq 0 \Rightarrow \mathcal{O}_u \subseteq \overline{\mathcal{O}_{\rho}^*},
\end{equation*}
where $\mathcal{O}_u$ is the $\bG$-conjugacy class containing $u$.
\end{prop}

\begin{proof}
We start by assuming that either $\bG$ is $\GL_n(\mathbb{K})$ or that $Z(\bG)$ is connected and the derived subgroup $\bG_{\der}$ is simple but not of type $\A$. Applying \cite[Th\'eor\`eme 9.1]{achar-aubert:2007:supports-unipotents-de-faisceaux}, which is available to us because of \cref{lem:p-acceptable,cor:all-Lusztig-results}, we see that the statement holds. Assume now that $\varphi : \bG \to \bH$ is an isotypic morphism (c.f.\ \cref{pa:isotypic-morphism}) then $\varphi$ restricts to a homeomorphism between the varieties of unipotent elements in $\bG$ and $\bH$ because the restriction is a $\bG$-equivariant bijective morphism, c.f., \cite[2.5.6(b)]{geck:2003:intro-to-algebraic-geometry}. In particular, it preserves the partial order on the unipotent conjugacy classes given by the closure relation. With this we now simply follow the reduction steps given in the proof of \cref{thm:existence-wave-front-sets}.
\end{proof}

\begin{pa}
The next result we consider is an observation which is due to Geck and Malle, see the proof of \cite[Proposition 3.5]{geck-malle:2000:existence-of-a-unipotent-support}. From this point forward we assume that $Z(\bG)$ is connected. Assume $\rho \in \mathcal{E}(G,s,\mathfrak{C})$ is an irreducible character then we denote by $a_{\rho} \in \mathbb{Z}_{\geqslant 0}$ the $a$-value of the unique special character in the family $\Irr(W_{\bG^{\star}}(s)\mid \mathfrak{C})$, c.f., \cite[4.1.1]{lusztig:1984:characters-of-reductive-groups}. By \cite[4.26.3]{lusztig:1984:characters-of-reductive-groups} there exists a unique positive integer  $n_{\rho} \in \mathbb{Z}_{>0}$ such that $n_{\rho}\cdot\rho(1) \in \mathbb{Z}[q]$ is a polynomial in $q$ with integer coefficients and
\begin{equation*}
\pm n_{\rho}\cdot\rho(1) = q^{a_{\rho}} + \text{ higher powers of }q.
\end{equation*}
With this we have the following.
\end{pa}

\begin{prop}[Lusztig, Geck--Malle]
Assume $p$ is a good prime for $\bG$ and $Z(\bG)$ is connected. Then for any irreducible character $\rho \in \Irr(G)$ we have
\begin{equation*}
\langle \Gamma_{\iota_0},\rho\rangle = \frac{|A_{\bG}(u)|}{n_{\rho}},
\end{equation*}
where $\iota_0 = (\mathcal{O}_{\rho}^*,\Ql) \in \mathcal{V}_{\bG}^{\uni}$ and $u \in \mathcal{O}_{\rho}^*$ is a class representative.
\end{prop}

\begin{proof}
Before proving the result we will need the following reduction argument. Assume $\pi : \bG \to \widetilde{\bG}$ is a regular embedding. If $\widetilde{\rho} \in \Irr(\widetilde{G})$ is an irreducible character then the restriction $\rho = \Res^{\widetilde{G}}_G(\widetilde{\rho})$ is also irreducible because $Z(\bG)$ is connected, see \cite[\S11]{lusztig:1988:reductive-groups-with-a-disconnected-centre}. As in the proof of \cref{lem:regular-embedding} we have for any unipotent element $u \in G$ that
\begin{equation*}
\langle \Gamma_u^{\widetilde{G}},\widetilde{\rho}\rangle_{\widetilde{G}} = \langle \Gamma_u^G,\rho\rangle_G.
\end{equation*}
As $Z(\bG)$ and $Z(\widetilde{\bG})$ are both connected we have $\pi$ induces an isomorphism $A_{\bG}(u) \to A_{\widetilde{\bG}}(u)$ and a bijection between the unipotent conjugacy classes of $G$ and $\widetilde{G}$. Hence, we easily see that the result holds for $\bG$ if and only if it holds for $\widetilde{\bG}$.

Now, let $\pi : \widetilde{\bG} \to \bG$ be a surjective isotypic morphism as in \cref{lem:connected-simply-connected} (see \cite[\S 8.8]{lusztig:1984:characters-of-reductive-groups}). The map $\pi$ induces an isomorphism $A_{\widetilde{\bG}}(u) \to A_{\bG}(u)$ for any unipotent element $u \in \widetilde{\bG}$ because the kernel of $\pi$ is connected. Let us denote by $\widetilde{\rho} = \Inf_G^{\widetilde{G}}(\rho)$ the inflation of $\rho$. From the definition it is clear that $n_{\widetilde{\rho}} = n_{\rho}$ and by the proof of \cref{lem:connected-simply-connected} we see that $\langle \Gamma_{\iota_0}^G,\rho\rangle = \langle \Gamma_{\iota_0}^{\widetilde{G}},\widetilde{\rho}\rangle$. Therefore, we may and will assume that the derived subgroup $\bG_{\der}$ is simply connected.

Now according to \cite[\S8.8]{lusztig:1984:characters-of-reductive-groups} we can find two connected reductive algebraic groups $\widetilde{\bG}$ and $\bH$ and a pair of regular embeddings
\begin{equation*}
\bG \hookrightarrow \widetilde{\bG} \hookleftarrow \bH
\end{equation*}
such that the following holds: $\bH$ is a direct product $\bH_1 \times \cdots \times \bH_r$ such that $Z(\bH_i)$ is connected and the derived subgroup of $\bH_i$ is simple and simply connected for all $1 \leqslant i \leqslant r$. Applying twice the first reduction argument we see that the result holds in $\bG$ if and only if it holds in $\bH$. In particular, we may assume that $\bG$ has the same form as $\bH$.

Using the same arguments as in the proof of \cref{thm:existence-wave-front-sets} it is clear that we need only prove the statement assuming that $Z(\bG)$ is connected and the derived subgroup of $\bG$ is simple and simply connected. We start with the case where either $\bG$ is $\GL_n(\mathbb{K})$ or the derived subgroup is not of type $\A$. Let us denote by $\rho^* = \pm D_G(\rho) \in \Irr(G)$ the Alvis--Curtis dual of $\rho$. From the proof of \cite[Proposition 3.5]{geck-malle:2000:existence-of-a-unipotent-support} we see that
\begin{equation*}
\langle \Gamma_{\iota_0},\rho\rangle = \langle \Gamma_{(\mathcal{O}_{\rho^*},\Ql)}, \rho \rangle = \frac{|A_{\bG}(u)|}{n_{\rho}},
\end{equation*}
which is applicable by \cref{lem:p-acceptable,cor:all-Lusztig-results}. Note that we have $\mathcal{O}_{\rho}^* = \mathcal{O}_{\rho^*}$ by \cite[Theorem 11.2]{lusztig:1992:a-unipotent-support}. Thus the result holds in this case.

It remains to deal with the case where $Z(\bG)$ is connected and the derived subgroup $\bG_{\der}$ is isomorphic to $\SL_n(\mathbb{K})$. For such a group we have $|A_{\bG}(u)| = n_{\rho} = 1$ for all unipotent elements $u \in \bG$ and irreducible characters $\rho \in \Irr(G)$. Hence, the statement reduces to the statement that
\begin{equation}\label{eq:mult-1}
\langle \Gamma_u^G,\rho\rangle_G = 1
\end{equation}
for all $u \in {\mathcal{O}_{\rho}^*}^F$. Now computing the multiplicity on the left we see by Frobenius reciprocity that
\begin{equation}\label{eq:frob-reciprocity}
\langle \Gamma_u^G,\rho\rangle_G = \langle \Gamma_u^{G_{\der}},\Res^G_{G_{\der}}(\rho)\rangle_{G_{\der}}.
\end{equation}
By \cite[\S3]{lusztig:1988:reductive-groups-with-a-disconnected-centre} the restriction $\Res^G_{G_{\der}}(\rho) = \sum_{\sigma \in \mathcal{A}} \sigma$ is a sum of irreducible characters such that $\mathcal{A}$ is an orbit under the action of the $F$-coinvariants $Z(\bG_{\der})_F$ of the centre of $\bG_{\der}$. Thus, it suffices to show that
\begin{equation*}
\langle \Gamma_u^{G_{\der}},\sum_{\sigma \in \mathcal{A}} \sigma\rangle_{G_{\der}} = 1
\end{equation*}
for any $Z(\bG_{\der})_F$-orbit $\mathcal{A} \subseteq \Irr(G_{\der})$ and $u \in {\mathcal{O}_{\sigma}^*}^F$ with $\sigma \in \mathcal{A}$ some (any) representative of the orbit. This problem doesn't depend upon $G$ so we may use any group to solve it. In particular, taking $\bG = \GL_n(\mathbb{K})$ we see that this is true by \cref{eq:mult-1,eq:frob-reciprocity}. This completes the proof.
\end{proof}

\begin{cor}
Assume $Z(\bG)$ is connected $\bG/Z(\bG)$ is an almost direct product of simple groups of type $\A$. Then for any irreducible character $\rho \in \Irr(G)$ and any unipotent element $u \in G$ we have
\begin{equation*}
\langle \Gamma_u,\rho \rangle = \begin{cases}
1 &\text{if } u \in \mathcal{O}_{\rho}^*,\\
0 &\text{if } \mathcal{O}_u \not\subset \overline{\mathcal{O}_{\rho}^*}.
\end{cases}
\end{equation*}
\end{cor}

\begin{proof}
This follows immediately from the fact that, for such a group, $C_{\bG}(u)$ is connected for any unipotent element $u \in \bG$ and that $n_{\rho} = 1$ for any irreducible character $\rho \in \Irr(G)$.
\end{proof}

\renewcommand*{\bibfont}{\small}
\begin{spacing}{0.96}
\printbibliography
\end{spacing}
\end{document}